\documentclass[11pt,reqno,tbtags]{amsart}
\usepackage{amssymb}
\usepackage{natbib}
\bibpunct[, ]{[}{]}{;}{n}{,}{,}

\usepackage{fancyhdr}
\numberwithin{equation}{section}
\allowdisplaybreaks

\newtheorem{theorem}{Theorem}[section]
\newtheorem{lemma}[theorem]{Lemma}

\theoremstyle{definition}

\theoremstyle{remark}

\newenvironment{romenumerate}{\begin{enumerate}
 }{\end{enumerate}}

\newcounter{oldenumi}
{\setcounter{oldenumi}{\value{enumi}}
\begin{romenumerate} \setcounter{enumi}{\value{oldenumi}}}
{\end{romenumerate}}

\newcounter{thmenumerate}
\newenvironment{thmenumerate}
{\setcounter{thmenumerate}{0}%
 \def\item{\par
 \refstepcounter{thmenumerate}\textup{(\roman{thmenumerate})\enspace}}
}
{}

\newcounter{xenumerate}   

\newenvironment{proofof}[1]{\noindent {\bf
Proof of #1}.}{\hfill $\square$\par\smallskip\par}

\begingroup
  \divide\count255 by 60
  \multiply\count255 by -60
  \advance\count255 by \time
  \ifnum \count255 < 10 \xdef\klockan{\the\count1.0\the\count255}
  \else\xdef\klockan{\the\count1.\the\count255}\fi
\endgroup

\def\rompar(#1){\textup(#1\textup)}    

\def\xexp(#1){e^{#1}}

\newcommand\R{\mathbb R}

\newcommand\N{\mathbb N}  

\newcommand\Z{\mathbb Z}
\newcommand\I{\mathbb I}

\newcounter{CC}

\newcommand\ran{\operatorname{\mathrm ran}}
\newcommand\dev{\operatorname{\mathrm dev}}
\newcommand\var{\operatorname{\mathrm var}}

\newcommand\E{\operatorname{\mathbb E{}}}
\renewcommand\P{\operatorname{\mathbb P{}}}

\newcommand\cF{\mathcal F}
\newcommand\cG{\mathcal G}

\newcommand\cP{\mathcal P}

\def\[#1]{[\![#1]\!]}

\begin{document}
\title
{A quantitative differential equation approximation for a routing model}

\date{14 June 2013} 


\author{Malwina J. Luczak} \thanks{Research supported by an EPSRC Leadership Fellowship EP/J004022/2.}
\address{School of Mathematical Sciences, Queen Mary, University of London}
\email{m.luczak@qmul.ac.uk}

\keywords{Markov chains, concentration of measure, coupling, law of large numbers, load balancing}
\subjclass[2000]{60J75, 60C05, 60F15}

\begin{abstract}
We present some new concentration of measure inequalities for discrete time Markov chains, and
illustrate their application by analysing a well-studied routing model in greater depth than had previously been possible.
In the model, calls arrive for each pair of endpoints in a fully-connected network as a Poisson process, and calls have exponential durations.
Each call is routed either along the link connecting its endpoints, or,
if the direct route is unavailable, along a two-link path between them, via an intermediate node.
We use an explicit and simple coupling to show a strong concentration of measure property, and
deduce that the evolution of the process may be approximated by a differential equation.
The technique is likely to be useful to prove laws of large numbers in other settings.
\end{abstract}

\maketitle

\section{Introduction}\label{S:intro}

We present some new concentration of measure inequalities for discrete time Markov chains, and illustrate their
application by analysing a well-studied routing model in greater depth than had previously been possible.  The concentration of measure
inequalities will be presented in Section~\ref{S:conc}; we now introduce the routing model and describe our results.

We consider a class of routing problems in continuous time, where calls
have Poisson arrivals and exponential durations, studied earlier
in~\cite{aku,ch,ghk,gm,lm08,lu}.
The setting is as follows. For each $n \in \N$, we have a fully connected {\em communication graph} $K_n$, with node set
$V_n=\{1, \ldots, n\}$ and link set $L_n = \{\{u,v\}: 1 \le u < v \le n\}$. Each link $\{u,v\} \in L_n$ has
capacity $C=C(n)< \infty$, where $C(n) \in \Z^+$ for each $n$.
Each arriving call is to be routed either along a link $\{u,v\}$ or along a path between $u$ and $v$ consisting of a pair of links $\{u,w\}$ and $\{v,w\}$,
for a pair $u,v$ of distinct nodes (endpoints of the call) and some intermediate node $w \not = u,v$, if
possible. A call in progress will use one unit of capacity of
each of the links it occupies, for its entire duration.
Calls arrive as a Poisson process with rate $\lambda {n \choose 2}$, where $\lambda= \lambda (n) > 0$. 
The endpoints of each call are uniform over the links of the complete graph $K_n$. If a call is
for nodes $u$ and $v$, then we route it over the direct link $\{u,v\}$ between $u$ and $v$ if possible, that is if $\{u,v\}$ has fewer than
$C$ calls currently using it. Otherwise, we pick an ordered list of $d=d(n)$
possible intermediate nodes $(w_1, \ldots, w_d)$ from $V_n \setminus \{u,v\}$,
uniformly at random with replacement, and the call is routed along
one of the two-link routes $\{\{u,w_1\}, \{v,w_1\}\}, \ldots, \{\{u,w_d\}, \{v,w_d\}\}$,
chosen to minimise the larger of the current loads on its two links,
subject to the capacity constraints.
Ties are broken in favour of the first `best' route in the ordered list.
If none of the $d$ two-link paths is available, then the call is lost. Call durations are unit mean
exponential random variables, independent of one another and of the arrivals and choices processes.

Here, we focus on the analysis of this algorithm as $n$ tends to infinity.
We prove that, asymptotically, for suitable initial conditions and suitable functions $\lambda (n)$, $d(n)$ and $C(n)$, for each node $v$, the
proportion of links at $v$ that carry $k$ calls is well approximated by the solution of a differential equation.
(Note that, when $\lambda$, $d$ and $C$ 
vary with $n$, there is no single limiting differential equation, but rather a sequence of approximating differential equations,
with dimension tending to infinity if $C \to \infty$ with $n$.)

Such law of large numbers results have been difficult to prove in this and related models, due to apparent and potentially strong dependencies between system
elements (in this context, links). Here we are able to prove that these dependencies are negligible; it turns out that, in a suitable sense, links in certain
collections evolve approximately independently of one another. Our technique appears to be new, and is likely to be useful in other settings.
It relies on a coupling, which is used to prove that slowly-changing functions of the process (for instance, the number of links around a node $v$ with load exactly $k$,
for each node $v$ and each $k \in \{0,1, \ldots, C\}$) are well concentrated at each time $t$. Thanks to the strong concentration of measure, it is then
possible to show that the expected drifts of functions of interest factorise approximately, leading to a differential equation approximating these functions.
The basic principle of our approach is, in essence, rather simple; however, there are considerable difficulties arising from the complicated evolution of
the process in question.

In~\cite{bl}, Brightwell and the author carry out an improved analysis of the coupling introduced in this paper to analyse the process in equilibrium, if
the arrival rate $\lambda$ is either sufficiently small or sufficiently large.

For each link $e=\{u,v\} \in L_n$, let $X_t^{(n)} (e,0)$
denote the number of calls in progress at time $t$ which are
routed along the link $e$, that is the number of directly routed calls between
the end nodes $u$ and $v$ of $e$ that are in progress at time $t$.
For each link $e=\{u,v\} \in L_n$ and node $w \in V_n \setminus e$,
let $X_t^{(n)} (e,w)$ denote the number of calls in progress at time $t$ which are
routed along the path consisting of links $\{u,w\}, \{v,w\}$, that is the number of calls between
the end nodes $u$ and $v$ of $e$ routed via $w$ that are in progress at time $t$.
We call $X_t^{(n)} = (X_t^{(n)}(e,0), X_t^{(n)} (e,w): e \in L_n, w \in V_n \setminus e)$ the {\em load vector} at time $t$,
and let $S = \{0,1, \ldots, C\}^{n(n-1)^2/2}$ denote the state space, containing the set of all possible load vectors.
Then $X^{(n)}=(X^{(n)}_t)_{t \geq 0}$ is a continuous-time discrete-space Markov chain.
We will normally drop the superscript $n$, 
to avoid unnecessarily cluttering the notation.


Given a load vector $x \in S$ and a link $e=\{u,v\} \in L_n$, let $x(e)$ denote the load of link $e$. Then
$$x(\{u,v\}) = x(\{u,v\},0) + \sum_{w \not \in \{u,v\}} (x(\{u,w\},v)+ x(\{v,w\},u).$$
Given a load vector $x$, node $v$ and $k \in \Z^+$, let $f_{v,k} (x)$ be the number of links $\{v,w\}$, $w \not = v$,  in $x$ such that $x(\{v,w\})=k$
(that is, the number of links with one end $v$ carrying exactly $k$ calls).

For a vector $\xi = (\xi(k): k=0, \ldots, C)$, let
$\xi(\le j) = \sum_{k=0}^j \xi(k)$.
Define $F: \R^{C+1} \to \R^{C+1}$ by
\begin{eqnarray}
F_0 (\xi) & = & - \lambda \xi(0)- \lambda g_0(\xi)+ \xi(1), \nonumber \\
F_k (\xi) & = & \lambda \xi(k-1) - \lambda \xi(k)
 +  \lambda g_{k-1} (\xi) - \lambda g_k(\xi)
-  k \xi(k) \nonumber \\
&&\mbox{} + (k+1) \xi(k+1), \qquad 0 < k < C, \label{eq.F} \\
F_C (\xi) & = & \lambda \xi (C-1)
 +  \lambda g_{C-1} (\xi) -
  C \xi(C), \nonumber 
\end{eqnarray}
where functions $g_j$, $j=0, \ldots, C-1$, are given by
\begin{eqnarray}
g_j (\xi) & = & 2 \xi(C) \xi(j) \xi(\le j) \sum_{r=1}^d (1-\xi(\le j)^2)^{r-1}
(1-\xi(\le j-1)^2)^{d-r} \label{eq.g} \\
&& \mbox{} + 2 \xi(C)  \xi(j) \sum_{i=j+1}^{C-1} \xi(i) 
\sum_{r=1}^d(1-\xi(\le i)^2 )^{r-1}
(1-\xi(\le i-1)^2)^{d-r}\nonumber
\end{eqnarray}
In~(\ref{eq.F}), the linear terms account for departures and directly routed arrivals.
Each function $g_j$ is proportional
to the rate of arrivals of alternatively routed calls onto links that carry $j$ calls.
When $d=1$, we have the simpler expression $g_j (\xi) =  2 \xi(C)(1-\xi (C))\xi(j)$ for
$0 \le j \le C-1$.

Let $\Delta^{C+1}_\le$ denote the set of non-negative
$\xi$ such that $\sum_{k=0}^{C} \xi(k) \le 1$, and let $\Delta^{C+1}_=$ be the set of non-negative
$\xi$ such that $\sum_{k=0}^{C} \xi(k) = 1$.
We will show (see Lemma~\ref{lem.lipschitz}) that, for each $d \ge 1$,
$F$ is Lipschitz, with constant $8d^2(\lambda+1)(C+1)^2$, over
$\Delta^{C+1}_\le$, with respect to the $\ell_{\infty}$ norm.
Hence we will prove that, for all $\xi_0 \in \Delta^{C+1}_=$,
\begin{equation}
\frac{d\xi_t}{dt} = F(\xi_t)
\label{eq.diff-eq}
\end{equation}
has a unique solution starting from $\xi_0$, valid for all times and such that
$\xi_t \in \Delta^{C+1}_=$ for all $t \ge 0$.

Given a pair of nodes $u,v$ and 
$j \in \{0, \ldots, C\}$, let $\I_{uv}^j: S \to \{0,1\}$ be defined by $\I_{uv}^j (x) =1$ if $x (\{u,v\})=j$ and $\I_{uv}^j (x) = 0$ otherwise.
Note that $\I_{uv}^j = \I_{vu}^j$, and
that $\I_{vv}^j$ is identically $0$ for each $v$ and $j$.
Let functions $\phi^1,\phi^2, \phi^3: S \to {\mathbb R}$ be defined by
\begin{eqnarray} 
\phi^1(x) &=& \max_{u,v: u \not = v} \max_{j,k}\Big |\frac{1}{n-2}\sum_{w} \I_{uw}^j(x) \I_{vw}^k(x) \nonumber \\
&& \mbox{} - \frac{1}{(n-2)^2}\sum_{w \not = u,v} \I_{uw}^j(x) \sum_{w' \not = u,v}\I_{vw'}^k(x)\Big |;\label{eq:phi1} \\
\phi^2(x) &=& \max_{u,v: u \not = v} \max_j \frac{1}{n-2} |f_{u,j}(x) - f_{v,j}(x)| \nonumber \\
&=& \max_{u,v: u \not = v} \max_j \frac{1}{n-2}
| \sum_{w \not = u} \I_{uw}^j(x) - \sum_{w \not = v} \I_{vw}^j(x)| \nonumber \\
&=& \max_{u,v: u \not = v} \max_j \frac{1}{n-2}
| \sum_{w \not = u,v} \I_{uw}^j(x) - \sum_{w \not = u,v} \I_{vw}^j(x)|; \label{eq:phi3} \\
\phi^3(x) &=& \max_{u,v: u \not = v} \frac{1}{n-2} \sum_{w\not = u,v} x(\{u,v\},w). \label{eq:phi2}
\end{eqnarray}
Let $\phi = \max \{\phi^1, \phi^2,\phi^3\}$. This function $\phi$ is related to the function called $\phi$ by Crametz and Hunt in~\cite{ch}.

We shall prove that, if $\phi (X_0)$ is small, for instance $\phi (X_0) = O \Big ( \frac{\log n}{\sqrt{n}} \Big )$, then $\phi (X_t)$ remains small for a time interval of order 1, and over that period each function $(n-1)^{-1}f_{v,j}(X_t)$ is well-approximated by the solution to the differential equation~(\ref{eq.diff-eq}) with initial condition $\xi_0 (j) = (n-1)^{-1} f_{v,j} (X_0)$ for $j=0, \ldots, C$. Note that, if $\phi^2 (X_0)$ is small, then, for each $j$, all the functions $f_{v,j}(X_0)$ for different $v \in V_n$ are nearly equal, so all the $(n-1)^{-1}f_{v,j}(X_t)$ can be approximated by the same solution to the differential equation~(\ref{eq.diff-eq}).

Let $S_1$ be the set of all states $x$ such that $\|x\|_1\le 2 \lambda {n \choose 2}$.

\begin{theorem} \label{thm.main-result}
Suppose that $\lambda$ and $t_0$ are positive reals, and $d$ and $C$ positive integers.
Set $\gamma = 1/(2^{25}(d^8+d^4C/\lambda)(8\lambda t_0+1)^3e^{800d\lambda t_0})$,
and suppose that
$$
n \ge n_0(\lambda,d,C,t_0) = \max \Big( 2^{18}(\lambda+1/\lambda)^4 d^4 (C+1)^6 (t_0+1/t_0)^2,
e^{8/\gamma} \Big).
$$
Let $\xi_0$ be in $\Delta^{C+1}_=$, and
let $(\xi_t)$ be the unique solution to the differential equation~(\ref{eq.diff-eq})
on $[0,t_0]$, subject to initial condition $\xi_0$.  Let $X_0$ be in $S_1$.
Let $B_n$ be the event that, for each $v \in V_n$, $k \in \{ 0, \dots, C\}$ and $t \in [0,t_0]$,
\begin{eqnarray*}
\lefteqn{|f_{v,k} (X_t) -(n-1) \xi_t (k)| \le \Big( \sup_{u,j} \Big| f_{u,j}(X_0)-(n-1)\xi_0(j) \Big|} \\
&&\mbox{} + 64(\lambda +1)(t_0+1) d^2 (C+1)^3 \Big(n \phi(X_0) + 3 n^{1/2} \log n\Big) \Big)
e^{216(\lambda+1)d^2 (C+1)^3 t_0}.
\end{eqnarray*}
Then
$\P (\overline{B_n} ) \le e^{-\frac12 \gamma\log^2 n}$.

In particular, suppose that there is $v_0 \in V_n$ such that $\xi_0(j) = \frac{1}{n-1} f_{v_0,j}(X_0)$ for $j =0, \ldots, C$.
Let $B'_n$ be the event that, for each $t \le t_0$, $k \in \{0, \ldots, C\}$, and $v \in V_n$,
\begin{eqnarray*}
\lefteqn{|f_{v,k} (X_t) -(n-1) \xi_t (k)| \le 64(\lambda +1)(t_0+1) d^2 (C+1)^3} \\
&&\mbox{} \times \Big(2n \phi(X_0) + 3 n^{1/2} \log n\Big) e^{216(\lambda+1)d^2 (C+1)^3 t_0}.
\end{eqnarray*}
Then
$\P (\overline{B'_n} ) \le e^{-\frac12 \gamma\log^2 n}$.
\end{theorem}

For the special case $d=1$, we obtain sharper bounds, replacing the term $(C+1)^3$ in the exponent by
$(C+1)$.

\begin{theorem} \label{thm.main-result-d=1}
Suppose that $\lambda$ and $t_0$ are positive reals, and $C$ a positive integer,
and suppose that $d=1$.  Set $\gamma = 1/(2^{25}(1+C/\lambda)(8\lambda t_0+1)^3e^{800\lambda t_0})$, and suppose that
$$
n \ge n_0(\lambda,1,C,t_0) = \max \Big(2^{18}(\lambda+1/\lambda)^4 (C+1)^6 (t_0+1/t_0)^2,
e^{8/\gamma} \Big).
$$
Let $\xi_0$ be in $\Delta^{C+1}_=$, and
let $(\xi_t)$ be the unique solution to the differential equation~(\ref{eq.diff-eq})
on $[0,t_0]$, subject to initial condition $\xi_0$.  Let $X_0$ be in $S_1$.
Let $B_n$ be the event that, for each $v \in V_n$, $k \in \{ 0, \dots, C\}$ and $t \in [0,t_0]$,
\begin{eqnarray*}
\lefteqn{|f_{v,k} (X_t) -(n-1) \xi_t (k)| \le \sup_{u,j} \Big| f_{u,j}(X_0)-(n-1)\xi_0(j) \Big|} \\
&&\mbox{} + 64(\lambda +1)(t_0+1) (C+1)^3 \Big(n \phi(X_0) + 3 n^{1/2} \log n\Big) e^{216(\lambda+1) (C+1) t_0}.
\end{eqnarray*}
Then
$\P (\overline{B_n} ) \le e^{-\frac12 \gamma\log^2 n}$.
\end{theorem}

Suppose that, for each $n$, $X^{(n)}_0=x^{(n)}_0$ a.s., for some deterministic load vector $x^{(n)}_0 \in S_1$ such that, for some constant $c$,
$\phi (x^{(n)}_0) \le \frac{c \log n}{\sqrt{n}}$ and $\max_{v,j} |(n-1)^{-1} f_{v,j} (x^{(n)}_0) - \xi_0(j)| \le \frac{c \log n}{\sqrt{n}}$.
Suppose also that, as $n \to \infty$, $\lambda$ and $t_0$ are bounded away from~0, and that $\lambda d^2C^3t_0 = o(\log n)$ and
$d \lambda t_0 = o(\log \log n)$.  Then, for sufficiently large $n$, the condition on $n$
in Theorem~\ref{thm.main-result} is satisfied, and the theorem implies that, for $\epsilon > 0$, if $A^\epsilon_n$ is the event that, for each $v \in V_n$,
each $k \in \{ 0, \dots, C\}$, and each $t \in [0,t_0]$,
$$
|f_{v,k} (X_t) -(n-1) \xi_t (k)| \le n^{1/2 + \epsilon},
$$
then $\P(\overline{A^\epsilon_n}) \to 0$ as $n \to \infty$.
For $d=1$, the corresponding conditions are that, as $n \to \infty$, $\lambda$ and $t_0$ are bounded away from~0, and that $\lambda C t_0 = o(\log n)$
and $\lambda t_0 = o(\log \log n)$.

Results analogous to Theorems~\ref{thm.main-result} and~\ref{thm.main-result-d=1}, with different constants, hold for $X_0$ not necessarily in $S_1$, as long as $\|X_0\|_1 \le c {n \choose 2}$ for some constant $c$.





In the simplest case where $d=1$, if the direct link is at full capacity, only one two-link
alternative route is considered, and it is used if there is spare capacity on both links.
This case, with constant arrival rate $\lambda$ and constant capacity $C$, was first studied by Gibbens, Hunt and Kelly~\cite{ghk} and then by Crametz and Hunt~\cite{ch}, and
Graham and M\'el\'eard~\cite{gm}.
For $k=0,1, \ldots, C$, let $Y^{(n)}_t (k)$ denote the proportion of links that carry $k$ calls at time $t$ in a system with $n$ nodes.  It is
conjectured in~\cite{ghk} and shown in~\cite{ch} that, under suitable conditions,
$(Y^{(n)}_t(k): k = 0, \ldots, C)$ converges in distribution as
$n \to \infty$ to a deterministic vector $(\xi_t(k): k = 0, \ldots, C)$ obtained as the solution to the differential equation derived from the average drift
of $(Y^{(n)}_t(k): k = 0, \ldots, C)$, with appropriate initial conditions. The convergence result in~\cite{ch} is non-quantitative. Graham and M\'el\'eard~\cite{gm} do give a quantitative
result concerning independence of small collections of links under the assumptions that initial link loads are iid and that initially there are no
alternatively routed calls in the network. This result can be used to deduce a quantitative law of large numbers. (Also, it would be possible to
quantify convergence in the more general case they consider.)

In the case of $\lambda$ and $C$ constant, and $d=1$, Theorem~\ref{thm.main-result-d=1} is a more refined, quantitative, version of the law of large numbers in~\cite{ch}. Also, our result in this case is related to those in~\cite{gm}. Unlike~\cite{gm}, we do not need to
assume that initially all the nodes are exactly exchangeable. Instead, our law of large numbers result holds for a large class of deterministic initial states, and holds simultaneously for all nodes.
Theorem~\ref{thm.main-result} and the remaining cases in Theorem~\ref{thm.main-result-d=1} are completely new.

For $d \ge 2$ constant and constant $\lambda$, this model is a variant of one that has attracted
earlier interest.  Luczak and Upfal~\cite{lu} study a version (both in discrete and in
continuous time) where the total capacity
of each link $\{u,v\}$ is divided into
three parts, one for `direct' calls, one for indirectly routed calls with one end $u$ and one for indirectly routed calls with one end $v$.
Equivalently, each `undirected' link $\{u,v\}$ has capacity $C_1(n)$ and is a first-choice path for calls between $u$ and $v$; also,
for each link $\{u,v\}$ there are two directed links, $uv$ and $vu$, each with
capacity $C_2(n)$. Link $uv$ is used for indirectly routed calls with one end $u$ and link $vu$ is used for indirectly routed calls with one end $v$.
The results of~\cite{lu} for the discrete-time model were strengthened and extended by Luczak, McDiarmid and Upfal~\cite{lmu03}, who also studied the discrete-time version of the model that is the focus of this paper. The long-term behaviour of the continuous-time model was analysed in~\cite{aku} and also in~\cite{lm08}, where calls are not routed on direct links at all.

Theorem~\ref{thm.main-result} holds also in the case above where direct links are not
used (i.e., each arrival is allocated to the best among $d$ indirect routes), with a suitably modified $F$ in~(\ref{eq.diff-eq}). Indeed, for $0<k<C$, we take instead
$$
F_k (\xi) = \lambda g_{k-1} (\xi) - \lambda g_k(\xi) -  k \xi(k) + (k+1) \xi(k+1),
$$
where the functions $g_k(\xi)$ are amended by dropping the factor $\xi(C)$; $F_0(\xi)$ and
$F_C(\xi)$ are modified in the same way. The proof is essentially identical, indeed slightly
simpler in a few places.

In~\cite{lm08} (as well as in~\cite{lmu03} for a corresponding discrete-time model), the
class of routing strategies choosing a path for a new call from among $d$ random alternatives is called the GDAR
(General Dynamic Alternative Routing) Algorithm.  The particular model we study in this paper is called the BDAR (Balanced Dynamic Alternative Routing) Algorithm.  The FDAR (First Dynamic Alternative Routing) Algorithm
always chooses the first possible alternative two-link route among the $d$ chosen.
As in other models of this type, the `power of two choices' phenomenon has been observed, that is, with the BDAR algorithm for $d \ge 2$, the capacity
$C=C(n)$ required to ensure that most calls are routed successfully is much smaller than with the FDAR algorithm.  This phenomenon is not exhibited by the model studied in this paper, as
proved in Theorem~1.1 of~\cite{lmu03}, but it does occur in the variant discussed above where direct links are not used.


In particular, see~\cite{lm08}, for the variant where there is capacity division into three parts, but with zero capacity for direct calls, the following is true after a `burn-in' period of length
$O(\log n)$.  Suppose we use the FDAR algorithm and each indirect link has capacity $C(n) \sim \alpha \frac{\log n}{\log \log n}$, where $\alpha > 0$ is a constant.
If $\alpha > 2/d$, then there exists a constant $K > 0$ such that the mean number of calls lost in an interval of length $n^K$ is $o(1)$.
If $\alpha < 2/d$ then for each $K$ there exists a constant $c >0$ such that the mean number of calls lost in an interval of length $n^K$ is at least $n^c$.
On the other hand, suppose we use the BDAR algorithm with $d \ge 2$ choices, and let $K > 0$ be a constant. There exists a constant $c > 0$ such that,
if $C(n) \ge \frac{\log \log n}{\log d} + c$ then the expected number of lost calls in an interval of length $n^K$ is $o(1)$; and if
$C(n) \le \frac{\log \log n}{\log d} -c$ then the expected number of lost calls in such an interval is at least $n^{K+2-o(1)}$ as $n \to \infty$.

Our methods apply to any GDAR algorithm, and indeed any of the variants discussed above. The law of large numbers for the BDAR algorithm proved here is valid for the model without
direct links in the parameter range considered in~\cite{lm08}, i.e., with constant $\lambda$ and $d$, and $C=C(n) = O(\log\log n)$ for $d\ge 2$, and $C=C(n) = O(\log n / \log\log n)$ for $d=1$.

Brightwell and the author~\cite{bl} use the same technique as in this paper, with a more detailed analysis of the coupling, to treat the process in equilibrium, in the cases where $\lambda$ is either sufficiently small or sufficiently large.  We prove rapid mixing to equilibrium, and show that $\frac{1}{n-1} f_{v,k}$
is well-concentrated around its expectation, which in turn is well-approximated by the unique fixed point of~(\ref{eq.diff-eq}).  For $d=1$, this proves
the approximation suggested in~\cite{ghk}, and it also provides an alternative proof of some of the results in~\cite{lm08}, for these ranges of $\lambda$.


The rest of the paper is organised as follows. In Section 2 we develop a concentration of measure inequality that will be a fundamental ingredient of our proofs. In Section 3, we formally
write down the generator of the Markov chain in question, and give an informal explanation of our proof strategy. In Section 4, we describe a simple coupling between two copies of the Markov chain, and show that, under the coupling, they do not get much further apart over time, according to suitable notions of distance. In Section~5, we use the coupling to show that nice functions of the Markov chain are well concentrated around their expectations. Section 6 contains 
estimates of the expectation of the generator of the Markov chain. In Section 7, we prove our main results, Theorems~\ref{thm.main-result} and~\ref{thm.main-result-d=1}. In Section 8, we discuss the issue of initial conditions needed for Theorem~\ref{thm.main-result}
and~\ref{thm.main-result-d=1} to apply. In Section 9 we discuss how our techniques can be extended to analyse other models.



\section{Concentration inequalities}

\label{S:conc}

We present some concentration of measure inequalities that will be used in our proofs, and may also be useful in the analysis of other Markov chains with similar properties. These inequalities
generalise results presented in~\cite{l08}.

Let $X=(X_t)_{t \in \Z^+}$ be a discrete-time Markov chain with
a discrete state space $S$ and transition probabilities $P(x,y)$ for $x,y \in S$. 
We allow $X$ to be lazy, that is we allow $P(x,x) > 0$ for some $x \in S$.
For $x \in S$, we set $N(x) = \{y: P(x,y) > 0\}$, and assume that $N(x)$ is finite for each $x \in S$.

This 
setting is natural, and many models
in applied probability and combinatorics fit into this framework,
including those discussed in Section~\ref{S:intro}.

Let
$\Omega = S^{\N} = \{\omega= (\omega_0, \omega_1, \ldots): \omega_i \in S
\quad \forall i \}$.
Members $\omega$ of $\Omega$ will correspond to possible paths of the chain $X$, in that
$X_i (\omega)=\omega_i$ for $i \in \Z^+$.
Then for each $t \in \Z^+$, $X_t$ may be viewed as a random variable on a measurable space
$(\Omega, \cF)$, where
$\cF = \sigma (\cup_{t=0}^{\infty} \cF_t)$ and $\cF_t = \sigma
(X_i: i \le t)$. The
$\sigma$-fields $\cF_t$ form the natural filtration for $X$.

Let $\cP (S)$ be the power set of the discrete set $S$.
The law of the Markov chain is a probability measure $\P$ on
$(\Omega, \cF)$, determined uniquely by the transition matrix
$P$ together with
the initial state $X_0=x_0$, according to
\begin{eqnarray*}
\P (\{ \omega: \omega_j = x_j \mbox{ for all } j \le i\}) = 
\prod_{j=0}^{i-1} P(x_j,x_{j+1}),
\end{eqnarray*}
for each $i \in \Z^+$ and $x_1, \ldots, x_i \in S$. To be
precise,
this defines the law of $(X_t)$ conditional on 
$X_0=x_0$, denoted by $\P_{x_0}$ in what follows.
Let $P^t(x,y)$ be the $t$-step transition probability from $x$ to
$y$, given inductively by
$$P^t (x,y) = \sum_{z \in S} P^{t-1} (x,z) P(z,y).$$
Let $\E_{x_0}$ denote the expectation operator corresponding to $\P_{x_0}$; then
$\E_{x_0} [f(X_t)]$ 
is
the expectation of the function $f$ with respect to measure $\delta_{x_0} P^t$.

For $t \in \Z^+$ and $f: S \to \R$, define the function $P^t f$ by
$$(P^t f)(x) = \E_x [f(X_t)] = \sum_y P^t(x,y) f(y),  \quad x \in S.$$

The following concentration of measure result for
real-valued functions of $X_t$ is presented to set the scene, and because it may prove to be of independent interest.
We will not use it in the proof of Theorems~\ref{thm.main-result} and~\ref{thm.main-result-d=1}.

\begin{theorem}
\label{thm.conca-general}
Let $P$ be the transition matrix of a discrete-time Markov chain with
discrete state space $S$. Let $f: S \to \R$ be a function.
\begin{thmenumerate}
\item
Let $(\alpha_i)_{i \in \Z^+}$ be a sequence of positive constants
such that for all $i \in \Z$,
\begin{equation}
\label{cond-gen}
\sup_{x \in S, y \in N(x)}|\E_x [f(X_i)]- \E_y [f(X_i)]| \le \alpha_i.
\end{equation}
Then for all $a > 0$, $x_0 \in S$,
and $t > 0$,
\begin{equation*}
\P_{x_0} (|f(X_t)-\E_{x_0} [f(X_t)] |\ge a ) \le
2e^{-a^2/2(\sum_{i=0}^{t-1} \alpha_i^2)}.
\end{equation*}
\item
More generally, let $S_0$ be a non-empty subset of $S$, and let \\
$(\alpha_i)_{i \in \Z^+}$ be a sequence of positive constants
such that, for all $i\in \Z$,
\begin{equation*}
\sup_{x,y\in S_0: y \in N(x)}|\E_x [f(X_i)]-\E_y [f(X_i)]| \le \alpha_i.
\end{equation*}
Let
$S_0^0= \{ x \in S_0: N(x) \subseteq S_0\}$.
Then for all $x_0 \in S_0^0$, $a >0$ and $t >0$,
\begin{equation*}
\P_{x_0} \Big (\{|f(X_t)-\E_{x_0} [f(X_t)] |\ge a \}\cap \{X_s
\in S_0^0 : \mbox{ } 0 \le s \le t-1\} \Big )\le
2e^{-a^2/2(\sum_{i=0}^{t-1} \alpha_i^2)}.
\end{equation*}
\end{thmenumerate}
\end{theorem}

This result is suitable for use in circumstances where the best available bound on $|\E_x [f(X_i)]-\E_y [f(X_i)]|$ does
not vary much over $y \in N(x)$.  In a situation where better bounds are available for ``most'' transitions out of a state
$x$, then our next inequality, Theorem~\ref{thm.concb-general}, is more appropriate.


Our proof of Theorem~\ref{thm.conca-general}
makes use of a concentration inequality from~\cite{cmcd98}.
Let $(\widetilde{\Omega}, \widetilde{\cF}, \widetilde{\P})$  be a
probability space, with $\widetilde{\Omega}$ finite.
(The arguments used here could be extended
to many cases where $\widetilde{\Omega}$ is countably infinite.)
Let
$\widetilde{\mathcal G} \subseteq \widetilde{\cF}$ be a
$\sigma$-field of subsets of $\widetilde{\Omega}$. Then there exist disjoint sets $\widetilde{G}_1, \ldots, \widetilde{G}_m$ such that $\widetilde{\Omega} = \cup_{r=1}^m \widetilde{G}_r$ and every set in $\widetilde{\mathcal G}$ can be written as a union of some of the sets $\widetilde{G}_r$. Given a bounded random variable $Z$ on
$(\widetilde{\Omega}, \widetilde{\cF}, \widetilde{\P})$, the {\it conditional supremum} $\sup (Z \mid \widetilde{\mathcal G})$ of
$Z$ in $\widetilde{\mathcal G}$ is the $\widetilde{\mathcal G}$-measurable
function given by
\begin{equation*}
\sup (Z \mid \widetilde{\mathcal G}) (\widetilde{\omega}) = \min_{\widetilde{A} \in \widetilde{\mathcal G}: \widetilde{\omega} \in \widetilde{A}} \max_{\widetilde{\omega}'
  \in \widetilde{A}} Z(\widetilde{\omega}')= \max_{\widetilde{\omega}' \in {\widetilde G}_{r}} Z(\widetilde{\omega}'),
\end{equation*}
where $\widetilde{\omega} \in \widetilde{G}_{r}$.
Thus $\sup (Z\mid \widetilde{\cG})$ takes the value at
$\widetilde{\omega}$ equal to the maximum value of $Z$ over the event $\widetilde{G}_r$ in
$\widetilde{\mathcal G}$ containing $\widetilde{\omega}$.

The {\em conditional range} $\ran (Z)$ of $Z$ in $\widetilde{\mathcal G}$
is the $\widetilde{\mathcal G}$-measurable function
\begin{equation*}
\ran (Z \mid \widetilde{\mathcal G}) = \sup (Z \mid \widetilde{\mathcal G})+\sup (-Z \mid \widetilde{\mathcal
  G}),
\end{equation*}
that is, for $\widetilde{\omega} \in \widetilde{G}_r$,
\begin{eqnarray*}
\ran (Z \mid \widetilde{\mathcal G}) (\widetilde{\omega}) =
\max_{\widetilde{\omega}_1, \widetilde{\omega}_2
  \in \widetilde{G}_r} |Z(\widetilde{\omega}_1)-Z(\widetilde{\omega}_2)|.
\end{eqnarray*}

Let $t \in \N$,
let $\{\emptyset, \widetilde{\Omega} \} = \widetilde{\cF}_0 \subseteq \widetilde{\cF}_1
\subseteq \ldots \subseteq \widetilde{\cF}_t$ be a filtration in $\widetilde{\cF}$,
and let $Z_0, \ldots , Z_t$ be the martingale
defined by $Z_i = \widetilde{\E} (Z|{\widetilde{\cF}}_i)$ for each $i=0, \ldots, t$.
For each $i$, let $\ran_i$ denote $\ran (Z_i|{\widetilde{\cF}}_{i-1})$;
by definition, $\ran_i$ is an $\widetilde{\cF}_{i-1}$-measurable function.
For each $j$, let the {\em sum of squared conditional ranges} $R_j^2$ be the random variable $\sum_{i=1}^j \ran^2_i$,
and let the {\em maximum sum of squared conditional ranges} $\widehat r_j^2$
be the supremum of the random variable $R_j^2$, that is
$$\widehat r_j^2 = \sup_{\widetilde{\omega} \in \widetilde \Omega} R_j^2 (\widetilde \omega).$$

The following result is Theorem 3.14 in~\cite{cmcd98}.

\begin{lemma}
\label{thm.mart}
Let $Z$ be a bounded random variable on a finite probability space $(\widetilde{\Omega}, \widetilde{\cF}, \widetilde{\P})$
with $\widetilde{\E} (Z) = m$. Let $\{\emptyset, \widetilde{\Omega} \} = \widetilde{\cF}_0 \subseteq
\widetilde{\cF}_1 \subseteq \ldots \subseteq \widetilde{\cF}_t$ be a filtration in $\widetilde{\cF}$, and
assume that $Z$ is $\widetilde{\cF}_t$-measurable.  Then for any $a \ge 0$,
$$ \widetilde{\P} (|Z-m |\ge a) \le 2e^{-2a^2/\widehat r_t^2}.$$
More generally, for any $a \ge 0$ and any value $r_t^2$,
$$ \widetilde{\P} (\{|Z-m |\ge a \} \cap \{R_t^2 \le r_t^2\}) \le 2e^{-2a^2/r_t^2}.$$
\end{lemma}

\begin{proofof}{Theorem~\ref{thm.conca-general}}
Let us start with (i).
Let $f:S \rightarrow \R$ be a function. Fix a time $t \in \N$,
and an initial state $x_0 \in S$ and consider the evolution of $X_t$ conditional on
$X_0=x_0$ for $t$ steps, that is until time $t$. Since we have assumed
that there are only a finite number of possible transitions from any
given $x \in S$, we can build this process until
time $t$ on a
finite probability space $(\widetilde{\Omega}, \widetilde{\cF},
\widetilde{\P}_{x_0})$: we can take $\widetilde{\Omega}$ to be
the finite set of all possible paths $x_0, \ldots, x_t$ of the process starting at time 0 in
state $x_0$ until time $t$.
For each time $j= 0, \ldots, t$, let $\widetilde{\cF}_j=\sigma (X_0, \ldots , X_j)$.
Also, we let $\widetilde{\cF} = \widetilde{\cF}_t$.

Consider the random variable $Z=
f(X_t): \widetilde{\Omega} \to \R$; note that
$f(X_t)$ is $\widetilde{\cF}_{t}$-measurable. Also, for $j =0, \ldots, t$ let $Z_j$
be given by
$$Z_j = \widetilde{\E}_{x_0} [f(X_t)|\widetilde{\cF}_j] = \widetilde{\E}_{x_0}
[f(X_t)| X_0, \ldots , X_j]= (P^{t-j} f)(X_j),$$
where we have used the Markov property in the last equality.

Fix $1 \le j \le t$; we want to upper bound $\ran_j= \ran (Z_j \mid
\widetilde{\cF}_{j-1})$. The $\sigma$-field $\widetilde{\cF}_{j-1}$ can be decomposed into events
$\{\widetilde{\omega}: \widetilde{\omega}_i = x_i \mbox{ for all } i \le j-1\}$,
for different possible paths $x_0,x_1, \ldots, x_{j-1}$ of $X$.
Fix
$x_1, \ldots , x_{j-1} \in S$, and for $x \in N(x_{j-1})$ consider
\begin{eqnarray*}
h(x)
 & = & \widetilde{\E}_{x_0} [f(X_t)|X_0=x_0, \ldots, X_j=x] 
 =  (P^{t-j}f)(x).
\end{eqnarray*}
Note that $Z_j(\widetilde{\omega})\in \{h(x): x \in N(x_{j-1}) \}$
for $\widetilde{\omega}$ such that $X_{j-1}
(\widetilde{\omega}) = x_{j-1}$. It follows from (\ref{cond-gen}) that, for such $\widetilde{\omega}$,
\begin{eqnarray*}
\ran_j (\widetilde{\omega})
& =  & \sup_{x,y\in N(x_{j-1})} |h(x)-h(y)|\\
& \le & 2 \sup_{x\in N(x_{j-1})} |(P^{t-j}f)(x_{j-1})- (P^{t-j}f)(x)| \le 2\alpha_{t-j}.
\end{eqnarray*}
It follows that
$$R^2_t(\widetilde{\omega}) \le 4 \sum_{i=0}^{t-1} \alpha^2_i,$$
uniformly over $\widetilde{\omega} \in \widetilde{\Omega}$.
Part (i) of Theorem~\ref{thm.conca-general} now follows from
Lemma~\ref{thm.mart}.

To prove~(ii), observe that the bound
$$\ran_j (\widetilde{\omega})= \ran (Z_j \mid \widetilde{\cF}_{j-1})
(\widetilde{\omega})
\le 2 \alpha_{t-j}$$
still holds on the event $A_t=\{\widetilde{\omega}: X_j (\widetilde{\omega}) \in S_0^0 \mbox{ for } j=0,
\ldots, t-1\}$.
\end{proofof}


The next, more refined, concentration of measure result is the one that we actually use in our proofs.
This result is applicable in situations where the bounds on $|\E_x [f(X_i)] - \E_y [f(X_i)]|$ are
non-uniform over $y \in N(x)$, in particular where there is only a small probability of a transition from $x$ to some $y$
where this difference is large.

\begin{theorem}
\label{thm.concb-general}
Let $P$ be the transition matrix of a discrete-time Markov chain with
discrete state space $S$.
Let $f: S \to \R$ be a function.
Suppose the set $S_0$ and functions $a_{x,i}: S_0 \to \R$ ($x\in S_0, i\in \Z^+$) are such that
\begin{equation*}
|\E_x [f(X_i)] -  \E_y [f(X_i)]| \le a_{x,i}(y), \quad i \in \Z^+, \, x,y \in S_0.
\end{equation*}
Let
$S_0^0 = \{x \in S_0: N(x) \subseteq S_0 \}$.
Assume that, for some sequence $(\alpha_i)_{i \in \Z^+}$ of positive constants,
\begin{equation}
\label{cond-gen-4}
\sup_{x \in S_0^0} (P a_{x,i}^2)(x)
\le \alpha_i^2.
\end{equation}
Let $t >0$, and let $\beta = 2\sum_{i=0}^{t-1}
\alpha_i^2$. Suppose also that $\alpha$ is such that
\begin{equation}
\label{cond-gen-5}
2\sup_{0 \le i \le t -1}  \sup_{x \in S^0_0, y \in N(x)} a_{x,i}(y) \le \alpha.
\end{equation}
Finally, let
$A_t= \{\omega: X_s(\omega) \in S_0^0: 0 \le s \le t-1\}$.
Then, for all $a >0$,
\begin{equation*}
\P_{x_0} \Big ( \left \{ |f(X_t)-\E_{x_0} [f(X_t)] |\ge a \right \} \cap A_t \Big)
\le 2e^{-a^2/(2\beta+2\alpha a/3)}.
\end{equation*}
\end{theorem}

To prove Theorem~\ref{thm.concb-general}, we use another result
from~\cite{cmcd98}.
As before, let $Z$ be a bounded random variable on a finite
probability space $(\widetilde{\Omega}, \widetilde{\cF}, \widetilde{\P})$.
Fix $t  \in \N$, and let $\{\emptyset, \widetilde{\Omega}\}=
\widetilde{\cF}_0 \subseteq \ldots \subseteq \widetilde{\cF}_{t} \subseteq \widetilde{\cF}$ be a
filtration. For $i=0, \ldots, t$, let $Z_i = \widetilde{\E} (Z \mid
\widetilde{\cF}_i)$.
For $i=1, \ldots, t$, define the {\it conditional variance}
$$\var_i = \widetilde{\var} (Z_i \mid \widetilde{\cF}_{i-1})= \widetilde{\E} \Big
((Z_i - \widetilde{\E }(Z_i \mid
\widetilde{\cF}_{i-1}))^2 \mid \widetilde{\cF}_{i-1} \Big ).$$
Further, let $V= \sum_{i=1}^{t} \var_i$, the {\it sum of conditional
variances}. Also, for each $i=1, \ldots,
t$, define the $i$-th {\it conditional deviation}
$$\dev_i = \sup (|Z_i-Z_{i-1}| \mid \widetilde{\cF}_{i-1}),$$ and let the {\it conditional
deviation} be $\dev = \max_i
\dev_i$.
Note that $V$ and $\dev$ are random variables in $(\widetilde{\Omega},
\widetilde{\cF}, \widetilde{\P})$.
The following result is a `two-sided' version (and a simple consequence) of Theorem 3.15 in~\cite{cmcd98}.

\begin{lemma}
\label{thm.mart-b}
Let $Z$ be a random variable on a finite probability space $(\widetilde{\Omega},
\widetilde{\cF}, \widetilde{\P})$ with $\E (Z) = m$. Let
$\{\emptyset, \widetilde{\Omega} \} = \widetilde{\cF}_0 \subseteq
\widetilde{\cF}_1 \subseteq
\ldots \subseteq \widetilde{\cF}_{t}$ be a filtration in $\widetilde{\cF}$, where
$Z$ is $\widetilde{\cF}_{t}$-measurable.
Let $\widehat{\alpha} = \max_{\widetilde{\omega} \in \widetilde{\Omega}} \dev (\widetilde{\omega})$, the maximum conditional deviation. Let $\widehat{\beta}= \max_{\widetilde{\omega} \in \widetilde{\Omega}} V(\widetilde{\omega})$, the maximum sum of conditional variances. Assume that $\widehat{\alpha}$ and $\widehat{\beta}$ are finite.
Then for any $a \ge 0$,
$$ \P (|Z-m |\ge a) \le 2e^{-a^2/(2 \widehat{\beta}+ 2\widehat{\alpha} a/3 )}.$$

More generally, for any $a \ge 0$ and any values $\alpha,\beta \ge 0$,
$$ \P (\{|Z-m |\ge a \} \cap \{V \le \beta\} \cap \{\dev \le \alpha \})
\le 2e^{-a^2/(2\beta+ 2\alpha a/3)}.$$
\end{lemma}


\begin{proofof}{Theorem~\ref{thm.concb-general}}
We start, as in the proof of the previous theorem, by assuming that $S_0 = S$.
Let $f:S \rightarrow \R$ be a function. Fix a time $t \in \N$,
and an $x_0 \in S$; consider the evolution of $X=(X_t)_{t \ge 0}$
conditional on $X_0=x_0$ for $t$ steps, that is until time $t$. Again this
process can be supported by a
finite probability space $(\widetilde{\Omega}, \widetilde{\cF}, \widetilde{\P}_{x_0})$.

For each $j= 0, \ldots, t$ let
$\widetilde{\cF}_j =\sigma (X_0, \ldots , X_j)$, and
let $\widetilde{\cF} = \widetilde{\cF}_t$.
We consider the random variable $Z=
f(X_t): \widetilde{\Omega} \to \R$.
For $j =0, \ldots, t$,  $Z_j$
is given by
$$Z_j = \widetilde{\E}_{x_0} [f(X_t)|\widetilde{\cF}_j]
= \widetilde{\E}_{x_0} [f(X_t)| X_0, \ldots , X_j]
= (P^{t-j} f)(X_j).$$
We want to apply Lemma~\ref{thm.mart-b}, so we need to
calculate the conditional variances $\var_i$. We use the
fact that the variance of a random variable $Y$ is equal to $\frac12 \E
(Y-\widetilde{Y})^2$, where $\widetilde{Y}$ is another random variable with
the same distribution as $Y$ and independent of $Y$.

Fix $1 \le j \le t$ and $x_1, \ldots , x_{j-1} \in S$, and for $x \in S$ consider
$$
h(x) = \widetilde{\E}_{x_0} [f(X_t)|X_0=x_0, \ldots, X_{j-1}=x_{j-1},X_{j}=x]
=  (P^{t-j}f)(x).
$$
Then, for $\widetilde{\omega}$ such that $X_{j-1}(\widetilde{\omega}) =
x_{j-1}$, $Z_j (\widetilde{\omega}) \in \{h(x): x \in N(x_{j-1}) \}$, so
\begin{eqnarray*}
\var_j (\widetilde{\omega}) & = & \frac12 \sum_{x,y} P (x_{j-1},x)
P (x_{j-1},y)(h(x)-h(y))^2\\
 & = & \frac12 \sum_{x,y} P(x_{j-1},x)
P(x_{j-1},y)\Big ( (P^{t-j}
f)(x)-(P^{t-j} f)(y)\Big )^2\\
& \le &  \sum_{x,y} P(x_{j-1},x)
P(x_{j-1},y)\Big ( (P^{t-j}
f)(x)-(P^{t-j} f)(x_{j-1})\Big )^2\\
&&\mbox{} + \sum_{x,y} P(x_{j-1},x)
P(x_{j-1},y)\Big ( (P^{t-j}
f)(x_{j-1})-(P^{t-j} f)(y)\Big )^2\\
& \le & 2 \sum_{x\in N(x_{j-1})} P(x_{j-1},x)
\Big ( (P^{t-j}f)(x)- (P^{t-j}f)(x_{j-1})\Big )^2\\
& \le & 2 \sum_x P(x_{j-1},x)
a_{x_{j-1},t-j} (x)^2\\
& = & 2 (P a_{x_{j-1},t-j}^2)(x_{j-1}) \le 2 \alpha_{t-j}^2,
\end{eqnarray*}
by assumption~(\ref{cond-gen-4}). It follows that
$\widehat{\beta} \le \beta = 2 \sum_{i=0}^{t-1}
\alpha_i^2$.

We now bound $\dev= \max_{1 \le j \le t}\dev_j$. For
$\widetilde{\omega}$
such that $X_{j-1}(\widetilde{\omega}) = x_{j-1}$,
\begin{eqnarray*}
\dev_j (\widetilde{\omega})
& = & \sup_{x\in N(x_{j-1})}|(P^{t-j}f)(x)-(P^{t-j+1}f)(x_{j-1})|\\
& = &  \sup_{x\in N(x_{j-1})}|(P^{t-j}f)(x)-(P(P^{t-j} f))(x_{j-1})|\\
& \le &   \sup_{x\in N(x_{j-1})} \sum_y P (x_{j-1},y) |(P^{t-j}f)(x)-(P^{t-j}f)(y)|\\
& \le &   \sup_{x\in N(x_{j-1})} \sum_y P (x_{j-1},y) |(P^{t-j}f)(x)-(P^{t-j}f)(x_{j-1})|\\
&& \mbox{} + \sum_y P (x_{j-1},y)|(P^{t-j}f)(y)-(P^{t-j}f)(x_{j-1})|\\
& \le & 2 \sup_{x\in N(x_{j-1})}|(P^{t-j}f)(x)-(P^{t-j}f)(x_{j-1})| \le  \alpha,
\end{eqnarray*}
for each $1 \le j \le t$, by assumption (\ref{cond-gen-5}).
Therefore $\widehat{\alpha} \le \alpha$.

Theorem~\ref{thm.concb-general} now follows from the first statement in
Lemma~\ref{thm.mart-b} in the case where $S_0=S$.
In general, the above bounds on $V$ and $\dev$ hold
on the event $A_t=\{\omega: X_i (\omega) \in S_0^0 \mbox{ for } i=0,
\ldots, t-1\}$, and so the full statement of Theorem~\ref{thm.concb-general} also follows from
the second inequality in Lemma~\ref{thm.mart-b}.
\end{proofof}


We now prove that the expectation of a well concentrated function $f$ multiplied by an indicator function approximately factorises, with bounds in terms of bounds on $f$ and its deviations from its mean. This result will be used several times in our proof of Theorems~\ref{thm.main-result} and~\ref{thm.main-result-d=1}.

\begin{lemma}
\label{lem.expec}
Let $(\Omega, {\mathcal F}, \P)$ be a probability space and let $X$ be a random variable on $(\Omega, {\mathcal F}, \P)$ taking values in a measurable space $(S,{\mathcal S})$. Let $f: S \to {\mathbb R}$ be a measurable function, and suppose that $\P (|f(X) - \E f(X)| \ge a) \le b$ and $\P (|f(X)| \le c) = 1$. Let $A \in {\mathcal S}$. Then
$$\Big | \E [ \I_{X \in A} f(X)] - \P (X \in A) \E [f(X)] \Big | \le a \P (X \in A)+ bc.$$
\end{lemma}
\begin{proof}
We have
\begin{eqnarray*}
\E [ \I_{X \in A} f(X)] & = & \E [ \I_{X \in A} f(X) \I_{|f(X) - \E [f(X)]| \le a}]\\
&&\mbox{} +  \E [ \I_{X \in A} f(X) \I_{|f(X) - \E [f(X)]| > a}].
\end{eqnarray*}
Now,
\begin{eqnarray*}
\E [ \I_{X \in A} f(X) \I_{|f(X) - \E [f(X)]| \le a}] & \le & \E [ \I_{X \in A} (\E [f(X)] + a )]\\
& = & \P (X \in A) \E [f(X)] + \P (X \in A) a,
\end{eqnarray*}
and
\begin{eqnarray*}
\E [ \I_{X \in A} f(X) \I_{|f(X) - \E [f(X)]| \le a}] & \ge & \E [ \I_{X \in A} (\E [f(X)] - a )]\\
& = & \P (X \in A) \E [f(X)] - \P (X \in A) a.
\end{eqnarray*}
Also,
$\E [ \I_{X \in A} f(X) \I_{|f(X) - \E [f(X)]| > a}] \le c \E [ \I_{X \in A} \I_{|f(X) - \E [f(X)]| > a}] \le c b$
and
$\E [ \I_{X \in A} f(X) \I_{|f(X) - \E [f(X)]| > a}] \ge -c \E [ \I_{X \in A} \I_{|f(X) - \E [f(X)]| > a}] \ge -c b.$
The result follows.
\end{proof}


\section{Generator of the Markov chain}

\label{S: gen}

We now return to our routing model.
Recall that $f_{v,j} (x)$ denotes the number of links with one end $v$ carrying exactly $j$ calls, and that
$\I_{uv}^j(x) =1$ if $x(\{u,v\}) =j$ and $\I_{uv}^j(x)=0$ otherwise, 
for all $u,v,j$. Also, we let $\I_{uv}^{\le j} (x) =1$ if $x(\{u,v\}) \le j$ and
$\I_{uv}^{\le j}(x) = 0$ otherwise.
Further,
we define $\I_{uv,w}^{\le j}(x) =1$ if $x(\{u,w\}) \lor x(\{v,w\}) \le j$, and $\I_{uv,w}^{\le j}(x) =0$ otherwise.

Let $A$ be the generator operator of the Markov process $X$. By standard theory of Markov chains, for each $t \ge 0$, each $v \in V_n$ and each $j \in \{0, \ldots, C\}$,
\begin{eqnarray*}
\frac{d \E [f_{v,j} (X_t)]}{dt} =  \E [A f_{v,j}(X_t)],
\end{eqnarray*}
so in order to prove Theorem~\ref{thm.main-result} we need to approximate $\E [A f_{v,j}(X_t)]$.

For $x \in S$ and $0 < j < C$, we can write
\begin{eqnarray*}
A f_{v,j}(x) & = & \lambda f_{v,j-1}(x) - \lambda f_{v,j}(x)
 + \lambda  g_{v,j-1}(x) - \lambda g_{v,j}(x)\\
&& \mbox{} - j f_{v,j}(x) + (j+1) f_{v,j+1}(x), \\
A f_{v,0}(x) & = &  - \lambda f_{v,0}(x)- \lambda g_{v,0}(x) + f_{v,1}(x), \\
A f_{v,C}(x) & = & \lambda f_{v,C-1}(x)+\lambda g_{v,C-1}(x)  - C f_{v,C}(x),
\end{eqnarray*}
where the $g_{v,j}(x)$ represent contributions due to alternatively routed arrivals with one end $v$, and are given, for $j=0, \ldots, C-1$, by
\begin{eqnarray}
g_{v,j} &=&
\frac{1}{(n-2)^d}\Big[ \sum_{r=1}^d\sum_{u,{\bf w}}  \I_{uv}^C\I_{vw_r}^{j} \I_{uw_r}^{\le j} \prod_{s=1}^{r-1} (1-\I_{uv,w_s}^{\le j}) \prod_{s=r+1}^d (1-\I_{uv,w_s}^{\le j-1})\nonumber \\
&+& \sum_{r=1}^d \sum_{u,{\bf w}} \I_{uv}^{C}\I_{vw_r}^j \sum_{i=j+1}^{C-1} \I_{uw_r}^{i}\prod_{s=1}^{r-1} (1-\I_{uv,w_s}^{\le i}) \prod_{s=r+1}^d (1-\I_{uv,w_s}^{\le i-1}) \nonumber \\
&+& \sum_{r=1}^d \sum_{u,v',{\bf w}_r}  \I_{uv'}^{C}\I_{uv}^{j} \I_{v'v}^{\le j} \prod_{s=1}^{r-1} (1-\I_{uv',w_s}^{\le j}) \prod_{s=r+1}^d (1-\I_{uv',w_s}^{\le j-1}) \nonumber \\
&+& \!\sum_{r=1}^d\sum_{u,v',{\bf w}_r} \I_{uv'}^{C}\I_{uv}^{j} \sum_{i=j+1}^{C-1} \I_{v'v}^{i}\prod_{s=1}^{r-1} \!(1-\I_{uv',w_s}^{\le i}) \!\prod_{s=r+1}^d \! (1-\I_{uv',w_s}^{\le i-1})\Big], \label{eq-g}
\end{eqnarray}
with ${\bf w} = (w_1, \ldots, w_d) \in V_n^d$, ${\bf w}_r = (w_1, \ldots, w_{r-1}, w_{r+1}, \ldots, w_{d}) \in V_n^{d-1}$. Here, $\sum_{u,{\bf w}}$
denotes the sum over all $u \not = v$, and over all $w_1, \ldots, w_d$ such that
each $w_r \not = u,v$,
and
$\sum_{u,v',{\bf w}_r}$ denotes the sum over all $u \not = v$, $v' \not = u,v$ and over all
$w_1,\ldots, w_{r-1},w_{r+1}, \ldots, w_d$ such that each $w_j \not = u,v'$.

In (\ref{eq-g}), the first term is the probability that the direct link chosen for a new call with one end $v$ is blocked and, on the two-link path selected
for the call, the link including $v$ has load $j$, while its partner link has load at most~$j$.  The second term is the probability that the direct link chosen
for a new call with one end $v$ is blocked and, on the two-link path selected for the call, the link including $v$ has load $j$ and its partner link has load
greater than~$j$. The third term is the probability that $v$ is chosen as an intermediate node for a call 
blocked from its direct link, the route through $v$ is the best out of the $d$ routes selected and $j$ is the maximum load of a link on the route.
The fourth term is the probability that $v$ is chosen as an intermediate node for a call
blocked from its direct link, the route through $v$ is the best of the $d$ routes selected, 
and $j$ is not the maximum load of a link on the route.

In particular, when $d=1$, for $j=0, \ldots, C-1$, for each $v \in V_n$,
\begin{eqnarray*}
g_{v,j} & = & \frac{1}{n-2}\sum_{u, w \in V_n} (\I_{uv}^{C}\I_{vw}^{j}\I_{uw}^{\le C-1} + \I_{uw}^{C}
\I_{vw}^{j}\I_{uv}^{\le C-1}).
\end{eqnarray*}
Furthermore, when $d=2$, then $g_{v,j}$ is a sum of $8$ contributions. These contributions correspond to the case where $v$ is an end node and the case where $v$ is an intermediate node for a call. In the case where $v$ is an end node, we have a subcase where the first route of those selected is allocated to a new call and a subcase where the second route of those selected is allocated to a new call.
We further need to distinguish a subcase where the link ending in $v$ has the maximum load, and a subcase where the link ending in $v$ does not have the maximum load on the route allocated to a new call.
In the case where $v$ is an intermediate node, we need to distinguish a subcase where $v$ is the first alternative node selected and a subcase where $v$ is the second alternative node selected. Also, we have a subcase where the
link with load $j$ has the maximum load, and a subcase where a link with load $j$ does not have the maximum load on the route allocated to a new call.

The contribution due to the case where $v$ is an end node, the first route of those selected is allocated to a new call, and the link ending in $v$ has the maximum load is of the form
$$\frac{1}{(n-2)^2}\sum_{u, w_1,w_2 \in V_n} \I_{uv}^{C}\I_{vw_1}^{j} \I_{uw_1}^{\le j} (1-  \I_{uv,w_2}^{\le j-1}).$$
The contribution due to the case where $v$ is an end node, the second route of those selected is allocated to a new call, and the link ending in $v$ does not have the maximum load is of the form
$$\frac{1}{(n-2)^2}\sum_{u,w_1,w_2 \in V_n} \I_{uv}^{C}\I_{vw_2}^{j}
 \sum_{i= j+1}^{C-1}
\I_{uw_2}^{i} (1-  \I_{uv,w_1}^{\le i}).$$
The contribution due to the case where $v$ is an intermediate node and is selected first, and where a link with load $j$ has the maximum load on the route allocated to a new call is of the form
$$\frac{1}{(n-2)^2}\sum_{u, v',w \in V_n} \I_{uv'}^{C}\I_{uv}^{j}\I_{v'v}^{\le j} (1-  \I_{uv',w}^{\le j-1}).$$
The remaining contributions can be expressed analogously, and the form of $g$ for $d > 2$ is derived similarly.

We note that each $g_{v,j}$ is a sum of products of indicators of sets of load vectors with properties pertaining to loads of particular links.
Our plan is to justify the intuition that, subject to suitable initial conditions, the loads on different links behave nearly as iid random variables at each time $t$,
and the precise estimates we use involve sums over reasonably large collections of links.
We need several specific manifestations of this near-independence and symmetry.
First, the geometry of the network is not important; this means that, for fixed nodes $u$ and $v$, the loads on links $uw$ and $vw$ are not strongly correlated, on average over $w$, so that the average value of $\I_{uw}^j \I_{vw}^k$ over $w$ is close to the product of the average values of $\I_{uw}^j$
and $\I_{vw}^k$. (In other words, the function $\phi^1$ defined earlier is small.) Secondly, for fixed nodes $u$ and $v$, the loads on links incident with $u$ have approximately the same distribution as the loads on links incident with $v$. (This means that the function $\phi^2$ is small.) Thirdly, we require that the alternatively routed calls are fairly uniformly distributed over the network. (This implies that the function $\phi^3$ is small.)
%
Finally, we will show that each $f_{v,j} (X_t)$ is well concentrated around its expectation, which then implies that, for fixed nodes $u$ and $v$, $\E [f_{u,j} (X_t)f_{v,j} (X_t)]$ is approximately equal to $\E [f_{u,j} (X_t)] \E [f_{v,j} (X_t)]$.
 Naturally, quantitative versions of these properties will need to be assumed to hold at time~0, and we will show that they are maintained throughout the time period of interest.
%
This will then allow us to express $\E [g_{v,j}(X_t)]$ as a (scaled) sum of products of terms of the form $\E [ f_{v,i}(X_t)]$, 
%
and hence lead to approximate differential equations satisfied by the $\E [f_{v,j}(X_t)]$ for $j=0,1, \ldots, C$, for each $v \in V_n$, expressed in terms of themselves.


Let $f_{v,\le j} (x) = \sum_{i \le j} f_{v, i} (x)$. Let $f_{v,j}(t)=\E [f_{v,j}(X_t)]$, and let $f_{v,\le j}(t)=\E [f_{v,\le j}(X_t)]$.
Let $\I_{uv}^{j} (t) = \E [ \I_{uv}^{j} (X_t)]$ and let $\I_{uv}^{\le j} (t) = \E [ \I_{uv}^{\le j} (X_t)]$. We will show that
the expectation of the first term in~(\ref{eq-g}) with respect to the law of $X_t$ is approximately
\begin{eqnarray*}
\lefteqn{\frac{1}{(n-2)^d}\sum_{r=1}^d\sum_{u}\I_{uv}^{C}(t) \E \Big [\Big (\sum_{{\bf w}}
\I_{vw_r}^{j} \I_{uw_r}^{\le j} \prod_{s=1}^{r-1} (1-\I_{uv,w_s}^{\le j})} \\
&&\mbox{} \times \prod_{s=r+1}^d (1-\I_{uv,w_s}^{\le j-1})\Big )(X_t)\Big ]\\
& \approx & \frac{1}{(n-1)^{2d}}\sum_{r=1}^d\sum_{u}\I_{uv}^{C}(t) \E \Big [\Big (\sum_{{\bf w}, {\bf w'}}
\I_{vw_r}^{j} \I_{uw'_r}^{\le j} \prod_{s=1}^{r-1} (1-\I_{uw_s}^{\le j}\I_{vw'_s}^{\le j})\\
&& \mbox{} \times \prod_{s=r+1}^d (1-\I_{uw_s}^{\le j-1}\I_{vw'_s}^{\le j-1})\Big )(X_t)\Big ]\\
& \approx & \frac{1}{(n-1)^{2d}}f_{v,C}(t) f_{v,j}(t) f_{v,\le j}(t) \sum_{r=1}^d
(1-(f_{v,\le j}(t))^2)^{r-1} \\
&&\mbox{} \times (1-(f_{v,\le j-1}(t))^2)^{d-r}.
\end{eqnarray*}
Handling the other terms similarly, we prove that $\E [g_{v,j}(X_t)]$ is close to
\begin{eqnarray*}
\lefteqn{\frac{2f_{v,C}(t) f_{v,j}(t) f_{v,\le j}(t)}{(n-1)^{2d}} \sum_{r=1}^d
(1-(f_{v,\le j}(t))^2)^{r-1} (1-(f_{v,\le j-1}(t))^2)^{d-r}}\\
&& \mbox{} + \frac{2f_{v,C}(t) f_{v,j}(t)}{(n-1)^{2d}} \sum_{i=j+1}^{C-1} f_{v,i}(t)\sum_{r=1}^d (1-(f_{v,\le i}(t))^2)^{r-1} (1-(f_{v,\le i-1}(t))^2)^{d-r}.
\end{eqnarray*}

Hence we will see that the functions $\E [f_{v,j}(X_t)]$ ($j=0, \ldots, C$, $v \in V_n$) approximately solve the differential equation~(\ref{eq.diff-eq}). As the $f_{v,j} (X_t)$ are well concentrated around their expectations, Theorems~\ref{thm.main-result} and~\ref{thm.main-result-d=1} will follow.


\section{Coupling}

\label{sec:couple}

In this section, we describe and analyse a natural coupling between two copies of process $X$. We start by defining notions of `distance' between the two copies, with the aim of showing that the expected distance grows slowly in time, at least for time $O(1)$.

Given two load vectors $x,y$, the $\ell_1$-distance between them is
$$\|x-y\|_1 = \sum_{e \in L_n} |x(e,0) - y(e,0)| + \sum_{e \in L_n, w \not \in e} |x (e,w) - y(e,w)|,$$
which measures the sum of the differences in loads between $x$ and $y$.
Then $\| \cdot \|$ is a metric on $S$.
For $v \in V_n$ we will also consider functions
\begin{eqnarray*}
\|x-y\|_v & = & \sum_{e: v \in e}|x(e,0) -y(e,0)| + \sum_{e: v \in e}\sum_{w \not \in e}  |x (e,w) - y(e,w) | \\
&&\mbox{} + \sum_{e: v \not \in e}  |x(e,v)-y(e,v) |\\
& = & \sum_{u\not = v}|x(\{u,v\},0) -y(\{u,v\},0)| \\
&&\mbox{} + \sum_{u \not = v} \sum_{w \not = u,v} |x (\{u,v\},w) - y(\{u,v\},w) | \\
&& \mbox{} + \sum_{\{u,w\}: u \not = w, v \not \in \{u,w\}}
|x(\{u,w\},v)-y(\{u,w\},v)|.
\end{eqnarray*}
Then $2\|x-y\|_{v}$ 
gives an upper bound on the sum of the differences between the loads
of links $\{v,w\}$, $w \not = v$ (i.e. links around node $v$)
in $x$ and $y$.

Let $\widetilde{S} \subseteq S$ be the set of load vectors $x$ such that
$$
\|x\|_1=\sum_e x(e,0) + \sum_{e,w \not \in e} x(e,w) \le 6 \lambda {n \choose 2},
$$
that is the subset of the state space consisting of those load vectors where the total
number of calls in the network is at most $6 \lambda {n \choose 2}$.

For the remainder of this section, and also in other parts of the paper, we will work with a jump chain $\widehat{X}$, that corresponds to $X$ while
the chain is in $\widetilde{S}$. This discrete time chain $\widehat{X}$ is not the standard embedded chain, but a slower moving version that will often not change
its state at a given step, as we now describe.
Given that the current state, at time $t \in \Z^+$, is $x \in \widetilde{S}$, the next event is an arrival with probability
\begin{equation} \label{eq.plambdan}
p(\lambda,n) = \frac{\lambda {n \choose 2}}{\lambda {n \choose 2}+ \lfloor 6\lambda {n \choose 2} \rfloor}
\end{equation}
and a {\it potential} departure with probability $1-p(\lambda,n)$.
Given that the event is an arrival, each pair of endpoints $u,v$ is chosen
with probability $1/{n \choose 2}$, then each $d$-tuple of intermediate nodes is
chosen with probability $(n-2)^{-d}$, and the call is routed along the
two-link route chosen first among the $d$ selected that minimises the maximum load of a link. Given
that the event is a potential departure, the calls currently in the
system are enumerated from 1 up to at most $\lfloor 6 \lambda {n \choose 2}\rfloor $,
and then a number is chosen uniformly at random from $\{1, \ldots, \lfloor 6\lambda {n \choose 2} \rfloor\}$.
If there is a call assigned to this number, it departs, and otherwise nothing happens.
If $x \notin \widetilde{S}$, then the chain does not move: we shall show that it is unlikely for the chain to leave $\widetilde{S}$ over the
time scales we are considering.  Let $(\widehat{\cF}_t)$ denote the
natural filtration of $(\widehat{X}_t)$.

Let $S_0 \subseteq \widetilde{S}$ be the set of states $x$ such that
$\|x\|_1 \le 4 \lambda {n \choose 2}$.  Recall also that $S_1 \subseteq S_0$ is the set of states $x$ such that
$\|x\|_1\le 2 \lambda {n \choose 2}$.  We will be 
interested in the evolution of the chain starting from $S_1$ and before it leaves~$S_0$.

Consider the following family of Markovian couplings $(\widehat{X}^{x_0},\widehat{Y}^{y_0})$
of pairs of copies $\widehat{X}^{x_0},\widehat{Y}^{y_0}$ of the discrete
jump chain starting from states $x_0,y_0$ respectively, where $x_0,y_0 \in S_0$. (In what follows, we will drop the superscripts
$x_0,y_0$ from the notation and refer simply to $\widehat{X}$ and $\widehat{Y}$.)

Let $t \ge 0$, and let $x,y$ be both in $\widetilde{S}$.
Given that $\widehat{X}_{t-1} = x$ and $\widehat{Y}_{t-1} = y$, the transition at time $t$ (from state
$(\widehat{X}_{t-1}, \widehat{Y}_{t-1})$ to $(\widehat{X}_t, \widehat{Y}_t)$) is
an arrival in both $\widehat{X}$ and $\widehat{Y}$, or a potential
departure in both  $\widehat{X}$ and $\widehat{Y}$. Given that the transition
is an arrival, we choose the same call endpoints and the same $d$-tuple of intermediate nodes in
both. Given that the transition is a potential departure, we pair the
calls occupying the same route in both $\widehat{X}$ and $\widehat{Y}$, as much as possible.
We also pair off the remaining calls arbitrarily, as much as possible, in some fashion depending only
on the current states. (We can pair off all the calls if $\|x\|_1 = \|y\|_1$, otherwise some remain unpaired in the process with more calls.)
We assign to each pair, and to each unpaired call, a distinct number in
$\{1, \ldots, \lfloor 6 \lambda {n \choose 2} \rfloor\}$.
If the transition at time $t$ is a potential departure, we choose the same
uniformly random number from $\{1, \ldots, \lfloor 6 \lambda {n \choose 2} \rfloor \}$
for both $\widehat{X}$ and $\widehat{Y}$. If the number corresponds to a pair of calls,
both depart; if it corresponds to an unpaired call, this call 
departs; otherwise, nothing happens.

Suppose now that $\widehat{X}_{t-1}=x$ and $\widehat{Y}_{t-1} =y$ where at least one of $x$ and $y$, say $x$, is not in
$\widetilde{S}$.  We then let $\widehat{X}_t = \widehat{X}_{t-1}=x$, while $\widehat{Y}_t$
is obtained from $\widehat{Y}_{t-1}$ by following the transition
probabilities of the jump chain from $y$ if $y \in \widetilde{S}$ or $\widehat{Y}_t = \widehat{Y}_{t-1}=y$ if $y \not \in \widetilde{S}$.

The process
$(\widehat{X}_t, \widehat{Y}_t)$ is a Markov chain
adapted to its natural filtration~$({\mathcal G}_t)$.
Given that $\widehat{X}_{t-1}, \widehat{Y}_{t-1} \in \widetilde{S}$,
on the event that the jump at time $t$ is a potential departure,
$\|\widehat{X}_t - \widehat{Y}_t \|_1 \le \|\widehat{X}_{t-1}-\widehat{Y}_{t-1}\|_1$.
(The distance remains unchanged if paired calls from the same route depart or if there is no departure at all; it decreases by 2 if paired calls on different routes depart, and decreases by 1 if an unpaired call departs.)
The distance between $\widehat{X}$ and
$\widehat{Y}$ can only increase by 2 at a jump, and then only if the jump is an arrival and if we select at least one of the links where
$\widehat{X}_{t-1}$ and $\widehat{Y}_{t-1}$ differ.
This happens with probability at most
$$
\frac{2d+1}{{n \choose 2}} \sum_{e \in L_n} \Big |\widehat{X}_{t-1} (e) - \widehat{Y}_{t-1}(e)\Big |
\le \frac{3d}{{n \choose 2}} \sum_{e \in L_n} \Big |\widehat{X}_{t-1} (e) - \widehat{Y}_{t-1}(e)\Big |,
$$
and $\sum_{e \in L_n} |\widehat{X}_{t-1} (e) - \widehat{Y}_{t-1}(e) |$ is equal to
\begin{eqnarray*}
 \lefteqn{\sum_{\{u,v\}: u \not = v} \Big |(\widehat{X}_{t-1} (\{u,v\},0) - \widehat{Y}_{t-1}(\{u,v\},0))}\\
&& \mbox{}+  \sum_{w \not =u, v} (\widehat{X}_{t-1}(\{v,w\},u) - \widehat{Y}_{t-1}(\{v,w\},u))\\
&& \mbox{} +  \sum_{w \not = u, v} (\widehat{X}_{t-1}(\{u,w\},v) - \widehat{Y}_{t-1}(\{u,w\},v))\Big |
\le 2\|\widehat{X}_{t-1} - \widehat{Y}_{t-1}\|_1.
\end{eqnarray*}
It follows that, 
uniformly over all $x,y \in \widetilde{S}$,
$$\E \left [\|\widehat{X}_t-\widehat{Y}_t\|_1 \mid \widehat{X}_{t-1}=x,
\widehat{Y}_{t-1}=y \right ]\le \Big (1 + \frac{12d}{{n \choose 2}} \Big ) \| x - y\|_1.$$

We have assumed that
$\widehat{X}_0=x_0$ and $\widehat{Y}_0 = y_0$, where $x_0, y_0 \in S_0$, that is $\|\widehat{X}_0\|_1
\le 4 \lambda {n \choose 2}$  and $\|\widehat{Y}_0\|_1 \le 4 \lambda {n \choose 2}$.
Note that, whenever $\|\widehat{X}_{t-1}\|_1 \ge 4 \lambda {n \choose 2}$,
$$\P [\|\widehat{X}_t\|_1 - \|\widehat{X}_{t-1}\|_1=1 \mid \widehat{X}_{t-1} ] \le \frac{\lambda \binom{n}{2}}{\lambda \binom{n}{2} + \lfloor 6 \lambda \binom{n}{2} \rfloor} \le \frac{1}{6},$$
provided $n \ge \max(3, \frac{1}{\lambda})$, and
$$\P [\|\widehat{X}_t\|_1 - \|\widehat{X}_{t-1}\|_1=-1 \mid \widehat{X}_{t-1} ] \ge \frac{4\lambda \binom{n}{2}}{\lambda \binom{n}{2} +
\lfloor 6 \lambda \binom{n}{2} \rfloor} \ge \frac{4}{7} \ge \frac{1}{2}.$$
Therefore, by standard inequalities (see, for instance, Lemma 2.5 in~\cite{lm08},
with $p=1/6$, $q=1/2$ and $a = 2\lambda \binom{n}{2}$), for any constant $c > 0$,
\begin{equation} \label{eq.crossing}
\P \Big (\exists t \le cn^2: \|\widehat{X}_t \|_1 \lor \|\widehat{Y}_t\|_1 \ge 6
\lambda {n \choose 2} \Big ) \le
2cn^2 \left( \frac13 \right)^{2\lambda \binom{n}{2}}.
\end{equation}

Let $D_s$ be the event $\{\widehat{X}_u \in \widetilde{S}, \widehat{Y}_u \in \widetilde{S} \mbox{ for all } u \le s\}$. Then
\begin{eqnarray*}
\E (\|\widehat{X}_t - \widehat{Y}_t\|_1 \I_{D_{t-1}}) & = & \E \Big [\E(\|\widehat{X}_t - \widehat{Y}_t\|_1 \I_{D_{t-1}} \mid \cG_{t-1}) \Big ]\\
& \le & 
\E \Big [\Big (1 + \frac{12d}{{n \choose 2}} \Big ) \| \widehat{X}_{t-1} - \widehat{Y}_{t-1}\|_1 \I_{D_{t-1}} \Big ]\\
& \le & \Big (1 + \frac{12d}{{n \choose 2}} \Big ) \E  [\|\widehat{X}_{t-1} - \widehat{Y}_{t-1}\|_1 \I_{D_{t-2}}].
\end{eqnarray*}
By induction, for starting states $x_0,y_0 \in S_0$,
$$
\E (\|\widehat{X}_t - \widehat{Y}_t\|_1 \I_{D_{t-1}}) \le \Big (1 + \frac{12d}{{n \choose 2}} \Big )^t \|x_0-y_0\|_1.
$$
Since the chain stops once it leaves $\widetilde{S}$ and each jump changes $\|\widehat{X}_t\|_1$ and $\|\widehat{Y}_t\|_1$ by at most 1, on the event $\overline{D_{t-1}}$, $\|\widehat{X}_t - \widehat{Y}_t\|_1 \le 2(6 \lambda {n \choose 2}+1)$.
Hence
\begin{eqnarray*}
\E  (\|\widehat{X}_t - \widehat{Y}_t\|_1) & = & \E (\|\widehat{X}_t - \widehat{Y}_t\|_1 \I_{D_{t-1}})
 +  \E (\|\widehat{X}_t - \widehat{Y}_t\|_1 (1-\I_{D_{t-1}}))\\
& \le & \Big (1 + \frac{12d}{{n \choose 2}} \Big )^t \|x_0-y_0\|_1 + 2(6 \lambda {n \choose 2} +1)
2cn^2 \left( \frac{1}{3} \right)^{2\lambda \binom{n}{2}}\\
&\le& \Big (1 + \frac{12d}{{n \choose 2}} \Big )^t \|x_0-y_0\|_1 + 14 \lambda n^5
\left( \frac13 \right)^{\frac12 \lambda n^2},
\end{eqnarray*}
for any constant $c \le n$, and any $t \le cn^2$.
If $n \ge \max(1000,1/\lambda)$, this last term is at most~1.  Therefore, for $n \ge \max(1000,1/\lambda,c)$, and $t \le cn^2$, uniformly over starting states $x_0,y_0$ in $S_0$,
\begin{equation*}
\E (\|\widehat{X}_t - \widehat{Y}_t\|_1) \le 2 \Big (1 + \frac{12d}{{n \choose 2}} \Big )^t \|x_0-y_0\|_1.
\end{equation*}



Let $v$ be a node. While in $\widetilde{S}$, we can only change the loads of links at $v$ (i.e. links $\{v,w\}$, for
$w \not = v$) if we choose a link with end $v$ at a jump time.  Also, we can only make
$\|\widehat{X}_t - \widehat{Y}_t\|_{v}$ bigger than $\|\widehat{X}_{t-1} -
\widehat{Y}_{t-1}\|_{v}$ at an arrival time, if either we pick one of the links $\{v,w\}$
(if any) in which $\widehat{X}_{t-1}$ and $\widehat{Y}_{t-1}$ differ, or if we
pick a link $\{u,w\}$ (where $u \not = v$ and $w \not = u,v$)
in which $\widehat{X}_{t-1}$ and $\widehat{Y}_{t-1}$ differ, and also node
$v$ as an endpoint or an intermediate node for a new call.
The former happens with conditional probability at most
\begin{eqnarray*} \lefteqn{\frac{2d+1}{{n \choose 2}} \sum_{w \not = v}
\Big |\widehat{X}_{t-1}(\{v,w\}) - \widehat{Y}_{t-1}(\{v,w\})\Big|} \\
&\le& \frac{3d}{{n \choose 2}} \sum_{w \not = v} \Big |\widehat{X}_{t-1}(\{v,w\}) - \widehat{Y}_{t-1}(\{v,w\})\Big| \\
&\le& \frac{6d}{{n \choose 2}} \|\widehat{X}_{t-1} - \widehat{Y}_{t-1}\|_{v}.
\end{eqnarray*}
and
The latter happens with conditional probability at most
$$
\frac{1}{n} \frac{2d^2}{{n \choose 2} - (n-1)} \sum_{\{u,w\} \in L_n: v \not \in \{u,w\}}
|\widehat{X}_{t-1} (\{u,w\}) - \widehat{Y}_{t-1}(\{u,w\})|
$$
$$
\le \frac{8d^2}{n{n \choose 2}}\|\widehat{X}_{t-1} - \widehat{Y}_{t-1}\|_1,
$$
provided $n \ge 4$.  Also, always with probability 1,
$\|\widehat{X}_t - \widehat{Y}_t\|_{v} \le \|\widehat{X}_{t-1} -
\widehat{Y}_{t-1}\|_{v} + 2$.
Then, for $n \ge 4$,
\begin{eqnarray*}
\E (\|\widehat{X}_t - \widehat{Y}_t\|_v \I_{D_{t-1}}) & \le & \Big ( 1 + \frac{12d}{{n \choose 2}} \Big )
\E (\|\widehat{X}_{t-1} - \widehat{Y}_{t-1}\|_{v} \I_{D_{t-2}}) \\
&& \mbox{} + \frac{16d^2}{n{n \choose 2}} \Big (1 + \frac{12d}{{n \choose 2}} \Big )^{t-1} \|x_0-y_0\|_1,
\end{eqnarray*}
and so $\E (\|\widehat{X}_t - \widehat{Y}_t\|_{v} \I_{D_{t-1}})$ is at most
$$
\Big ( 1 + \frac{12d}{{n \choose 2}} \Big )^t \|x_0-y_0\|_{v}
+ \frac{16d^2t}{n{n \choose 2}} 
\Big (1 + \frac{12d}{{n \choose 2}} \Big )^{t-1}\|x_0-y_0\|_1.
$$

Hence, as before, for all $n \ge \max(1000,1/\lambda,c)$ and $t \le cn^2$, for each $v \in V_n$, and for all $x_0,y_0 \in S_0$,
\begin{eqnarray*}
\E (\|\widehat{X}_t - \widehat{Y}_t\|_{v})& \le & 2\Big ( 1 + \frac{12d}{{n \choose 2}} \Big )^t \|x_0-y_0\|_{v}
\\
&& \mbox{} + \frac{32d^2t}{n{n \choose 2}}\Big (1 + \frac{12d}{{n \choose 2}} \Big )^{t-1}\|x_0-y_0\|_1 .
\end{eqnarray*}


Recall that, for a load vector $x$, $f_{v,k}(x)$ is
the number of links around $v$ carrying exactly $k$
calls. Similarly, $f_{v,\le k}(x)= \sum_{i \le k} f_{v,i}(x)$ is the number of links $\{v,w\}$, $w \not = v$, such that $x(\{v,w\}) \le k$; that is, the number of links around $v$ carrying at most $k$
calls. Let $P$ denote the transition matrix of the jump chain $(\widehat{X}_t)$ restricted to $\widetilde{S}$.
Note that, for each $v,k$, and each $x,y \in \widetilde{S}$,
\begin{eqnarray*}
|f_{v,k}(x)- f_{v,k}(y)| & \le & \sum_{w \not = v} |x(\{v,w\})- y(\{v,w\})| \le 2\|x-y\|_{v},\\
|f_{v,\le k}(x)- f_{v,\le k}(y)| & \le & \sum_{w \not = v} |x(\{v,w\})- y(\{v,w\})|\le 2\|x-y\|_{v}.
\end{eqnarray*}
Hence, for $x_0, y_0 \in S_0$, for $n \ge \max(1000,1/\lambda,c)$, for $t \le cn^2$ and each $v,k$,
\begin{eqnarray}
|(P^tf_{v,k})(x_0)- (P^t f_{v,k})(y_0)| & \le & \E |f_{v,k}(\widehat{X}_t) - f_{v,k} (\widehat{Y}_t)|
\le  2\E \|\widehat{X}_t - \widehat{Y}_t\|_{v} \nonumber\\
 & \le & 4\Big ( 1 + \frac{12d}{{n \choose 2}} \Big )^t \|x_0-y_0\|_{v} \nonumber \\
&& \mbox{} + \frac{64d^2t}{n{n \choose 2}} \Big (1 + \frac{12d}{{n \choose 2}} \Big )^{t-1} \|x_0-y_0\|_1. \label{eq.dist-3}
\end{eqnarray}
Similarly, for $x_0, y_0 \in S_0$, $n \ge \max(1000,1/\lambda,c)$, $t \le cn^2$ and each $v,k$,
\begin{eqnarray*}
|(P^tf_{v,\le k})(x_0)- (P^t f_{v,\le k})(y_0)| \!\!& \le & 4\Big ( 1 + \frac{12d}{{n \choose 2}} \Big )^t \|x_0-y_0\|_{v} 
\\
&&\mbox{} + \frac{64d^2t}{n{n \choose 2}} \Big (1 + \frac{12d}{{n \choose 2}} \Big )^{t-1} \|x_0-y_0\|_1. 
\end{eqnarray*}


Given $u,v \in V_n$ and $j,k \in \{0,1, \ldots, C\}$, set
\begin{eqnarray*}
f_{u,v,j,k} &=& \frac{1}{(n-2)^{2d-1}}\sum_{r=1}^d \sum_{w_r} \I_{uw_r}^j \sum_{w'_r}\I_{vw'_r}^{k} \prod_{s=1}^{r-1} \sum_{w_s,w'_s}(1-\I_{uw_s}^{\le j}\I_{vw'_s}^{\le j}) \\
&&\mbox{} \times \prod_{s=r+1}^d \sum_{w_s,w'_s}(1-\I_{uw_s}^{\le j-1}\I_{vw'_s}^{\le j-1}); \\
f_{u,v,\le j,k} &=& \frac{1}{(n-2)^{2d-1}}\sum_{r=1}^d \sum_{w_r} \I_{uw_r}^{\le j} \sum_{w'_r}\I_{vw'_r}^{k} \prod_{s=1}^{r-1} \sum_{w_s,w'_s}(1-\I_{uw_s}^{\le j}\I_{vw'_s}^{\le j}) \\
&&\mbox{} \times \prod_{s=r+1}^d \sum_{w_s,w'_s}(1-\I_{uw_s}^{\le j-1}\I_{vw'_s}^{\le j-1}),
\end{eqnarray*}
where the sums are over all $w_r,w'_r,w_s,w'_s \not = u,v$. Then
\begin{eqnarray*}
f_{u,v,j,k} & = & \frac{1}{(n-2)^{2d-1}} (f_{u,j} - \I_{uv}^j) (f_{v,k} - \I_{uv}^k) \\
&&\mbox{} \times \sum_{r=1}^d\Big ((n-2)^2 - (f_{u,\le j} - \I_{uv}^{\le j})(f_{v,\le j} - \I_{uv}^{\le j})\Big )^{r-1}\\
&&\mbox{} \times \Big ((n-2)^2 - (f_{u,\le j-1} - \I_{uv}^{\le j-1})(f_{v,\le j-1} - \I_{uv}^{\le j-1})\Big )^{d-r}, \\
f_{u,v,\le j,k} & = & \frac{1}{(n-2)^{2d-1}} (f_{u,\le j} - \I_{uv}^{\le j}) (f_{v,k} - \I_{uv}^k) \\
&&\mbox{} \times \sum_{r=1}^d \Big ((n-2)^2 - (f_{u,\le j} - \I_{uv}^{\le j})(f_{v,\le j} - \I_{uv}^{\le j})\Big )^{r-1}\\
&&\mbox{} \times \Big ((n-2)^2 - (f_{u,\le j-1} - \I_{uv}^{\le j-1})(f_{v,\le j-1} - \I_{uv}^{\le j-1})\Big )^{d-r}.
\end{eqnarray*}

In the case $d=1$, we have
\begin{eqnarray*}
f_{u,v,j,k} (x) &=& \frac{1}{n-2}\sum_{w \not = u,v} \I_{uw}^{j}(x) \sum_{w' \not = u,v}\I_{vw'}^k(x), \\
f_{u,v,\le j, k} (x) &=& \frac{1}{n-2}\sum_{w \not = u,v} \I_{uw}^{\le j}(x) \sum_{w' \not = u,v}\I_{vw'}^{k}(x),
\end{eqnarray*}
so that
\begin{eqnarray*}
f_{u,v,j,k} &=& \frac{1}{n-2}(f_{u,j} - \I_{uv}^j) (f_{v,k} - \I_{uv}^k), \\
f_{u,v,\le j,k} &=& \frac{1}{n-2}(f_{u,\le j} - \I_{uv}^{\le j}) (f_{v,k} - \I_{uv}^k).
\end{eqnarray*}
Then
\begin{eqnarray*}
\lefteqn{|f_{u,v,j,k}(x) - f_{u,v,j,k}(y)|} \\
& \le & \frac{1}{n-2}(f_{v,k}(x) - \I_{uv}^k(x))|f_{u,j}(x) - \I_{uv}^{j}(x) - (f_{u,j}(y) - \I_{uv}^{j}(y))|\\
&&\mbox{} +  \frac{1}{n-2}(f_{u,j}(y) - \I_{uv}^j(y))|f_{v,k}(x) - \I_{uv}^{k}(x) - (f_{v,k}(y) - \I_{uv}^{k}(y))|\\
& \le & 2 \|x-y\|_u + 2 \|x-y\|_v.
\end{eqnarray*}


Similarly,
\begin{eqnarray*}
\big| f_{u,v,\le j, k} (x) - f_{u,v,\le j,k} (x) \big| \le 2\|x-y\|_u + 2\|x-y\|_v.
\end{eqnarray*}



A calculation similar to the one above shows that, for any $d \ge 1$, if $f$ is one of the functions $f_{u,v,j,k}$,
$f_{u,v,\le j,k}$, then
$$
|f(x) - f(y)| \le 2d^2 (\|x-y\|_u + \|x-y\|_v).
$$
and so,
for all $x_0,y_0 \in S_0$, all
$n \ge \max(1000,1/\lambda,c)$, and all $t \le cn^2$,
\begin{eqnarray}
\lefteqn{|(P^tf)(x_0)- (P^t f)(y_0)|} \nonumber \\
& \le & 4d^2\Big ( 1 + \frac{12d}{{n \choose 2}} \Big )^t \|x_0-y_0\|_{v} + 4d^2\Big ( 1 + \frac{12d}{{n \choose 2}} \Big )^t \|x_0-y_0\|_{u} \nonumber \\
&&\mbox{} + \frac{128d^4t}{n{n \choose 2}} \Big (1 + \frac{12d}{{n \choose 2}} \Big )^{t-1}\|x_0-y_0\|_1. \label{eq.dist-2}
\end{eqnarray}


\section{Concentration of measure for the routing model}

\label{sec:conc-route}

We will now apply Theorem~\ref{thm.concb-general} to the jump Markov chain $\widehat{X}$ and functions $f_{v,k}$, $f_{v,\le k}$, $f_{u,v,j,k}$, $f_{u,v,\le j,k}$.

From now on, we assume that our process starts in some fixed state $X_0 = x_0 \in S_1$.
We write $\P$ and $\E$ when discussing probabilities relating to $\widehat{X}$, instead of $\P_{x_0}$ and $\E_{x_0}$, which was convenient in the derivation of the concentration inequalities in Section~\ref{S:conc}.

We start with the functions $f_{v,k}$. By~(\ref{eq.dist-3}), for all $x,y \in S_0$, we can take
$$
a_{x,i}(y)
= 4\Big ( 1 + \frac{12d}{{n \choose 2}} \Big )^i \|x-y\|_{v}
 +  \frac{64d^2i}{n{n \choose 2}} \Big (1 + \frac{12d}{{n \choose 2}} \Big )^{i-1} \|x-y\|_1,
$$
for $i \le cn^2$ and $n \ge \max(1000,1/\lambda,c)$.

The key is that, for any $x \in S_0$ and any $i \ge 0$, if $y \in N(x)$ is chosen with probability $P(x,y)$, then it is very likely that
$\|x-y\|_v = 0$, and thus $a_{x,i}(y)$ is relatively small.  This enables us to use the full power of Theorem~\ref{thm.concb-general}.
Indeed, for each $x \in S_0$, we have
\begin{equation} \label{eq.change-v}
\sum_{y: \|x-y\|_{v} > 0} P(x,y) \le p(\lambda,n) \frac{2+d}{n} + (1- p(\lambda,n)) \frac{C(n-1)}{\lfloor 6 \lambda \binom{n}{2} \rfloor}
\le \frac{1}{6n} \left( 2 + d + 2C/\lambda \right),
\end{equation}
with $p(\lambda,n)$ as in (\ref{eq.plambdan}).
To see this, note that, conditional on the next jump being an arrival, the probability that the load on some link at $v$ is altered is at most $(2+d)/n$.
Also, as the total number of calls involving node $v$ is at most $C(n-1)$, conditional on the jump being a potential departure, the probability that the departure is from a link at $v$ is at most $C(n-1)/ \lfloor 6 \lambda \binom{n}{2} \rfloor$.

Note also that, for all $x,y$ such that $y \in N(x)$, we have $\|x-y\|_{v} \le 1$ and $\|x-y\|_1 \le 1$.
It follows that for each $x \in S_0$, for $i \le c n^2$ and $n \ge \max(1000,1/\lambda,c)$,
\begin{eqnarray*}
\big(Pa_{x,i}^2\big) (x) & \le & \frac{32(2+d+2C/\lambda)}{6n}\Big (1 + \frac{12d}{{n \choose 2}} \Big )^{2i} + 2^{13}\Big (\frac{d^2i}{n{n \choose 2}}\Big )^2 \Big (1 + \frac{12d}{{n \choose 2}} \Big )^{2i-2}\\
& \le & \frac{32(2+d+2C/\lambda)}{6n}e^{96dc} + 2^{16}\Big (\frac{d^2c}{n}\Big )^2e^{96dc}\\
&\le& \frac{2^{16}(d^4+C/\lambda)(c + 1)^2}{n}e^{96dc}.
\end{eqnarray*}
So we can take $\alpha_i^2 = \frac{2^{16}(d^4+ C/\lambda)(c + 1)^2}{n}e^{96dc}$, and thus
$\beta \le 2^{17}(d^4+C/\lambda)(c + 1)^3ne^{96dc}$, for $t \le cn^2$.
Also we can take
\begin{eqnarray*}
\alpha = 4e^{48dc} + \frac{256d^2c}{n}e^{48dc} \le 8 e^{48dc}
\end{eqnarray*}
for $n \ge \max(1000,1/\lambda,64d^2c)$ and $t \le cn^2$.

For $t >0$, let $A_t$ be the event that $\widehat{X}_s \in S_0^0$ for $0 \le s \le t-1$.
By Theorem~\ref{thm.concb-general}, 
for $t \le cn^2$ and $a \le n$,
\begin{eqnarray*}
\lefteqn{\P \Big ( \{|f_{v,k} (\widehat{X}_t) - \E [f_{v,k}(\widehat{X}_t)]| \ge a\} \cap A_t \Big)}
\\
&\le& 2e^{-a^2/(2^{18}(d^4+C/\lambda)(c+1)^3ne^{96dc} + 32e^{48dc}a/3)} \\
&\le& 2e^{-a^2/2^{19}(d^4+C/\lambda)(c+1)^3ne^{96dc}}. 
\end{eqnarray*}
Similarly, for $t \le cn^2$ and $a \le n$,
\begin{equation} \label{eq.discrete-conc}
\P \Big ( \{|f_{v,\le k} (\widehat{X}_t) - \E [f_{v,\le k}(\widehat{X}_t)]| \ge a\} \cap A_t \Big)
\le 2e^{-a^2/2^{19}(d^4+C/\lambda)(c+1)^3ne^{96dc}}. 
\end{equation}


We now consider functions $f_{u,v,j,k}$ and $f_{u,v,\le j,k}$.  By~(\ref{eq.dist-2}), for all $x,y \in S_0$,
\begin{eqnarray*}
a_{x,i}(y)
& = & 4d^2\Big ( 1 + \frac{12d}{{n \choose 2}} \Big )^i \|x-y\|_{v} +
4d^2\Big ( 1 + \frac{12d}{{n \choose 2}} \Big )^i \|x-y\|_{u}\\
&& \mbox{} + \frac{128d^4i}{n{n \choose 2}} \Big (1 + \frac{12d}{{n \choose 2}} \Big )^{i-1}\|x-y\|_1,
\end{eqnarray*}
for $i \le cn^2$ and $n \ge \max(1000,1/\lambda,c)$. This leads to
\begin{eqnarray*}
\big(Pa_{x,i}^2\big) (x) & \le & \frac{64d^4(d+2+2C/\lambda)}{6n}\Big (1 + \frac{12d}{{n \choose 2}} \Big )^{2i} +
2^{15}\Big (\frac{d^4i}{n{n \choose 2}}\Big )^2 \Big (1 + \frac{12d}{{n \choose 2}} \Big )^{2i-2}\\
& \le & \frac{64d^4(d+2+2C/\lambda)}{6n}e^{96dc} + 2^{18}\Big (\frac{d^4c}{n}\Big )^2e^{96dc} \\
&\le& \frac{2^{18}(d^8+d^4C/\lambda)(c+1)^2}{n}e^{96dc},
\end{eqnarray*}
for each $x \in S_0$, for $t \le cn^2$ and $n \ge \max(1000, 1/\lambda, 64d^2c)$.

So we can take $\alpha_i^2 = \frac{2^{18}(d^8+d^4C/\lambda)(c+1)^2}{n}e^{96dc}$, and so
$\beta \le 2^{19}(d^8+d^4C/\lambda)(c+1)^3ne^{96dc}$. Also, for $n \ge \max(1000,1/\lambda,64d^2c)$, we can take
\begin{eqnarray*}
\alpha = 8d^2e^{48dc} + \frac{512d^4c}{n}e^{48dc} \le 16d^2 e^{48dc}.
\end{eqnarray*}

By Theorem~\ref{thm.concb-general}, 
for $t \le cn^2$ and $a \le n$,
\begin{eqnarray*}
\lefteqn{\P \Big ( \{|f_{u,v,j,k} (\widehat{X}_t) - \E [f_{u,v,j,k}(\widehat{X}_t)]| \ge a\} \cap A_t \Big)}
\\
&\le& 2e^{-a^2/(2^{20}(d^8+d^4C/\lambda)(c+1)^3ne^{96dc} + 64d^2e^{48dc}a/3)} \\
&\le& 2e^{-a^2/2^{21}(d^8+d^4C/\lambda)(c+1)^3ne^{96dc}}, 
\end{eqnarray*}
and, similarly,
\begin{eqnarray*}
\lefteqn{\P \Big ( \{|f_{u,v,\le j,k} (\widehat{X}_t) - \E [f_{u,v,\le j, k}(\widehat{X}_t)]| \ge a\} \cap A_t \Big)} \\
&\le& 2e^{-a^2/2^{21}(d^8+d^4C/\lambda)(c+1)^3ne^{96dc}}. 
\end{eqnarray*}

To relate the continuous-time process $X$ and the discrete-time chain $\widehat{X}$,
note that, while $X$ remains in $S_0^0$, departures in $X$ can be represented by a Poisson
process of potential departures with rate $\lfloor 6 \lambda {n \choose 2} \rfloor$,
together with a process of `choices' defined as in the description of the transitions of $\widehat X$.
For this representation, the number $Z_t$ of events (arrivals and potential departures) in $X$ during the
interval $[0,t]$ is Poisson with mean $rt$, where
$$
r = \lambda {n \choose 2} + \lfloor 6 \lambda {n \choose 2} \rfloor
\le 7\lambda {n \choose 2},
$$
and the events correspond precisely to the jumps of $\widehat{X}$.

As in~\cite{lm06}, we choose a suitable ``width'' $w$, and consider the interval $I$ of values
$z \in \Z^+$ such that $|z-rt| \le w$.
Since $Z_t$ is Poisson with mean $rt$, we have
$\P (Z_t \notin I) \le 2e^{-w^2/3rt}$.
We shall take $w \ge 2 \sqrt{ rt} \log n$, so that
$\P(Z_t\notin I) \le e^{-\log^2 n}$.
On the event that $\widehat{X}_z \in S_0$, we have from (\ref{eq.change-v}) that
$$
\E ( |f_{v,k}(\widehat{X}_{z+1}) - f_{v,k}(\widehat{X}_z)| \mid \widehat{\cF}_z ) \le 2 \frac{2+d+2C/\lambda}{6n},
$$
since $f_{v,k}$ can only change if $\| \widehat{X}_{z+1} - \widehat{X}_z \|_v > 0$.  Since $\widehat{X}_z$ stops as soon as it
leaves $S_0$, this inequality also holds on the event that $\widehat{X}_z \notin S_0$.
Let $\mu(z) = \E f_{v,k}(\widehat{X}_z)$.  By the above,
$$
\Big| \mu(z+1) - \mu(z) \Big| = \Big| \E \left[ \E ( f_{v,k}(\widehat{X}_{z+1}) - f_{v,k}(\widehat{X}_z) \mid \widehat{\cF}_z ) \right] \Big|
\le \frac{2+d+2C/\lambda}{3n}.
$$
So, for $z \in I$, $|\mu(z) - \mu(\lfloor rt \rfloor)| \le w(2+d+2C/\lambda)/3n$.

By Lemma~2.5 in~\cite{lm08}, if $n \ge \max(1000,1/\lambda,64d^2c)$, then for each $x_0 \in S_1$,
\begin{equation} \label{eq:leaveS_0}
\P (\overline {A_{cn^2}}) = \P (\exists \mbox{ } t \le cn^2: \widehat{X}_t \not \in S_0^0) \le cn^2 (7/12)^{2\lambda \binom{n}{2}-1} 
\le cn^2 e^{-n/4} \le e^{-n/8}.
\end{equation}
Thus, for any $t$, and $z \le cn^2$,
$|\E (f_{v,k}(X_t) \mid Z_t = z) - \mu(z)| \le n \P (\overline{A_z}) \le n e^{-n/8}$ by (\ref{eq:leaveS_0}).
Also,
$$
\E f_{v,k}(X_t) = \sum_{z \in \Z^+} \E (f_{v,k}(X_t) \mid Z_t = z) \P (Z_t = z),
$$
so, if $rt + w \le cn^2$,
\begin{eqnarray*}
\lefteqn{ \Big| \E f_{v,k}(X_t) - \mu(\lfloor rt \rfloor) \Big| } \\
&\le& \sum_{z\in I} \P(Z_t = z) \Big| \E (f_{v,k}(X_t) \mid Z_t = z) - \mu(\lfloor rt \rfloor) \Big| + n \P (Z_t \notin I) \\
&\le& w \frac{2+d+C/\lambda}{3n} + n e^{-n/8} + n e^{-\log^2 n}
\le w \frac{2+d+C/\lambda}{2n}.
\end{eqnarray*}
Moreover, provided $n \ge \max( 1000, 1/\lambda, 64d^2c)$, $w(2+d+2C/\lambda)/2n \le a \le n$ and $rt + w \le cn^2$, we have
\begin{eqnarray*}
\lefteqn{ \P\Big( \Big| f_{v,k}(X_t) - \E f_{v,k}(X_t) \Big| \ge 3a \Big) } \\
&\le& \P\Big( \Big| f_{v,k}(X_t) - \mu(\lfloor rt \rfloor) \Big| \ge 2a \Big)  \\
&\le& \sum_{z \in I} \P (Z_t = z) \P \Big(\Big| f_{v,k}(X_t) - \mu(\lfloor rt \rfloor) \big| \ge 2a \mid Z_t = z \Big) + \P (Z_t \notin I) \\
&\le& \sum_{z\in I} \P (Z_t = z) \P \Big( \Big\{ \Big| f_{v,k}(\widehat{X}_z) - \mu(z) \big| \ge a \Big\} \cap A_z \Big) + \P (\overline{A_{rt+w}}) +
\P (Z_t \notin I) \\
&\le& 2e^{-a^2/2^{19}(d^4+C/\lambda)(c+1)^3ne^{96dc}} + e^{-n/8} + 2e^{-w^2/3rt}.
\end{eqnarray*}
Taking $a = \frac 13 \sqrt n \log n$ and $w = 2\sqrt t n \log n$, we see that, if $t \le t_0 = c/8\lambda$ and
$n \ge \max( 1000, 512d^2\lambda t_0, 72t_0C^2/\lambda^2, 1/\lambda^2t_0 )$,
then
\begin{eqnarray}
\P \Big( \Big| f_{v,k}(X_t) - \E f_{v,k}(X_t) \Big| \ge \sqrt n \log n  \Big) &\le&
4 e^{-\log^2 n /2^{23}(d^4 + C/\lambda)(8\lambda t_0+1)^3e^{800d\lambda t_0}} \nonumber \\
&\le& 4e^{-\gamma \log^2 n},\label{eq.conc-f10}
\end{eqnarray}
where
$$
\gamma = \gamma(\lambda, d, C, t_0) = \frac{1}{2^{25}(d^8+d^4C/\lambda)(8\lambda t_0+1)^3e^{800d\lambda t_0}}.
$$

Similarly, under the same conditions, we have, for each $u,v$ and $j,k$,
%
\begin{eqnarray}
\P \Big ( |f_{v,\le k} (X_t) - \E [f_{v,\le k}(X_t)]| \ge \sqrt{n} \log n\Big )
\!& \le & \!4e^{-\gamma \log^2 n}, \label{eq.conc-f10a}\\
\P \Big ( |f_{u,v,j,k} (X_t) - \E [f_{u,v,j,k}(X_t)]| \ge \sqrt{n} \log n \Big )
\!& \le & \!4e^{-\gamma\log^2 n}, \label{eq.conc-f11}\\
\P \Big ( |f_{u,v,\le j,k} (X_t) - \E [f_{u,v,\le j, k}(X_t)]| \ge \sqrt{n} \log n \Big )
\!& \le & \!4e^{-\gamma \log^2 n}. \label{eq.conc-f11a}
\end{eqnarray}

As $t_0=c/8\lambda$, the mean number of events in $[0,t_0]$ is $rt_0 \le \frac{1}{2}cn^2$.
and the probability that there are more than $cn^2$ events in the interval $[0,t_0]$ is
at most $e^{-cn^2/6}$. Therefore
\begin{eqnarray} \label{eq.conc-g2}
\P \Big(\exists v \in V_n, t \le t_0, k \in \{0,\dots, C\}, |f_{v,k} (X_t) - \E [f_{v,k}(X_t)]| \ge \sqrt{n} \log n \Big) \nonumber \\
\le 4cCn^3 e^{-\gamma \log^2 n} + e^{-4t_0\lambda n^2/3}
\le 40t_0 \lambda C n^3 e^{-\gamma \log^2 n}. \qquad \phantom{p}
\end{eqnarray}


\section{Expectation of the generator}

\label{sec:gen-expec}

As before, we assume that our process starts in some fixed state $x_0 \in S_1$,
and we consider the law of the process started in this state, and running until some time $t_0 > 0$.
Recall that
$$
n_0(\lambda,d,C,t_0)= \max \Big(2^{18}(\lambda+1/\lambda)^4 d^4 (C+1)^6 (t_0+1/t_0)^2,
e^{8/\gamma(\lambda,d,C,t_0)} \Big),
$$
as in the statement of Theorems~\ref{thm.main-result} and~\ref{thm.main-result-d=1}.
Note that $n_0(\lambda,d,C,t_0) \ge 1000$, for any positive integers $d$ and $C$ and positive reals
$\lambda$ and $t_0$, and that the bounds of the previous section hold for all $n \ge n_0(\lambda,d,C,t_0)$.
We also have $e^{-\gamma \log^2 n} \le n^{-8}$, an inequality we shall use freely from now on.

Recall the definitions of $\phi^1$, $\phi^2$ and $\phi^3$ from (\ref{eq:phi1})--(\ref{eq:phi2}), and that
$\phi = \max \{\phi^1, \phi^2, \phi^3\}$.  Set $\widetilde{\phi} = \max \{ \phi^1,\phi^2\}$.
Recall also that
\begin{eqnarray*}
g_j (\xi) & = & 2 \xi(C) \xi(j) \xi(\le j) \sum_{r=1}^d (1-\xi(\le j)^2)^{r-1}
(1-\xi(\le j-1)^2)^{d-r} \\
&& \mbox{} + 2 \xi(C)  \xi(j) \sum_{i=j+1}^{C-1} \xi(i) 
\sum_{r=1}^d(1-\xi(\le i)^2 )^{r-1}
(1-\xi(\le i-1)^2)^{d-r}.
\end{eqnarray*}
Our first aim in this section is to show that, provided $\E \phi(X_t)$ is small,
$\E [g_{v,j}(X_t)]$ is close to $(n-1) g_j(\zeta_t^v)$, where $g_{v,j}$ is as in (\ref{eq-g}) and $\zeta_t^v$ is the vector with
components $\zeta_t(v,j) = (n-1)^{-1} \E [f_{v,j}(X_t)]$, for $j \in \{ 0, \dots, C\}$.  We then go
on to show that, if $\phi(x_0)$ is small, then also $\E \phi(X_t)$ is small for all $t \le t_0$.


\begin{lemma} \label{lem.gen-expec}
For all $t \le t_0$, for each $v \in V_n$ and each $j \in \{0, \ldots, C\}$,
\begin{eqnarray*}
\lefteqn{|\E [g_{v,j} (X_t)] - (n-1)g_j (\zeta_t^v)|} \\
& \le & 31 d^2 (C+1)^3 n \Big (\phi (x_0) + \frac{3\log n}{\sqrt n} \Big )
e^{208(\lambda+1)d^2 (C+1)^3 t_0},
\end{eqnarray*}
provided $n \ge n_0(\lambda,d,C,t_0)$.

If $d=1$ and $n \ge n_0(\lambda,1,C,t_0)$, we have the improved bound
\begin{eqnarray*}
\lefteqn{|\E [g_{v,j} (X_t)] - (n-1)g_j (\zeta_t^v)|} \\
& \le & 31 (C+1)^3 n \Big (\phi (x_0) + \frac{3\log n}{\sqrt n} \Big )
e^{208(\lambda+1) (C+1) t_0}.
\end{eqnarray*}
\end{lemma}

The lemma above will follow immediately from two other lemmas, the first of which is as follows.

\begin{lemma} \label{lem.gen-exp}
For any $n \ge n_0(\lambda,d,C,t_0)$, $v \in V_n$ and $j \in \{0, \dots, C\}$,
$$
\Big | \E [g_{v,j}(X_t)] - (n-1)g_j (\zeta_t^v) \Big | \le 12d^2 (C+1)^3 n
\E [\widetilde{\phi} (X_t)] + 20 d^2 (C+1) \sqrt{n} \log n.
$$
\end{lemma}


\begin{proof}
Suppose that $n \ge n_0(\lambda,d,C,t_0)$.

The function $g_{v,j}$ is a sum of four terms, which we separate out. Let
\begin{eqnarray*}
P_{v,j}^+ = \frac{1}{(n-2)^d}\sum_{r=1}^d\sum_{u} \I_{uv}^{C}\sum_{w_r}\I_{vw_r}^{j} \I_{uw_r}^{\le j} \prod_{s=1}^{r-1} \sum_{w_s}(1-\I_{uv,w_s}^{\le j})
\prod_{s=r+1}^d \sum_{w_s}(1-\I_{uv,w_s}^{\le j-1});
\end{eqnarray*}
\begin{eqnarray*}
P_{v,j}^- &=& \frac{1}{(n-2)^d}\sum_{r=1}^d\sum_{u}\I_{uv}^{C} \sum_{w_r}\I_{vw_r}^{j} \sum_{i=j+1}^{C-1}\I_{uw_r}^{i} \prod_{s=1}^{r-1} \sum_{w_s}(1-\I_{uv,w_s}^{\le i}) \\
&&\mbox{} \times \prod_{s=r+1}^d \sum_{w_s}(1-\I_{uv,w_s}^{\le i-1}).
\end{eqnarray*}
In both expressions above, the first sum is over all values of $u \not = v$, and the subsequent sums are over all values of $w_r$ or $w_s \not = u,v$.
Let further
\begin{eqnarray*}
Q_{v,j}^+ = \frac{1}{(n-2)^d}\sum_{r=1}^d\sum_u \I_{uv}^j \sum_{v'}\I_{uv'}^{C} \I_{v'v}^{\le j}
\prod_{s=1}^{r-1} \sum_{w_s} (1-\I_{uv',w_s}^{\le j})
\prod_{s=r+1}^d \sum_{w_s} (1-\I_{uv',w_s}^{\le j-1});
\end{eqnarray*}
\begin{eqnarray*}
Q_{v,j}^- &=& \frac{1}{(n-2)^d}\sum_{r=1}^d\sum_u \I_{uv}^j \sum_{v'} \I_{uv'}^{C}
\sum_{i=j+1}^{C-1}\I_{v'v}^{i}
\prod_{s=1}^{r-1} \sum_{w_s} (1-\I_{uv',w_s}^{\le i}) \\
&&\mbox{} \times \prod_{s=r+1}^d \sum_{w_s} (1-\I_{uv',w_s}^{\le i-1}).
\end{eqnarray*}
In the last two expressions above, the first sum is over all values of $u \not = v$, the second sum is over all values of $v' \not = u,v$, and the subsequent sums are over all
$w_s \not = u,v'$.
Then
$g_{v,j} = P_{v,j}^+ + P_{v,j}^- + Q_{v,j}^+ + Q_{v,j}^-$.

We define further `standardised' versions of $P_{v,j}^+$, $P_{v,j}^-$, $Q_{v,j}^+$, $Q_{v,j}^-$. Let
\begin{eqnarray*}
\widehat{P}_{v,j}^+ &=& \frac{1}{(n-2)^{2d}}\sum_{r=1}^d\sum_{u} \I_{uv}^{C}\sum_{w_r}\I_{vw_r}^{j} \sum_{w'_r}\I_{uw'_r}^{\le j} \\
&&\mbox{} \times \prod_{s=1}^{r-1} \sum_{w_s,w'_s}(1-\I_{uw_s}^{ \le j} \I_{vw'_s}^{\le j})
\prod_{s=r+1}^d \sum_{w_s,w'_s}(1-\I_{uw_s}^{\le j-1}\I_{vw'_s}^{\le j-1});
\end{eqnarray*}
\begin{eqnarray*}
\widehat{P}_{v,j}^- &=& \frac{1}{(n-2)^{2d}}\sum_{r=1}^d\sum_{u}\I_{uv}^{C} \sum_{w_r}\I_{vw_r}^{j} \sum_{i=j+1}^{C-1}\sum_{w'_r}\I_{uw'_r}^{i}\\
&&\mbox{} \times \prod_{s=1}^{r-1} \sum_{w_s,w'_s}(1-\I_{uw_s}^{\le i} \I_{vw'_s}^{\le i})
\prod_{s=r+1}^d \sum_{w_s,w'_s}(1-\I_{uw_s}^{\le i-1}\I_{vw'_s}^{\le i-1}).
\end{eqnarray*}
In these two expressions, the first sum is over all values of $u \not = v$, and the remaining sums are over all values of $w_r, w'_r, w_s, w'_s \not = u,v$.
Also, let
\begin{eqnarray*}
\widehat{Q}_{v,j}^+ &=& \frac{1}{(n-2)^{2d}}\sum_{r=1}^d\sum_{u}\I_{uv}^{j}
\sum_{v'} \I_{uv'}^{C}\sum_{v''} \I_{v''v}^{\le j}\\
&&\mbox{} \times \prod_{s=1}^{r-1} \sum_{w_s,w'_s}(1-\I_{uw_s}^{\le j}\I_{vw'_s}^{\le j})
\prod_{s=r+1}^d \sum_{w_s,w'_s}(1-\I_{uw_s}^{\le j-1}\I_{vw'_s}^{\le j-1});
\end{eqnarray*}
\begin{eqnarray*}
\widehat{Q}_{v,j}^- &=& \frac{1}{(n-2)^{2d}}\sum_{r=1}^d\sum_{u} \I_{uv}^{j}\sum_{v'} \I_{uv'}^{C}
\sum_{i=j+1}^{C-1}\sum_{v''}\I_{v''v}^{i} \\
&&\mbox{} \times \prod_{s=1}^{r-1} \sum_{w_s,w'_s} (1-\I_{uw_s}^{\le i}\I_{vw'_s}^{\le i})
\prod_{s=r+1}^d \sum_{w_s,w'_s}(1-\I_{uw_s}^{\le i-1} \I_{vw'_s}^{\le i-1}).
\end{eqnarray*}
In these two expressions, the first sum is over all values of $u \not = v$, and the subsequent sums are over all values of $v', v'', w_s, w'_s \not = u,v$.
In these standardised versions, the ``anchor'' nodes $u$ and $v$ for the final products are chosen so that the products can be extracted as far as possible as a common factor.
In the future, we will always use similar conventions regarding the ranges of the various
summations involved, and the choices of anchor nodes.

Set
$\widehat{g}_{v,j} = \widehat{P}^+_{v,j} + \widehat{P}^-_{v,j} + \widehat{Q}^+_{v,j} + \widehat{Q}^-_{v,j}$.
We shall now bound $|g_{v,j} - \widehat{g}_{v,j}|$ above via upper
bounds on the differences $|P_{v,j}^+ - \widehat{P}_{v,j}^+|$,
$|P_{v,j}^- - \widehat{P}_{v,j}^-|$, $|Q_{v,j}^+ - \widehat{Q}_{v,j}^+|$
and $|Q_{v,j}^- - \widehat{Q}_{v,j}^-|$.  These differences denote the maximum
difference of the functions over all load vectors $x$.

Noting that, for $0 \le j \le C$ and for any $u$,
\begin{eqnarray*}
\lefteqn{\Big |\frac{1}{n-2} \sum_{w_s} \Big (1 - \I_{uw_s}^{\le j} \I_{vw_s}^{\le j} \Big ) -
\frac{1}{(n-2)^2} \sum_{w_s,w'_s} \Big (1 - \I_{uw_s}^{\le j} \I_{vw'_s}^{\le j} \Big ) \Big |}\\
&=& \Big |\frac{1}{n-2} \sum_{w_s} \I_{uw_s}^{\le j} \I_{vw_s}^{\le j} -
\frac{1}{(n-2)^2} \sum_{w_s,w'_s} \I_{uw_s}^{\le j} \I_{vw'_s}^{\le j} \Big |
\le (C+1)^2\phi^1,
\end{eqnarray*}
we see that
$|P_{v,j}^+ - \widehat{P}_{v,j}^+| \le d^2 (C+1)^2 (n-2) \phi^1$.
Similarly,
$|P_{v,j}^- - \widehat{P}_{v,j}^-| \le d^2 (C+1)^3 (n-2) \phi^1$.
Also,
\begin{eqnarray*}
\lefteqn{|Q_{v,j}^+ - \widehat{Q}_{v,j}^+|}\\
&\le& d(d-1) [(C+1)^2 (n-2) \phi^1 + (C+1) (n-2) \phi^2]\\
&&\mbox{} + \frac{1}{(n-2)^{2d-1}}\sum_{r=1}^d \sum_{u}\I_{uv}^{j} \Big |\sum_{v'} \I_{uv'}^{C}\Big (\I_{v'v}^{\le j} -\frac{1}{n-2} \sum_{v''}\I_{v''v}^{\le j} \Big )\\
&&\mbox{} \times \prod_{s=1}^{r-1} \sum_{w_s,w'_s}(1-\I_{uw_s}^{\le j}\I_{vw'_s}^{\le j}) \prod_{s=r+1}^d \sum_{w_s,w'_s}(1-\I_{uw_s}^{\le j-1}\I_{vw'_s}^{\le j-1}) \Big |\\
&\le& d(d-1) (C+1)^2 (n-2) (\phi^1 + \phi^2) + \frac{1}{(n-2)^{2d-2}}\sum_{r=1}^d\sum_{u}\I_{uv}^{j} \\
&&\mbox{} \times \prod_{s=1}^{r-1} \sum_{w_s,w'_s}(1-\I_{uw_s}^{\le j}\I_{vw'_s}^{\le j}) \prod_{s=r+1}^d \sum_{w_s,w'_s}(1-\I_{uw_s}^{\le j-1}\I_{vw'_s}^{\le j-1})\\
&&\mbox{} \times \Big |\frac{1}{n-2}\sum_{v'} \I_{uv'}^{C}\I_{v'v}^{\le j} -\frac{1}{(n-2)^2}\sum_{v'}\I_{uv'}^{C} \sum_{v''}\I_{v''v}^{\le j} \Big |.
\end{eqnarray*}
Hence
\begin{eqnarray*}
|Q_{v,j}^+ - \widehat{Q}_{v,j}^+| & \le & d(d-1) (C+1)^2 (n-2) (\phi^1 + \phi^2) + d(C+1)(n-2) \phi^1\\
&\le & 2 d^2 (C+1)^2 (n-2) \widetilde{\phi}.
\end{eqnarray*}
Similarly,
$|Q_{v,j}^- - \widehat{Q}_{v,j}^-| \le 2d^2 (C+1)^3 (n-2) \widetilde{\phi}$.

It follows that
\begin{equation} \label{gghat}
|g_{v,j} - \widehat{g}_{v,j} | \le 6d^2(C+1)^3(n-2) \widetilde{\phi},
\end{equation}
and so
$|\E [g_{v,j}(X_t)] - \E [\widehat{g}_{v,j}(X_t)] |
\le 6d^2(C+1)^3(n-2) \E [\widetilde{\phi} (X_t)]$.
In the special case $d=1$, the estimates are easier, and we find that
\begin{equation} \label{gghat1}
|g_{v,j} - \widehat{g}_{v,j} | \le 4(C+1)(n-2) \widetilde{\phi}.
\end{equation}

We now bound
$| \E [\widehat{g}_{v,j}(X_t)] - (n\!-\!1)g_j (\zeta_t^v) |$,
with $(n\!-\!1)g_j(\zeta_t^v)$ given by
\begin{eqnarray*}
\frac{2\E[f_{v,C}]\E[f_{v,j}]\E[f_{v,\le j}]}{(n-1)^{d+1}}
\sum_{r=1}^d \left( (n\!-\!1) - \E[f_{v,\le j}]^2 \right)^{r-1}
\left( (n\!-\!1) - \E[f_{v,\le j-1}]^2 \right)^{d-r} \\
+ \frac{2\E[f_{v,C}]\E[f_{v,j}]}{(n-1)^{d+1}}  \sum_{i=j+1}^{C-1} \E[f_{v,\le i}]
\sum_{r=1}^d \left( (n\!-\!1) - \E[f_{v,\le i}]^2 \right)^{r-1} \\
\times \left( (n\!-\!1) - \E[f_{v,\le i-1}]^2 \right)^{d-r}.
\end{eqnarray*}
Here, and throughout what follows, we abuse notation by writing e.g.\ $\E[f_{v,C}]$
instead of $\E[f_{v,C}(X_t)]$: for the remainder of this proof, all of our functions
will be evaluated at $X_t$.

We start by estimating the difference between $\E \widehat{P}^+_{v,j}$ and
$$
\frac{\E[f_{v,C}]\E[f_{v,j}]\E[f_{v,\le j}]}{(n-1)^{d+1}}
\sum_{r=1}^d \left( (n\!-\!1) - \E[f_{v,\le j}]^2 \right)^{r-1}
\left( (n\!-\!1) - \E[f_{v,\le j-1}]^2 \right)^{d-r}:
$$
$|\E [\widehat{g}_{v,j}(X_t)] - (n\!-\!1)g_j (\zeta_t^v) |$ is the sum of
this and three similar terms.
Note that
$
\widehat{P}^+_{v,j} = \frac{1}{n-2} \sum_{u\not=v} \I_{u,v}^C f_{u,v,\le j,j}
$.
By~(\ref{eq.conc-f11a}), as $n \ge n_0(\lambda,d,C,t_0)$, 
\begin{eqnarray*}
f_{u,v,\le j,j} &=& \frac{1}{(n-2)^{2d-1}}\sum_{r=1}^d\sum_{w_r}\I_{vw_r}^{j}
\sum_{w'_r}\I_{uw'_r}^{\le j}\prod_{s=1}^{r-1}
\sum_{w_s,w'_s}(1-\I_{uw_s}^{ \le j} \I_{vw'_s}^{\le j}) \\
&&\mbox{} \prod_{s=r+1}^d \sum_{w_s,w'_s}(1-\I_{uw_s}^{\le j-1}\I_{vw'_s}^{\le j-1})
\end{eqnarray*}
satisfies, for each
$t \le t_0$,
\begin{eqnarray*}
\P \Big ( |f_{u,v,\le j,j} (X_t) - \E [f_{u,v,\le j, j}(X_t)]| \ge \sqrt{n} \log n \Big )
& \le & 4e^{-\gamma \log^2 n}.
\end{eqnarray*}
By Lemma~\ref{lem.expec}, as $n \ge n_0(\lambda,d,C,t_0)$,
\begin{eqnarray*}
\Big |\E [\I_{uv}^C f_{u,v,\le j,j} ] - \E [\I_{uv}^C] \E [f_{u,v,\le j,j} ] \Big | \le
\sqrt{n} \log n + 4ne^{-\gamma \log^2 n} \le \frac32\sqrt n \log n,
\end{eqnarray*}
for each $u \not=v$, and so
\begin{equation}
\Big |\E [\widehat{P}^+_{v,j} ] - \frac{1}{n-2}\sum_u  \E [\I_{uv}^C] \E [f_{u,v,\le j,j} ] \Big |
\le \frac{n-1}{n-2} \frac32 \sqrt{n} \log n \le 2 \sqrt n \log n  \label{hatP+}.
\end{equation}

Now let $E_t$ be the event that $|f_{v,j} (X_t) - \E f_{v,j}(X_t)| \le \sqrt{n} \log n$ and
$|f_{v,\le j} (X_t) - \E f_{v,\le j}(X_t)| \le \sqrt{n} \log n$ for all
$j \in \{0, \ldots, C\}$ and $v \in V_n$.
By~(\ref{eq.conc-f10}) and~(\ref{eq.conc-f10a}),
$\P(\overline{E_t}) \le 8 (C+1) n e^{-\gamma \log^2 n}$.
Note that, on $E_t$, 
\begin{eqnarray*}
\frac{1}{n-2} \Big |(f_{u,\le j} - \I_{uv}^{\le j}) (f_{v,\le j} - \I_{uv}^{\le j}) -
\E [f_{u, \le j}] \E [f_{v, \le j}] \Big | \le 3 \sqrt{n} \log n;
\end{eqnarray*}
\begin{eqnarray*}
\frac{1}{n-2} \Big |(f_{u,\le j} - \I_{uv}^{\le j}) (f_{v,j} - \I_{uv}^{j}) - \E [f_{u, \le j}]
\E [f_{v,j}] \Big | \le 3 \sqrt{n} \log n,
\end{eqnarray*}
for each $j$.
Thus, recalling that
\begin{eqnarray*}
f_{u,v,\le j,j} & = & \frac{1}{(n-2)^{2d-1}} (f_{u,\le j} - \I_{uv}^{\le j})(f_{v,j}
- \I_{uv}^{j}) \\
&&\mbox{} \times \sum_{r=1}^d \Big ((n-2)^2 -  (f_{u,\le j} - \I_{uv}^{\le j})(f_{v,\le j}
- \I_{uv}^{\le j})  \Big )^{r-1}\\
&&\mbox{} \times \Big ((n-2)^2 -  (f_{u,\le j-1} - \I_{uv}^{\le j-1})(f_{v,\le j-1}
- \I_{uv}^{\le j-1})  \Big )^{d-r},
\end{eqnarray*}
we see that, on $E_t$, the difference between $f_{u,v,\le j,j}$ and
\begin{eqnarray*}
\frac{1}{(n-2)^{2d-1}} \E [f_{u, \le j}] \E [f_{v,j}] \sum_{r=1}^d \Big ((n-2)^2 -
\E [f_{u,\le j}] \E [f_{v,\le j}] \Big )^{r-1}\\
\times \Big ((n-2)^2 -  \E [f_{u,\le j-1}] \E [f_{v,\le j-1}] \Big )^{d-r}
\end{eqnarray*}
is at most $3d^2 \sqrt{n} \log n$ in absolute value.

Thus, with probability at least $1- 16(C+1) n  e^{-\gamma \log^2 n}$, $f_{u,v,\le j,j}$ is
within distance $\sqrt{n} \log n$ of $\E [f_{u,v,\le j,j}]$ and within distance
$3 d^2 \sqrt{n} \log n$ of
\begin{eqnarray*}
H_{u,v,j}(t) = \frac{1}{(n-2)^{2d-1}} \E [f_{u, \le j}] \E [f_{v,j}] \sum_{r=1}^d \Big ((n-2)^2 -
\E [f_{u,\le j}] \E [f_{v,\le j}] \Big )^{r-1}\\
\times \Big ((n-2)^2 -  \E [f_{u,\le j-1}] \E [f_{v,\le j-1}] \Big )^{d-r}.
\end{eqnarray*}
Therefore the difference between $\E [f_{u,v,\le j,j}]$ and $H_{u,v,j}(t)$
is at most $4d^2 \sqrt{n} \log n$.
Since, for each $u,v$ and $j$,
$|\E [f_{u, \le j}] - \E [f_{v, \le j}]| \le (n-2) (C+1) \E [\phi^2]$, we have
$|H_{u,v,j}(t) - H_{v,v,j}(t)| \le d^2(n-2)(C+1)\E[\phi^2]$, so the difference between
$\E [f_{u,v,\le j,j}]$ and $H_{v,v,j}(t)$
is at most $4d^2 \sqrt{n} \log n + d^2 (n-2) (C+1)\E [\phi^2]$, for each $u$, $v$ and $j$.
Combining the above with (\ref{hatP+}), we see that
the difference between $\E [\widehat{P}^+_{v,j} ]$ and $\frac{1}{n-2} \E [f_{v,C}] H_{v,v,j}(t)$,
which is equal to
\begin{eqnarray*}
\frac{1}{(n-2)^{2d}} \E [f_{v,C}] \E [f_{v, \le j}] \E [f_{v,j}] \sum_{r=1}^d \Big ((n-2)^2 -
(\E [f_{v,\le j}])^2  \Big )^{r-1}\\
\times \Big ((n-2)^2 -  (\E [f_{v,\le j-1}])^2  \Big )^{d-r}
\end{eqnarray*}
is at most
$6d^2 \sqrt{n} \log n + d^2 n (C+1)\E [\phi^2]$ in absolute value.

A similar argument shows that the difference between $\E [\widehat{P}^-_{v,j} ]$ and
\begin{eqnarray*}
\frac{1}{(n-2)^{2d}} \E [f_{v,C}] \E [f_{v, j}] \sum_{i=j+1}^{C-1}\E [f_{v,i}] \sum_{r=1}^d
\Big ((n-2)^2 -  (\E [f_{v,\le i}])^2  \Big )^{r-1}\\
\times \Big ((n-2)^2 -  (\E [f_{v,\le i-1}])^2  \Big )^{d-r}
\end{eqnarray*}
is at most
$6d^2 C\sqrt{n} \log n + d^2 n C(C+1)\E [\phi^2]$ in absolute value.

For $\E [\widehat{Q}^+_{v,j} ]$,
%
we use an argument identical to the one above, considering
\begin{eqnarray*}
f_{v,u, \le j,C} &=& \frac{1}{(n-2)^{2d-1}}\sum_{r=1}^d \sum_{v'} \I_{uv'}^{C}\sum_{v''}
\I_{v''v}^{\le j} \\
&&\mbox{} \times \prod_{s=1}^{r-1} \sum_{w_s,w'_s}(1-\I_{uw_s}^{\le j}\I_{vw'_s}^{\le j})
\prod_{s=r+1}^d \sum_{w_s,w'_s}(1-\I_{uw_s}^{\le j-1}\I_{vw'_s}^{\le j-1})
\end{eqnarray*}
to show that the difference between $\E [\widehat{Q}_{v,j}^+]$ and
\begin{eqnarray*}
\frac{1}{(n-2)^{2d}} \E [f_{v,C}] \E [f_{v, \le j}] \E [f_{v,j}] \sum_{r=1}^d \Big ((n-2)^2 -  (\E [f_{v,\le j}])^2  \Big )^{r-1}\\
\times \Big ((n-2)^2 -  (\E [f_{v,\le j-1}])^2  \Big )^{d-r}
\end{eqnarray*}
is at most $6d^2 \sqrt{n} \log n + d^2 n (C+1)\E [\phi^2]$ in absolute value.


Similarly, the difference between $\E [\widehat{Q}_{v,j}^-]$ and
\begin{eqnarray*}
\frac{1}{(n-2)^{2d}} \E [f_{v,C}] \E [f_{v, j}] \sum_{i=j+1}^{C-1}\E [f_{v,i}] \sum_{r=1}^d \Big ((n-2)^2 -  (\E [f_{v,\le i}])^2  \Big )^{r-1}\\
\times \Big ((n-2)^2 -  (\E [f_{v,\le i-1}])^2  \Big )^{d-r}
\end{eqnarray*}
is at most $6d^2 C\sqrt{n} \log n + d^2 n C(C+1)\E [\phi^2]$ in absolute value.

In summary, we have shown that
\begin{eqnarray*}
\lefteqn{\Big | \E [g_{v,j}(X_t)] - (n-2)g_j (\eta_t^v) \Big |}\\
 &\le& \Big | \E [ g_{v,j}(X_t)] - \E [ \widehat{g}_{v,j}(X_t)] \Big| + \Big | \E [\widehat{g}_{v,j}(X_t)] - (n-2)g_j (\eta_t^v) \Big | \\
&\le& 12d^2 (C+1)^3 n \E [\widetilde{\phi} (X_t)] + 12 d^2 (C+1) \sqrt{n} \log n,
\end{eqnarray*}
where $\eta_t^v$ is the vector with components $\eta_t (v,j) = \frac{1}{n-2} \E f_{v,j}(X_t)$.

Now, each of the components $\zeta_t(v,j)$ and $\eta_t(v,j)$ is non-negative, and we have $\sum_{j=0}^C \zeta_t (v,j) \le 1$ and
$\sum_{j=0}^C \eta_t (v,j) \le \frac{n-1}{n-2}$.   Furthermore, $|\zeta_t (v,j) - \eta_t(v,j)| \le \frac{1}{n-2}$ for all $j$.
Also, exactly as in the proof of (\ref{Lipschitz-g}) below, whenever 
$\eta$ and $\zeta$ are in
$\{\xi \in \R^{C+1}: \xi(j) \ge 0 \mbox{ for each } j, \sum_j \xi(j) \le \frac{n-1}{n-2}\}$,
we have
\begin{eqnarray*}
|g_k (\eta) - g_k (\zeta)| \le 3d^2 (C+1)^2 \left( \frac{n-1}{n-2} \right)^3
\max_{0 \le j \le C} |\eta(j)-\zeta(j)|.
\end{eqnarray*}

It follows that, for $n \ge 6$,
\begin{eqnarray*}
|g_j (\eta_t^v) - g_j (\zeta_t^v)| \le 3d^2 (C+1)^2 \Big ( \frac{n-1}{n-2} \Big )^3\frac{1}{n-2}
\le 6 d^2 (C+1)^2 \frac{1}{n-2},
\end{eqnarray*}
and so, using the fact that $|g_j(\zeta_t^v)| \le 2d(C+1)$, 
\begin{eqnarray*}
\lefteqn{\Big | \E [g_{v,j}(X_t)] - (n-1)g_j (\zeta_t^v) \Big |} \\
& \le & 12d^2 (C+1)^3 n \E [\widetilde{\phi} (X_t)] + 12 d^2 (C+1) \sqrt{n} \log n + 6d^2 (C+1)^2 + 2d(C+1).
\end{eqnarray*}
As $n \ge (C+1)^2$, we may now write
\begin{eqnarray*}
\Big | \E [g_{v,j}(X_t)] - (n-1)g_j (\zeta_t^v) \Big |
\le 12d^2 (C+1)^3 n \E [\widetilde{\phi} (X_t)] + 20 d^2 (C+1) \sqrt{n} \log n,
\end{eqnarray*}
as claimed.
\end{proof}



We now study the changes of $\phi (X_t)$ over time.
For distinct $u$ and $v$, and $j,k \in \{ 0, \dots, C\}$, we define
$$
\phi^1_{u,v,j,k} = \frac{1}{n-2}\sum_{w} \I_{uw}^j \I_{vw}^{k} -
\frac{1}{(n-2)^2}\sum_{w \not = u,v} \I_{uw}^{j} \sum_{w' \not = u,v}\I_{vw'}^{k};
$$
\begin{eqnarray*}
\phi^2_{u,v,j} & = & \frac{1}{n-2} (f_{u,j} - f_{v,j}) =
\frac{1}{n-2} \Big( \sum_{w \not = u,v} \I_{uw}^{j} - \sum_{w \not =u,v} \I_{vw}^j \Big);
\end{eqnarray*}
$$
\phi^3_{u,v} (x) = \frac{1}{n-2} \sum_{w\not = u,v} x(\{u,v\},w).
$$
Then we have $\phi^1 = \max_{u,v, j,k} |\phi^1_{u,v,j,k}|$,
$\phi^2 = \max_{u,v,j} |\phi^2_{u,v,j}|$ and $\phi^3 = \max_{u,v} \phi^3_{u,v}$, where
all maximisations are over distinct nodes $u$ and $v$ and, where appropriate,
$j,k \in \{ 0, \dots, C\}$.  These functions are similar to ones
in~\cite{ch}: we prove an analogue of Lemma~2 in~\cite{ch}, leaving
some details to an appendix, but our task is more complex as we deal with $d > 1$,
and we fill in a key point that is dealt with rather brusquely in~\cite{ch}.

Once again, our argument uses the discrete chain $(\widehat{X}_t)$. As before, let $(\widehat{\cF}_t)$ denote the
natural filtration of $(\widehat{X}_t)$.
Let $\widetilde{A}_t = \{\widehat{X}_s \in \widetilde{S} \mbox{ for all } s \le t-1\}$.
For a function $f: S \to \R$, we define
$\Delta f (\widehat{X}_t) = f(\widehat{X}_t) - f(\widehat{X}_{t-1})$, the increment of the function on
one step of the discrete chain.  Our first goal is to provide upper bounds on
$\E [|\Delta \phi^1_{u,v,j,k}(\widehat{X}_t)| \mid \widehat{\cF}_{t-1}]$,
$\E [|\Delta \phi^2_{u,v,j}(\widehat{X}_t)| \mid \widehat{\cF}_{t-1}]$ and
$\E [|\Delta \phi^3_{u,v}(\widehat{X}_t)| \mid \widehat{\cF}_{t-1}]$, in terms of $\phi(\widehat{X}_{t-1})$,
valid, on the event $\widetilde{A}_t$, for all distinct nodes $u$ and $v$, and, where appropriate, all $j,k \in \{0,\dots,C\}$.

The proof of the following lemma consists of routine but tedious calculations, and these are relegated to the
appendix.

\begin{lemma} \label{lem.appendix}
Suppose $n \ge n_0(\lambda,d,C,t_0)$, and $t \le cn^2$, where $c = 8 \lambda t_0$.
For $\rho$ any one of the functions $\phi^1_{u,v,j,k}$, $\phi^2_{u,v,j}$,
or $\phi^3_{u,v}$, we have, on $\widetilde{A}_t$,
$$
\E [\big|\Delta \rho (\widehat{X}_t)\big| \mid \widehat{\mathcal F}_{t-1}]
\le \frac{c_1}{n^2} \phi (\widehat{X}_{t-1}) + \frac{c_2}{n^3},
$$
where $c_1 = 26 (1+1/\lambda) d^2 (C+1)^3$ and $c_2 = 64 \lambda d^2 (C+1)^3$.

If $d=1$, we have the same conclusion with $c_1 = 26 (1+1/\lambda) (C+1)$ and
$c_2 = 64\lambda (C+1)$.
\end{lemma}

Now we are in a position to prove the other result required for Lemma~\ref{lem.gen-expec}.

\begin{lemma} \label{lem.phi-over-time}
For all $t \le t_0$, and $n \ge n_0(\lambda,d,C,t_0)$, we have
$$
\E \phi(X_t) \le 2 e^{208(\lambda+1)d^2 (C+1)^3 t_0}
\Big (\phi (x_0) + \frac{3\log n}{\sqrt n} \Big).
$$

If $d=1$, and $n \ge n_0(\lambda,1,C,t_0)$, we have the improved bound
$$
\E \phi(X_t) \le 2 e^{208(\lambda+1) (C+1) t_0}
\Big (\phi (x_0) + \frac{3\log n}{\sqrt n} \Big).
$$
\end{lemma}


\begin{proof}
%
%
%
%
Let $m = \phi(x_0) + \frac{3\log n}{\sqrt n}$, and let $E_t$ be the event
$$
\widetilde{A}_t \cap \big\{ \phi(\widehat{X}_s) \le m \left( 1 + \frac{c_1}{n^2} \right)^s, \mbox{ for all }
s \le t-1 \big\} .
$$

Let $\rho$ denote any of the functions $\phi^1_{u,v,j,k}$, $\phi^2_{u,v,j}$ or $\phi^3_{u,v}$.
For each $t$, on the event $E_t$, we have from Lemma~\ref{lem.appendix} that
$$
\E \big(|\rho(\widehat{X}_t)| - |\rho(\widehat{X}_{t-1})| \mid \widehat{\cF}_{t-1} \big)
\le \E (\big|\Delta \rho(\widehat{X}_t)\big| \mid \widehat{\cF}_{t-1} ) \big| \le
\frac{c_1}{n^2} m \left( 1 + \frac{c_1}{n^2} \right)^{t-1} + \frac{c_2}{n^3},
$$
and therefore
$$
\E \big( |\rho(\widehat{X}_t)| \I_{E_t} \big) \le \E \big( |\rho(\widehat{X}_{t-1})|
\I_{E_{t-1}} \big) + \frac{c_1}{n^2} m \left( 1 + \frac{c_1}{n^2} \right)^{t-1} + \frac{c_2}{n^3}.
$$
This yields, for each $t$,
\begin{eqnarray*}
\E \big( |\rho(\widehat{X}_t)| \I_{E_t} \big)
&\le& |\rho(x_0)| + \sum_{s=0}^{t-1} \left( \frac{c_1}{n^2} m
\left( 1 + \frac{c_1}{n^2} \right)^{s-1} + \frac{c_2}{n^3} \right) \\
&\le& \phi(x_0) + m \left\{ \left( 1 + \frac{c_1}{n^2} \right)^t -1 \right\} + \frac{c_2t}{n^3} \\
&=& m \left( 1 + \frac{c_1}{n^2} \right)^t - \frac{3\log n}{\sqrt n} + \frac{c_2t}{n^3}.
\end{eqnarray*}
Let $c = 8\lambda t_0$, and run the discrete chain for $cn^2$ steps.
Note that
$
n \ge n_0(\lambda,d,C,t_0) \ge 2^{18}\lambda^4d^4(C+1)^6 t_0^2 \ge (cc_2)^2
$.
We conclude that, for $t \le cn^2$,
$$
\E \big( |\rho(\widehat{X}_t)| \I_{E_t} \big) \le m \left( 1 + \frac{c_1}{n^2} \right)^t -
\frac{2\log n}{\sqrt n}.
$$

We show by induction on $t$ that
$\P (\overline{E_{t+1}}) \le t e^{-\frac12 \gamma \log^2 n}$,
for all $t < cn^2$.
This holds for $t=0$.  If the induction hypothesis holds for $t-1$, then
$$
\E \big( |\rho(\widehat{X}_t)| \I_{\overline{E_t}} \big) \le C \P (\overline{E_t})
\le Ccn^2 e^{-\frac12 \gamma \log^2 n} \le \frac{\log n}{\sqrt n},
$$
and so
$$
\E \big( |\rho(\widehat{X}_t)| \big) = \E \big( |\rho(\widehat{X}_t)| \I_{E_t})  +
\E \big( |\rho(\widehat{X}_t)| \I_{\overline{E_t}} \big)
\le m \left( 1 + \frac{c_1}{n^2} \right)^t - \frac{\log n}{\sqrt n}.
$$
Thus
$$
\P(\overline{E_{t+1}}) \le \P(\overline{\widetilde{A}_t}) + \P(\overline{E_t}) +
(C+2)^2 n^2 \max_\rho \P \Big( |\rho(\widehat{X}_t)| \ge
\E |\rho(\widehat{X}_t)| + \frac{\log n}{\sqrt n} \Big),
$$
where the maximum is over all the functions $\phi^1_{u,v,k,j}$, $\phi^2_{u,v,j}$ and
$\phi^3_{u,v}$, noting that there are at most $(C+2)^2 n^2$ such functions.

Inequality (\ref{eq:leaveS_0}) implies that
$$
\P(\overline{\widetilde{A}_t}) \le e^{-n/8} \le \frac12 e^{-\frac12 \gamma \log^2 n},
$$
for $t \le cn^2$, since the chain starts at $x_0 \in S_1$.

To bound $\max_\rho \P \Big( |\rho(\widehat{X}_t)| \ge \E |\rho(\widehat{X}_t)| + \frac{\log n}{\sqrt n} \Big)$,
we establish concentration of measure results for the functions $|\phi^1_{u,v,j,k}|$,
$|\phi^2_{u,v,j}|$ and $\phi^3_{u,v}$, proceeding as in Sections~\ref{sec:couple}
and~\ref{sec:conc-route}.  Indeed, it is easy to see that
$$
(n-2) \Big| |\phi^1_{u,v,j,k}(x)| - |\phi^1_{u,v,j,k}(y)| \Big| \le 2 (\|x-y\|_u + \|x-y\|_v),
$$
$$
(n-2) \Big| |\phi^2_{u,v,j}(x)| - |\phi^2_{u,v,j}(y)| \Big| \le 2 (\|x-y\|_u + \|x-y\|_v),
$$
and
$$
(n-2) \Big| \phi^3_{u,v}(x) - \phi^3_{u,v}(y) \Big| \le \|x-y\|_u,
$$
for all $u$, $v$, $j$ and $k$, and all $x,y$ in $\widetilde S$.
Calculations exactly as leading up to (\ref{eq.discrete-conc}) and (\ref{eq:leaveS_0}) now give, for any $\rho$, any $t\le cn^2$, and any $a \le n$,
\begin{eqnarray*}
\P \Big( \Big| |\rho(\widehat{X}_t)| - \E |\rho(\widehat{X}_t)| \Big| \ge \frac{a}{n-2} \Big) &\le&
2e^{-a^2/2^{21}(d^4+C/\lambda)(c+1)^3n e^{96dc}} + e^{-n/8} \\
&\le& 4e^{-4\gamma a^2/n}.
\end{eqnarray*}
Applying this with $a = \frac12 \sqrt n \log n$ gives
$$
\P \Big( \Big| |\rho(\widehat{X}_t)| - \E |\rho(\widehat{X}_t)| \Big| \ge \frac{\log n}{\sqrt n} \Big)
\le 4e^{-\gamma \log^2 n} \le 4 n^{-4} e^{-\frac12\gamma \log^2 n}.
$$

We thus have, using also the induction hypothesis, that
$$
\P(\overline{E_{t+1}}) \le \frac12 e^{-\frac12\gamma \log^2 n} + (t-1)e^{-\frac12\gamma \log^2 n} +
4 (C+2)^2 n^2 n^{-4} e^{-\frac12\gamma \log^2 n} \le t e^{-\frac12\gamma \log^2 n},
$$
as required for the induction step.

Recall that $t_0 = c/8 \lambda$.  Let $D$ be the event that
there are no more than $cn^2$ events in the continuous-time chain $X$ during the interval
$[0,t_0]$, so $\P (\overline {D}) \le e^{-c n^2/6}$.  As $\phi$ is bounded above by $C$, for all $t \le t_0$,
\begin{eqnarray*}
\E [\phi (X_t)] & = & \E [\phi(X_t) \I_D \I_{E_{cn^2}}] + \E [\phi(X_t) \I_{\overline D} \I_{E_{cn^2}}]
+ \E [\phi(X_t) \I_{\overline{E_{cn^2}}}] \\
&\le& \Big( \phi(x_0) + \frac{3\log n}{\sqrt n} \Big) e^{c_1c} +C e^{-cn^2/6} +
C cn^2 e^{-\frac12\gamma \log^2 n} \\
&\le & 2 e^{c_1c} \Big( \phi(x_0) + \frac{3\log n}{\sqrt n} \Big)
=  2 e^{8c_1\lambda t_0} \Big (\phi (x_0) + \frac{3\log n}{\sqrt n} \Big ).
\end{eqnarray*}

Substituting for the value of $c_1$ gives the required result, both in the general case and the case $d=1$.
\end{proof}




\section{Proofs of Theorem~\ref{thm.main-result} and Theorem~\ref{thm.main-result-d=1}}

\label{sec:proof-main}

We now use the results of the previous section to derive the main theorem.
We need one routine lemma, showing that the function $F$ in~(\ref{eq.F}) is
Lipschitz with an appropriate constant, in the domain of interest to us.

\begin{lemma}
\label{lem.lipschitz}
Let $d$ and $C$ be positive integers. Let $\lambda > 0$. The function $F$ in~(\ref{eq.F}) is Lipschitz with constant $8d^2 (\lambda + 1) (C+1)^2$
on the set
$\Delta^{C+1}_\le$, with respect to the $\ell_{\infty}$ norm.

For $d=1$, $F$ has Lipschitz constant $2 \lambda + 2 C + 6$ on
$\Delta^{C+1}_\le$, with respect to the $\ell_{\infty}$ norm.

Moreover, for any $d$, $C$ and $\lambda$, the function $F$ is locally Lipschitz on $\R^{C+1}$ with
respect to the $\ell_\infty$ norm.
\end{lemma}

\begin{proof}
For $0 < k < C$,
\begin{eqnarray*}
|F_k (\xi) - F_k (\eta)| \le \lambda | \xi(k-1) - \eta(k-1)| + \lambda | \xi(k) - \eta(k)|\\
 + k | \xi(k) - \eta(k)| + (k+1) | \xi(k+1) - \eta(k+1)|\\
+ |g_{k-1} (\xi) - g_{k-1} (\eta)| + |g_k (\xi) - g_k (\eta)|.
\end{eqnarray*}
Now, for $\xi,\eta \in \Delta^{C+1}_\le$,
\begin{eqnarray*}
\big |2\xi(C) \sum_{r=1}^d \xi(k) \xi(\le k) (1-\xi(\le k)^2)^{r-1} (1-\xi(\le k-1)^2)^{d-r}\\
- 2\eta(C) \sum_{r=1}^d \eta(k) \eta(\le k) (1-\eta(\le k)^2)^{r-1} (1-\eta(\le k-1)^2)^{d-r} \big| \\
\le 2\big(2d + d(C+1) + d(d-1)(C+1)\big) \max_{0 \le j \le C} |\xi(j)-\eta(j)|\\
\le 3d^2 (C+1)\max_{0 \le j \le C} |\xi(j)-\eta(j)|.
\end{eqnarray*}
Also,
\begin{eqnarray*}
\big| \xi(C) \sum_{r=1}^d \xi(k) \sum_{i=k+1}^{C-1}\xi(i) (1-\xi(\le i)^2)^{r-1} (1-\xi(\le i-1)^2)^{d-r}\\
- \eta(C) \sum_{r=1}^d \eta(k) \sum_{i=k+1}^{C-1}\eta(i) (1-\eta(\le i)^2)^{r-1} (1-\eta(\le i-1)^2)^{d-r} \big|
\\
\le 3d^2 C(C+1) \max_{0 \le j \le C} |\xi(j)-\eta(j)|.
\end{eqnarray*}
It follows that, for $k=0, \ldots, C-1$,
\begin{eqnarray} \label{Lipschitz-g}
|g_k (\xi) - g_k (\eta)| \le 3d^2 (C+1)^2 \max_{0 \le j \le C} |\xi(j)-\eta(j)|.
\end{eqnarray}
So, for $0 < k < C$, for $\xi,\eta \in \Delta^{C+1}_\le$,
\begin{eqnarray*}
|F_k (\xi) - F_k (\eta)| & \le &
(2 \lambda + 2C + 6d^2 (C+1)^2) \max_{0 \le j \le C} |\xi(j)-\eta(j)|\\
& \le & 8d^2 (\lambda + 1) (C+1)^2 \max_{0 \le j \le C} |\xi(j)-\eta(j)|,
\end{eqnarray*}
and the same bound holds for $k=0$ and $k=C$.

For $d=1$, it is easy to see that, for $k=0, \ldots, C$ and
$\xi,\eta \in \Delta^{C+1}_\le$,
$|g_k(\xi) - g_k(\eta)| \le 6 \max_j |\xi(j)-\eta(j)|$,
and therefore, for each $k=0, \dots, C$,
\begin{eqnarray*}
|F_k (\xi) - F_k (\eta)| \le (2 \lambda + 2 C + 6) \max_{0 \le j \le C} |\xi(j)-\eta(j)|.
\end{eqnarray*}

Similar arguments show that $F$ is locally Lipschitz throughout its domain for any $d$, $C$
and $\lambda$.
\end{proof}


\begin{lemma}
\label{lem.unique-sol}
Let $d$ and $C$ be positive integers. Let $\lambda > 0$, and let $\xi_0$
be in $\Delta^{C+1}_=$.  The differential
equation~(\ref{eq.diff-eq}) has a unique solution $(\xi_t)$ subject to initial condition
$\xi_0$, valid for all times $t \ge 0$. Furthermore,
$\xi_t \in \Delta^{C+1}_=$ for all $t \ge 0$.
\end{lemma}
\begin{proof}
By Lemma~\ref{lem.lipschitz}, $F$ is locally Lipschitz with respect to the $\ell_{\infty}$ norm,
so the differential equation~(\ref{eq.diff-eq}) has a unique maximal solution $(\xi_t)$ valid
on $[0,t_{\max})$ for some $t_{\max} > 0$.  Moreover, if $t < t_{\max}$, then $\| \xi_t\|_\infty \to \infty$ as $t \to \infty$.

Note that $\sum_{j=0}^C F_j (\xi) = 0$ for all
$\xi$ and so $\sum_{j=0}^C \xi_t (j)$ is constant for all times $t< t_{\max}$, and hence is
equal to 1.  Also, $F_j (\xi) \ge 0$ whenever $\xi (j) = 0$. By standard arguments,
$\xi_t (j) \ge 0$ for all $j$ and all $t < t_{\max}$.  Thus
$\xi_t \in \Delta^{C+1}_=$
for all $t < t_{\max}$, and hence
$t_{\max} = \infty$.
%
\end{proof}

\begin{lemma}
Let $\lambda$ and $t_0$ be positive reals, let $d$ and $C$ be positive integers, and suppose that
$n \ge n_0(\lambda,d,C,t_0)$.
Let $\xi_0$  be in $\Delta^{C+1}_=$.
Then, 
for each $v$ and each $t \in [0,t_0]$,
\begin{eqnarray*}
\lefteqn{\sup_j \Big|\frac{1}{n-1} \E [f_{v,j} (X_t)] - \xi_t (j) \Big|} \\
&\le& \Big (\sup_j \Big| \frac{1}{n-1} f_{v,j}(x_0)-\xi_0(j) \Big| \\
&&\mbox{} + 63\lambda t_0 d^2 (C+1)^3 \Big (\phi (x_0) + \frac{3\log n}{\sqrt n} \Big ) \Big)
e^{216(\lambda+1)d^2 (C+1)^3 t_0}.
\end{eqnarray*}

For $d=1$, we have
\begin{eqnarray*}
\lefteqn{\sup_j \Big|\frac{1}{n-1} \E [f_{v,j} (X_t)] - \xi_t (j) \Big|} \\
&\le& \Big ( \sup_j \Big| \frac{1}{n-1} f_{v,j}(x_0)-\xi_0(j) \Big| \\
&&\mbox{} + 63\lambda t_0 (C+1)^3 \Big (\phi (x_0) + \frac{3\log n}{\sqrt n} \Big ) \Big)
e^{216(\lambda+1) (C+1) t_0}.
\end{eqnarray*}
\end{lemma}

\begin{proof}
For each $j$, for $t \le t_0$, we have
$$\xi_t (j) = \xi_0(j) + \int_0^t F_j (\xi_s) ds.$$
As before, for $v \in V_n$ and $j \in \{0, \ldots, C\}$,
$\zeta_t(v,j) = (n-1)^{-1}\E [f_{v,j} (X_t)]$, and
$\zeta_t^v$ is the vector $(\zeta_t(v,j): j \in \{0, \ldots, C\}$). Recall that $A$ is the generator
of the process $X$.  Then, for every $j$,
$$\frac{d \zeta_t(v,j)}{dt}  =  \frac{1}{n-1}\E ( A f_{v,j} (X_t)),$$
so
\begin{eqnarray*}
\lefteqn{\zeta_t(v,j) - \zeta_0(v,j) = \frac{1}{n-1}\int_0^t \E [ A f_{v,j} (X_s)]\, ds} \\
& = & \int_0^t \Big (\lambda \zeta_s(v,j-1) - \lambda \zeta_s (v,j) -  j \zeta_s(v,j) +
(j+1) \zeta_s(v,j+1)\Big )\, ds\\
&&\mbox{} + \int_0^t \Big (\frac{\lambda}{n-1} (\E [g_{v,j-1} (X_s)] -
\E [g_{v,j} (X_s)])\Big ) \,ds \\
&=& \int_0^t F_j(\zeta_s^v) \,ds + \lambda \int_0^t \Big( \frac{1}{n-1} \E[g_{v,j-1}(X_s)] -
g_{j-1}(\zeta_s^v) \Big) \,ds \\
&&\mbox{} - \lambda \int_0^t \Big( \frac{1}{n-1} \E[g_{v,j}(X_s)] -
g_j(\zeta_s^v) \Big) \,ds.
\end{eqnarray*}
For each $j$ and $t$, let $\epsilon_t (v,j) = \sup_{s \le t}|\zeta_s (v,j) - \xi_s(j)|$, and let
$\epsilon_t^v$ be the vector with components $\epsilon_t (v,j)$ ($j=0, \ldots, C$).
Let $L$ be the Lipschitz constant of $F$ over
$\Delta^{C+1}_\le$, as in
Lemma~\ref{lem.lipschitz}. Since $\zeta_s^v$ and $\xi_s$ are in $\Delta^{C+1}_=$ for each $s$,
we have, using Lemma~\ref{lem.gen-expec},
\begin{eqnarray*}
\|\epsilon_t^v\|_{\infty}  & \le & \|\epsilon_0^v\|_{\infty} +
\int_0^t \Big(L \|\epsilon_s^v\|_\infty
+ 2\lambda \sup_{s \le t}\max_j \Big|\frac{1}{n-1}\E [g_{v,j} (X_s)] - g_j (\zeta_s^v)\Big| \Big) ds\\
& \le  & \|\epsilon_0^v\|_{\infty} + L \int_0^t \|\epsilon_s^v\|_\infty ds\\
&&\mbox{} + 63\lambda t d^2 (C+1)^3 \Big (\phi (x_0) + \frac{3\log n}{\sqrt n} \Big )
e^{208(\lambda+1)d^2 (C+1)^3 t_0}.
\end{eqnarray*}
By Gronwall's lemma,
for each $t\le t_0$, $\|\epsilon_t^v\|_{\infty}$ is at most
$$
e^{Lt_0} \Big (\|\epsilon_0^v\|_{\infty}
 + 63\lambda t_0 d^2 (C+1)^3 \Big (\phi (x_0) + \frac{3\log n}{\sqrt n} \Big )
e^{208(\lambda+1)d^2 (C+1)^3 t_0} \Big),
$$
which gives the required result, since we may take $L = 8d^2(\lambda +1)(C+1)^2$ by Lemma~\ref{lem.lipschitz}.
The result for $d=1$ follows in an identical manner.
\end{proof}

\begin{proofof}{Theorem~\ref{thm.main-result}}
Let
\begin{eqnarray*}
w &=& \Big (\sup_{j,u} \Big| \frac{1}{n-1} f_{u,j}(x_0)-\xi_0(j) \Big| + 63\lambda t_0 d^2 (C+1)^3 \Big (\phi (x_0) + \frac{3\log n}{\sqrt n} \Big ) \Big) \\
&&\mbox{} \times e^{216(\lambda+1)d^2 (C+1)^3 t_0}.
\end{eqnarray*}
The previous lemma, along with (\ref{eq.conc-g2}), yields
\begin{eqnarray*}
\lefteqn{\P \Big(\sup_{v,k,t} |f_{v,k} (X_t) -(n-1) \xi_t (k)| > (n-1)w + n^{1/2} \log n \Big)} \\
&\le& \P \Big( \sup_{v,k,t} |f_{v,k} (X_t) - \E f_{v,k} (X_t)| > n^{1/2} \log n \Big ) \\
&\le& 40t_0\lambda Cn^3 e^{-\gamma\log^2 n} \le e^{-\frac12 \gamma \log^2 n},
\end{eqnarray*}
where the supremum is over all $v \in V_n$, $k \in \{ 0, \dots, C\}$, and $t \le t_0$.
\end{proofof}

The proof of Theorem~\ref{thm.main-result-d=1} is essentially identical.






\section{Initial conditions}

\label{sec:init-cond}

Given a node $v$, in order for the functions $f_{v,k} (X_t)$ ($k=0, \ldots, C$) to be well approximated by the solution $(\xi_t)$ to the differential equation~(\ref{eq.diff-eq}), given the initial state $X_0$ of the system, we must choose initial condition $\xi_0$ for~(\ref{eq.diff-eq}) in such a way that $\sup_{j \in \{0, \ldots, C\}} \epsilon_0 (v,j)$ is small, where $\epsilon_0 (v,j)= |\xi_0 (j) - (n-1)^{-1}\E [f_{v,j} (X_0)]|$. For instance, we can take
$\xi_0 (j) = (n-1)^{-1}\E [f_{v,j} (X_0)]$ for $j=0, \ldots, C$.
In addition, there are restrictions on allowed initial states $X_0$, to ensure that $\phi (X_0)$ is not too large.

Clearly, $X_0 = 0$ implies that $\phi^1 (X_0) = 0$, so the law of large numbers in Theorem~\ref{thm.main-result} holds if $\xi_0$
satisfies $\xi_0 (0) =1$ and $\xi_0 (j) = 0$ for $j=1, \ldots, C$.


Now consider an initial state obtained as follows. For a constant $c_0 > 0$, we throw $\lfloor c_0 {n \choose 2} \rfloor$ calls onto the network,
one at a time. Each call chooses endpoints $u$ and $v$ uniformly at random; it is routed onto $\{u,v\}$ if there is spare capacity. Otherwise, it
chooses an ordered list of $d$ intermediate nodes $(w_1, \ldots, w_d)$ uniformly at random with replacement and is routed onto the first route
$\{u,w_i\}, \{v,w_i\}$ minimising the maximum load of the two links, if this route has capacity. If each of the $d$ routes has a full link,
then the call is lost. Let $X_0$ be an initial state obtained from this $\lfloor c_0 {n \choose 2} \rfloor$-step allocation.  We observe here that the variables $X_0(\{u,w\})$ are all
identically distributed, and therefore so are indicator variables of the form $\I_{uv}^j(X_0)$,
for each $j \in 0, \dots, C$.
We will show that, with high probability, $\phi(X_0)$ is at most $3 n^{-1/2} \log n$.

We start by analysing $\phi^1(X_0)$, using the bound
\begin{eqnarray*}
\phi^1 (X_0) &=& \max_{u,v: u \not = v} \max_{j,k} \Big | \frac{1}{n-2} \sum_w \I_{uw}^j (X_0) \I_{vw}^k (X_0) \\
&&\mbox{} - \frac{1}{(n-2)^2} \sum_{w \not = u,v} \I_{uw}^j (X_0) \sum_{w' \not = u,v} \I_{vw'}^k (X_0)\Big |\\
&\le& \max_{u,v: u \not = v} \max_{j,k}  \frac{1}{n-2} \Big |\sum_w \I_{uw}^j (X_0) \I_{vw}^k (X_0) - \E [ \sum_w \I_{uw}^j (X_0) \I_{vw}^k (X_0) ] \Big |\\
&&\mbox{} + \max_{u,v: u \not = v} \max_{j,k}  \frac{1}{n-2} \Big | \E [ \sum_w \I_{uw}^j (X_0) \I_{vw}^k (X_0) ] \\
&&\mbox{} - \frac{1}{n-2} \E [ \sum_{w \not = u,v} \I_{uw}^j (X_0) \sum_{w' \not = u,v} \I_{vw'}^k (X_0) ] \Big |\\
&&\mbox{} + \max_{u,v: u \not = v} \max_{j,k} 
\frac{1}{(n-2)^2} \Big | \E [ \sum_{w \not = u,v} \I_{uw}^j (X_0) \sum_{w' \not = u,v} \I_{vw'}^k (X_0) ] \\
&&\mbox{} - \sum_{w \not = u,v} \I_{uw}^j (X_0) \sum_{w' \not = u,v} \I_{vw'}^k (X_0)\Big |
\end{eqnarray*}

Arguments similar to those in Sections~\ref{sec:couple} and~\ref{sec:conc-route} show that
$\sum_w \I_{uw}^j (X_0) \I_{vw}^k (X_0)$ and $\sum_{w \not = u,v} \I_{uw}^j (X_0) \sum_{w'\not=u,v} \I_{vw'}^k(X_0)$ are well-concentrated.  Specifically,
there exists a constant $\gamma_0 > 0$ such that for all $u,v,j,k$,
\begin{eqnarray*}
\P \Big (\Big |\sum_{w\not=u,v} \I_{uw}^j (X_0) \I_{vw}^k (X_0) - \E [ \sum_{w\not=u,v} \I_{uw}^j (X_0) \I_{vw}^k (X_0) ] \Big | \ge \sqrt{n} \log n \Big ) \\
\le 4 e^{-\gamma_0 \log^2 n},
\end{eqnarray*}
\begin{eqnarray*}
\P \Big (\frac{1}{n-2} \Big |\sum_{w\not=u,v} \I_{uw}^j (X_0) \sum_{w'\not=u,v} \I_{vw'}^k (X_0) - \E [ \sum_{w\not=u,v} \I_{uw}^j (X_0)
\sum_{w'\not=u,v} \I_{vw'}^k (X_0) ] \Big | \\
\ge \sqrt{n} \log n \Big ) \le 4 e^{-\gamma_0 \log^2 n}.
\end{eqnarray*}
We also note for future reference that we have, similarly, for all $u,j$,
\begin{equation}
\label{eq.dev}
\P \Big ( |f_{u,j} (X_0) - \E [f_{u,j} (X_0)| \ge \sqrt{n} \log n \Big)
\le 4 e^{-\gamma_0 \log^2 n}.
\end{equation}

We deduce that, with probability at least $1 - 8(C+1)^2n^2 e^{-\gamma_0 \log^2 n}$,
\begin{eqnarray*}
\phi^1(X_0) & \le& \max_{u,v: u \not = v} \max_{j,k}  \frac{1}{n-2} \Big | \E [ \sum_w \I_{uw}^j (X_0) \I_{vw}^k (X_0) ] \\
&&\mbox{} - \frac{1}{n-2} \E [ \sum_{w \not = u,v} \I_{uw}^j (X_0) \sum_{w' \not = u,v} \I_{vw'}^k (X_0) ] \Big |
+ 3 \frac{\log n}{\sqrt n}.
\end{eqnarray*}

We now fix $u,v,j,k$, and consider
\begin{eqnarray}
\lefteqn{ \Big | \E [ \sum_{w \not=u,v} \I_{uw}^j (X_0) \I_{vw}^k (X_0) ]
- \frac{1}{n-2} \E [ \sum_{w \not = u,v} \I_{uw}^j (X_0) \sum_{w' \not = u,v} \I_{vw'}^k (X_0) ] \Big | } \nonumber\\
&\le& \sum_{w\not=u,v} \Big| \E[\I_{uw}^j(X_0) \I_{vw}^k(X_0)] -
\E \I_{uw}^j(X_0) \E \I_{vw}^k(X_0) \Big| \nonumber \\
&+& \Big| \sum_{w\not=u,v} \E \I_{uw}^j(X_0) \E \I_{vw}^k(X_0) -
\frac{1}{n-2} \sum_{w\not=u,v} \E[ \I_{uw}^j(X_0) ] \sum_{w'\not=u,v} \E [\I_{vw'}^k(X_0)]\Big|
\nonumber \\
&+& \frac{1}{n-2} \sum_{w\not=u,v} \sum_{w'\not=u,v} \Big| \E[ \I_{uw}^j(X_0) ]
\E [\I_{vw'}^k(X_0)] - \E [ \I_{uw}^j (X_0) \I_{vw'}^k (X_0) ] \Big |. \label{eq.sum}
\end{eqnarray}
Since all the $\I_{uw}^j$ are identically distributed, as are all the $\I_{vw}^k$, the second
of the three terms in (\ref{eq.sum}) is identically zero.

To bound the first of the three terms in (\ref{eq.sum}), we note first that since, for fixed
$w$, all the variables $X_0(\{w',w\})$ are identically distributed, we have
$$
\E f_{w,j}(X_0) f_{w,k}(X_0) = (n-1)\E \I_{uw}^j(X_0) \I_{uw}^k(X_0) + (n-1)(n-2) \E \I_{uw}^j(X_0) \I_{vw}^k(X_0)
$$
for any distinct $u$ and $v$, and therefore
$$
\left| \frac{1}{(n-1)^2} \E f_{w,j}(X_0) f_{w,k}(X_0) - \E \I_{uw}^j(X_0) \I_{vw}^k(X_0) \right|
\le \frac{1}{n-1}.
$$
Also $\E \I_{uw}^j(X_0) = \frac{1}{n-1} \E f_{w,j}(X_0)$, for any $u$ and $w$.  Thus we have
\begin{eqnarray*}
\lefteqn{\Big| \E[\I_{uw}^j(X_0) \I_{vw}^k(X_0)] - \E \I_{uw}^j(X_0) \E \I_{vw}^k(X_0) \Big| }\\
&\le& \frac{1}{(n-1)^2} \Big| \E f_{w,j}(X_0) f_{w,k}(X_0) - \E f_{w,j}(X_0) \E f_{w,k}(X_0) \Big|
+ \frac{1}{n-1}
\end{eqnarray*}

From~(\ref{eq.dev}) we have that, for sufficiently large $n$,
\begin{eqnarray*}
\lefteqn{ \Big| \E \left[ f_{w,j}(X_0) f_{w,k}(X_0) \right] - \E f_{w,j}(X_0) \E f_{w,k}(X_0)
\Big| } \\
&=& \Big| \E \left[ \left(f_{w,j}(X_0) - \E f_{w,j}(X_0)\right) \left(f_{w,k}(X_0) - \E f_{w,k}(X_0)\right) \right] \Big| \\
&\le& \left( \sqrt n \log n \right)^2 + 8 e^{-\gamma_0 \log^2 n} n^2 \le 2 n \log^2 n.
\end{eqnarray*}

Hence, for any distinct $u$, $v$, $w$, and any $j$ and $k$,
\begin{eqnarray}
\lefteqn{\sum_{w\not=u,v} \left| \E \left[ \I_{uw}^j(X_0) \I_{vw}^k(X_0) \right] - \E \I_{uw}^j(X_0) \E \I_{vw}^k(X_0) \right|} \nonumber
\\
&\le& (n-2) \left( \frac{2n\log^2 n}{(n-1)^2} + \frac{1}{n-1} \right) \le 3\log^2 n. \label{eq.indep-3}
\end{eqnarray}

The same argument, applied with $f_{w,j}(X_0)$ and $f_{w',k}(X_0)$, gives
\begin{equation} \label{eq.indep-4}
\left| \E \left[ \I_{uw}^j(X_0) \I_{vw'}^k(X_0) \right] - \E \I_{uw}^j(X_0) \E \I_{vw'}^k(X_0) \right| \le \frac{3\log^2 n}{n}
\end{equation}
whenever $u$, $v$, $w$ and $w'$ are all distinct, and for any $j$ and $k$.
We thus obtain a bound on the final term in the sum~(\ref{eq.sum}):
$$
\frac{1}{n-2} \sum_{w\not=u,v} \sum_{w'\not=u,v} \Big| \E[ \I_{uw}^j(X_0) ]
\E [\I_{vw'}^k(X_0)] - \E [ \I_{uw}^j (X_0) \I_{vw'}^k (X_0) ] \Big | \le 4\log^2 n.
$$
Hence we have
$$
\Big | \E [ \sum_{w \not=u,v} \I_{uw}^j (X_0) \I_{vw}^k (X_0) ]
- \frac{1}{n-2} \E [ \sum_{w \not = u,v} \I_{uw}^j (X_0) \sum_{w' \not = u,v} \I_{vw'}^k (X_0) ] \Big | \le 7\log^2 n.
$$
It follows that, for $n$ large enough,
\begin{eqnarray*}
\P (\phi^1 (X_0) \ge 4 \frac{\log n}{\sqrt n} ) \le 8 (C+1)^2 n^2 e^{-\gamma_0 \log^2 n}.
\end{eqnarray*}

Furthermore,
\begin{eqnarray*}
\phi^2 (X_0) &=& \max_{u,v: u \not = v} \max_j \frac{1}{n-2} |f_{u,j}(X_0) - f_{v,j}(X_0)|\\
&\le& \max_{u,v: u \not = v} \max_j \frac{1}{n-2} \Big ( |f_{u,j} (X_0) - \E [f_{u,j} (X_0)] | \\
&&\mbox{} + |f_{v,j} (X_0) - \E [f_{v,j} (X_0)] | \Big ),
\end{eqnarray*}
and hence, by~(\ref{eq.dev}),
\begin{eqnarray*}
\P (\phi^2 (X_0) \ge 2 \frac{\log n}{\sqrt n} ) \le 8 (C+1) n^2 e^{-\gamma_0 \log^2 n}.
\end{eqnarray*}

For $\phi^3$, standard Poisson tail bounds yield that, for each fixed pair $\{u,v\}$, the probability that there are more than $c_0\log^2 n$
calls with endpoints $u$ and $v$ is at most
$(e/\log^2 n)^{c_0 \log^2 n} \le e^{-\gamma_0 \log^2 n}$ for sufficiently large $n$.  Thus
$$
\P \Big(\phi^3(X_0) > \frac{c_0 \log^2 n}{n} \Big) \le n^2 e^{-\gamma_0 \log^2 n}.
$$

Hence, as claimed, for $n$ large enough,
\begin{eqnarray*}
\P (\phi (X_0) \ge 3 \frac{\log n}{\sqrt n} ) \le 25(C+1)^2 n^2 e^{-\gamma_0 \log^2 n} \le e^{-\frac12 \gamma_0 \log^2 n}.
\end{eqnarray*}


\section{Extensions}


Theorems~\ref{thm.main-result} and~\ref{thm.main-result-d=1} imply a `global' law of large numbers approximation for the network, that is the number $f_k(X_t)$ of links with load $k$ is well approximated by the differential equation~(\ref{eq.diff-eq}). Indeed, for instance, by Theorem~\ref{thm.main-result}, when
$d \ge 2$, summing over all the nodes gives the following.
Let $B_n$ be the event that, for each $k$ and each $t \in [0,t_0]$,
\begin{eqnarray*}
\lefteqn{|f_k (X_t) - \binom{n}{2} \xi_t (k)| \le \Big( \sup_j \Big| f_j(X_0)-\binom{n}{2}\xi_0(j) \Big|} \\
&&\mbox{} + 23(\lambda +1)(t_0+1) d^2 (C+1)^3 \Big(n^2 \phi(X_0) + 3 n^{3/2} \log n\Big) \Big)
e^{216(\lambda+1)d^2 (C+1)^3 t_0}.
\end{eqnarray*}
Then $\P (\overline{B_n} ) \le e^{-\frac12 \gamma\log^2 n}$.
In the case $d=1$, an analogous result can be deduced from Theorem~\ref{thm.main-result-d=1}.
It would appear that these results are unlikely to be close to best possible: we would expect
to be able to approximate $f_k(X_t)$ with error of order $O(n)$, up to a logarithmic term, but
have not been able to prove such a result using our methods.  There are several places where
our argument would need to be improved, including the concentration of measure arguments used
in the proofs of Lemma~\ref{lem.phi-over-time} and of Lemma~\ref{lem.gen-exp}.

Our techniques can be adapted to analyse all the other variants of the model
mentioned in the introduction.  More generally, one would expect to be able to handle models
involving a large system (of size $n$), where any pair of elements (e.g.\ links) interact
at a rate tending to~0 as $n \to \infty$.  (In the present model, any pair of links share an
arrival stream at a rate of order $O(1/n)$.) These extensions may require a modified definition of function $\phi$.

\newcommand\AAP{\emph{Adv. Appl. Probab.} }
\newcommand\JAP{\emph{J. Appl. Probab.} }
\newcommand\JAMS{\emph{J. \AMS} }
\newcommand\MAMS{\emph{Memoirs \AMS} }
\newcommand\PAMS{\emph{Proc. \AMS} }
\newcommand\TAMS{\emph{Trans. \AMS} }
\newcommand\AnnMS{\emph{Ann. Math. Statist.} }
\newcommand\AnnPr{\emph{Ann. Probab.} }
\newcommand\CPC{\emph{Combin. Probab. Comput.} }
\newcommand\JMAA{\emph{J. Math. Anal. Appl.} }
\newcommand\RSA{\emph{Random Struct. Alg.} }
\newcommand\ZW{\emph{Z. Wahrsch. Verw. Gebiete} }
\newcommand\DMTCS{\jour{Discr. Math. Theor. Comput. Sci.} }

\newcommand\AMS{Amer. Math. Soc.}
\newcommand\Springer{Springer}
\newcommand\Wiley{Wiley}

\newcommand\vol{\textbf}
\newcommand\jour{\emph}
\newcommand\book{\emph}
\newcommand\inbook{\emph}
\def\no#1#2,{\unskip#2, no. #1,} 
\newcommand\toappear{\unskip, to appear}

\newcommand\webcite[1]{
\texttt{\def~{{\tiny$\sim$}}#1}\hfill\hfill}
\newcommand\webcitesvante{\webcite{http://www.math.uu.se/~svante/papers/}}
\newcommand\arxiv[1]{\webcite{arXiv:#1.}}

\def\nobibitem#1\par{}


\appendix

\section{Proof of Lemma~\ref{lem.appendix}} \label{SApp}

\begin{proofof} {Lemma~\ref{lem.appendix}}
We start with $\phi^1_{u,v,j,k}$, and note that
$$
\Delta \phi^1_{u,v,j,k} = \frac{1}{n-2} \Delta \left( \sum_{w\not=u,v} \I_{uw}^j \I_{vw}^k \right)
- \frac{1}{(n-2)^2} \Delta \left( \sum_{w\not=u,v} \I_{uw}^j \sum_{w'\not=u,v} \I_{vw'}^k \right).
$$
We therefore need to compute and compare the conditional expectations of these two increments.

On the event $\widetilde{A}_t = \{\widehat{X}_s \in \widetilde{S} \mbox{ for all } s \le t-1\}$, we have,
for any $u,v,j,k$,
\begin{eqnarray}
\lefteqn{\Big (\lambda {n \choose 2} + \big\lfloor 6 \lambda {n \choose 2} \big\rfloor \Big )
\E \Big [\Delta \Big( \sum_{w\not=u,v} \I_{uw}^j \I_{vw}^k \Big) (\widehat{X}_t)
\mid \widehat{{\mathcal F}}_{t-1} \Big ]} \nonumber\\
&=& \Big \{ \sum_{w\not=u,v} \I_{vw}^k \Big [ - j \I_{uw}^j + (j+1)\I_{uw}^{j+1} +
\lambda (\I_{uw}^{j-1} - \I_{uw}^j + g_{u,w,j-1}- g_{u,w,j} ) \Big ] \nonumber\\
&&\mbox{} + \sum_{w\not=u,v} \I_{uw}^j \Big [-k \I_{vw}^k +(k+1)\I_{vw}^{k+1} +
\lambda (\I_{vw}^{k-1} -\I_{vw}^{k} + g_{v,w,k-1}- g_{v,w,k}) \Big ] \nonumber \\
&&\mbox{} + \lambda \sum_{w\not=u,v} (P_{u,v,w,j-1,k-1} + P_{u,v,w,j,k})\Big \} (\widehat{X}_{t-1})
\nonumber \\
&&\mbox{} + \sum_{w \not = u,v} \widehat{X}_{t-1} (\{u,v\},w) (\I_{uw}^j \I_{vw}^k+ \I_{uw}^{j+1} \I_{vw}^{k+1})(\widehat{X}_{t-1}), \label{eq-comp1}
\end{eqnarray}
where, for instance, $\lambda \big( \I_{vw}^k g_{u,w,j}\big) (\widehat{X}_{t-1})$ is the
contribution for the case where a call is indirectly routed via the link $uw$ which has load $j$
in $\widehat{X}_{t-1}$, and $\lambda P_{u,v,w,j-1,k-1} (\widehat{X}_{t-1})$ is the contribution
for arrivals onto the route consisting of the links $\{ u,w\}$ and $\{ v,w\}$, with loads $j-1$
and $k-1$ respectively in $\widehat{X}_{t-1}$.
The term $\sum_{w \not=u,v} \widehat{X}_{t-1} (\{u,v\},w) \I_{uw}^{j+1} \I_{vw}^{k+1}(\widehat{X}_{t-1})$
represents departures of calls from the route consisting of links $\{u,w\}$ and $\{v,w\}$, with loads $j+1$ and $k+1$ respectively, while the term
$\sum_{w \not = u,v} \widehat{X}_{t-1} (\{u,v\},w) \I_{uw}^j \I_{vw}^k(\widehat{X}_{t-1})$ represents departures of calls from the route consisting of links $\{u,w\}$ and $\{v,w\}$, with loads $j$ and $k$ respectively, which, along with the arrivals to such routes, have otherwise been
overcounted.

Explicitly, in the expression above, for each $u$, $w$, and $j$,
\begin{eqnarray*}
g_{u,w,j} = P_{u,w,j}^+ + P_{u,w,j}^- + Q_{u,w,j}^+ + Q_{u,w,j}^-,
\end{eqnarray*}
with
\begin{eqnarray*}
P_{u,w,j}^+ = \frac{1}{(n-2)^d}\sum_{r=1}^d\I_{uw}^j\sum_{v',{\bf w}_r} \I_{v'u}^{C}
\I_{v'w}^{\le j} \prod_{s=1}^{r-1} (1-\I_{v'u,w_s}^{\le j})
\prod_{s=r+1}^d (1-\I_{v'u,w_s}^{\le j-1});
\end{eqnarray*}
\begin{eqnarray*}
P_{u,w,j}^- = \frac{1}{(n-2)^d}\sum_{r=1}^d\I_{uw}^j\sum_{v',{\bf w}_r}\I_{v'u}^C
\sum_{i=j+1}^{C-1}\I_{v'w}^i \prod_{s=1}^{r-1} (1-\I_{v'u,w_s}^{\le i})
\prod_{s=r+1}^d (1-\I_{v'u,w_s}^{ \le i-1}),
\end{eqnarray*}
%
\begin{eqnarray*}
Q_{u,w,j}^+ = \frac{1}{(n-2)^d}\sum_{r=1}^d\I_{uw}^j
\sum_{v',{\bf w}_r}\I_{v'w}^C \I_{v'u}^{\le j}
\prod_{s=1}^{r-1} (1-\I_{v'w,w_s}^{\le j}) \prod_{s=r+1}^d (1-\I_{v'w,w_s}^{\le j-1});
\end{eqnarray*}
\begin{eqnarray*}
Q_{u,w,j}^- = \frac{1}{(n-2)^d}\sum_{r=1}^d\I_{uw}^j \sum_{v',{\bf w}_r} \I_{v'w}^C
\sum_{i=j+1}^{C-1}\I_{v'u}^i \prod_{s=1}^{r-1} (1-\I_{v'w,w_s}^{\le i})
\prod_{s=r+1}^d (1-\I_{v'w,w_s}^{\le i-1}),
\end{eqnarray*}
where, according to our convention, $\sum_{v'}$ denotes the sum over all $v' \not= u,w$;
in the first two expressions, $\sum_{{\bf w}_r}$ denotes the sum over all
$w_1, \ldots,w_{r-1}$, $w_{r+1}, \ldots, w_d$ such that $w_s \not = v',u$ for any $s$,
whereas, in the final two expressions, $\sum_{{\bf w}_r}$ denotes the sum over all
$w_1, \ldots, w_{r-1}, w_{r+1}, \ldots, w_d$ such that $w_s \not = v',w$ for any $s$.
Also, explicitly,
\begin{eqnarray*}
P_{u,v,w,j,k} = \frac{1}{(n-2)^d}\sum_{r=1}^d\I_{uv}^C\I_{uw}^j\I_{vw}^k
\sum_{{\bf w}_r} \prod_{s=1}^{r-1} (1-\I_{uv,w_s}^{\le j \lor k})
\prod_{s=r+1}^d (1-\I_{uv,w_s}^{\le (j \lor k)-1}),
\end{eqnarray*}
where here $\sum_{{\bf w}_r}$ denotes the sum over all
$w_1, \ldots,w_{r-1}$, $w_{r+1}, \ldots, w_d$ such that $w_s \not = u,v$ for any $s$.



Similarly, on the event $\widetilde{A}_t = \{\widehat{X}_s \in \widetilde{S} \mbox{ for all } s \le t-1\}$,
\begin{eqnarray}
\lefteqn{\left( \lambda {n \choose 2} + \big\lfloor 6 \lambda {n \choose 2} \big\rfloor \right)
\E \Big[\Delta \Big(\sum_{w \not = u,v} \I_{uw}^j \sum_{w' \not = u,v} \I_{vw'}^k\Big)(\widehat{X}_t)
\mid \widehat{{\mathcal F}}_{t-1}\Big]} \nonumber\\
&=& \Big \{\Big(\sum_{w'\not=u,v} \I_{vw'}^k \Big)
\sum_{w \not=u,v} \Big[ -j \I_{uw}^j + (j+1) \I_{uw}^{j+1}
+ \lambda (\I_{uw}^{j-1} - \I_{uw}^j  \nonumber \\
&&\mbox{} +  g_{u,w,j-1}-  g_{u,w,j} )\Big ]
 + \Big( \sum_{w \not=u,v} \I_{uw}^j \Big)
\sum_{w'\not=u,v} \Big[ -k \I_{vw'}^k + (k+1) \I_{vw'}^{k+1} \nonumber \\
&&\mbox{} + \lambda (\I_{vw'}^{k-1} - \I_{vw'}^k + g_{v,w',k-1}-  g_{v,w',k} ) \Big] \nonumber \\
&&\mbox{} + \lambda
\sum_{w\not=u,v}(P_{u,v,w,j-1,k-1} + P_{u,v,w,j,k})\Big \}(\widehat{X}_{t-1}) \nonumber\\
&&\mbox{} +
\sum_{w \not = u,v} \widehat{X}_{t-1} (\{u,v\},w)
(\I_{uw}^j \I_{vw}^k + \I_{uw}^{j+1} \I_{vw}^{k+1})(\widehat{X}_{t-1}).
\label{eq-comp2}
\end{eqnarray}

Comparing corresponding terms in the two expressions (\ref{eq-comp1}) and (\ref{eq-comp2})
gives that, on the event $\widetilde{A}_t$,
\begin{eqnarray}
\lefteqn{\Big(\lambda {n \choose 2} + \big\lfloor 6 \lambda {n \choose 2} \big\rfloor \Big)
\E \big[ \big| \Delta \phi^1_{u,v,j,k}(\widehat{X}_t) \big| \mid \widehat{\cF}_{t-1}\big]} \nonumber\\
&\le& (2j + 2k + 2+ 4 \lambda) \phi^1 (\widehat{X}_{t-1}) \nonumber\\
&&\mbox{} +
\lambda \frac{n-3}{(n-2)^2} \sum_{w\not=u,v}(P_{u,v,w,j-1,k-1} + P_{u,v,w,j,k})(\widehat{X}_{t-1})
\nonumber \\
&&\mbox{} +
\frac{n-3}{(n-2)^2} \sum_{w \not = u,v} \widehat{X}_{t-1} (\{u,v\},w)
(\I_{uw}^{j} \I_{vw}^{k} + \I_{uw}^{j+1} \I_{vw}^{k+1})(\widehat{X}_{t-1}) \nonumber\\
&&\mbox{} + |a_{u,v,j,k-1} - a_{u,v,j,k} + a_{v,u,k,j-1} - a_{v,u,k,j}|(\widehat{X}_{t-1}),
\label{eq-bound}
\end{eqnarray}
where
\begin{eqnarray*}
a_{u,v,j,k}&=& \frac{\lambda}{n-2} \sum_{w} g_{u,w,j} \Big (\I_{vw}^k - \Big (\frac{1}{n-2}
\sum_{w' \not = u,v} \I_{vw'}^k \Big ) \Big )\\
&=& \frac{\lambda}{n-2} \sum_{w} (P_{u,w,j}^+ + P_{u,w,j}^-+Q_{u,w,j}^+ + Q_{u,w,j}^-) \\
&&\mbox{} \times
\Big (\I_{vw}^k - \Big (\frac{1}{n-2} \sum_{w' \not = u,v} \I_{v,w'}^k \Big ) \Big ).
\end{eqnarray*}

Bounding the middle two terms in the expression (\ref{eq-bound}) above is straightforward:
note that each $P_{u,v,w,j,k}$ is at most $d/(n-2)$, while 
$$
\sum_{w \not = u,v} \widehat{X}_{t-1} (\{u,v\},w) \le (n-2) \phi^3(\widehat{X}_{t-1}),
$$
and so, on $\widetilde{A}_t$,
\begin{eqnarray*}
\lefteqn{\Big(\lambda {n \choose 2} + \big\lfloor 6 \lambda {n \choose 2} \big\rfloor \Big)
\E \big[ \big|\Delta \phi^1_{u,v,j,k}(\widehat{X}_t)\big| \mid \widehat{\cF}_{t-1}\big]} \\
&\le& (4C + 2 + 4 \lambda) \phi^1 (\widehat{X}_{t-1})
+ \frac{2d\lambda}{n-2} + 2\phi^3 (\widehat{X}_{t-1})\\
&&\mbox{} + |a_{u,v,j,k-1} - a_{u,v,j,k} + a_{v,u,k,j-1} - a_{v,u,k,j}|(\widehat{X}_{t-1}) \\
&\le&(4C + 4 + 4 \lambda) \phi (\widehat{X}_{t-1}) + \frac{2d\lambda}{n-2} \\
&&\mbox{} + |a_{u,v,j,k-1} - a_{u,v,j,k} + a_{v,u,k,j-1} - a_{v,u,k,j}|(\widehat{X}_{t-1}).
\end{eqnarray*}

We will now show that both
$$
\frac{1}{n-2} \sum_{w\not=u,v}P_{u,w,j}^+ \Big (\frac{1}{n-2} \sum_{w' \not = u,v} \I_{vw'}^k \Big )
\quad \mbox{ and } \quad
\frac{1}{n-2} \sum_{w\not=u,v}P_{u,w,j}^+ \I_{vw}^k
$$
are close to their `standardised' version
\begin{eqnarray*}
\widehat{P}^+_{u,v,j,k} &=& \frac{1}{(n-2)^{2d+2}}\sum_{r=1}^d \sum_{w}\I_{uw}^j \sum_{w'} \I_{vw'}^k \sum_{v'} \I_{v'u}^C \sum_{v''}\I_{v''w}^{\le j}\\
&&\mbox{} \times \prod_{s=1}^{r-1} \sum_{w_s,w'_s}(1-\I_{uw_s}^{\le j} \I_{vw'_s}^{\le j})
\prod_{s=r+1}^d \sum_{w_s,w'_s}(1-\I_{uw_s}^{\le j-1} \I_{vw'_s}^{\le j-1}).
\end{eqnarray*}
Analogous bounds hold if $P_{u,w,j}^+$ is replaced by $P_{u,w,j}^-$,
$Q_{u,w,j}^+$, or $Q_{u,w,j}^-$.

First, an elementary calculation similar to earlier ones shows that
\begin{eqnarray*}
\lefteqn{\Big |\frac{1}{n-2} \sum_{w \not = u,v}P_{u,w,j}^+  \I_{vw}^k -\widehat{P}^+_{u,v,j,k} \Big|}\\
&\le& d(d-1) (C+1)^2\phi^1
+ \Big | \frac{1}{(n-2)^{2d}}\sum_{r=1}^d\sum_{w}\I_{uw}^j \I_{vw}^k\sum_{v'} \I_{v'u}^C
\I_{v'w}^{\le j}\\
&&\mbox{} \times \prod_{s=1}^{r-1} \sum_{w_s,w'_s}(1-\I_{uw_s}^{\le j} \I_{vw'_s}^{\le j})
\prod_{s=r+1}^d \sum_{w_s,w'_s}(1-\I_{uw_s}^{\le j-1} \I_{vw'_s}^{\le j-1})- \widehat{P}^+_{u,v,j,k} \Big |\\
&\le& d(d-1) (C+1)^2\Big (\phi^1 + \phi^2 + \frac{2}{n-2} \Big )\\
&&\mbox{} + \sum_{r=1}^d \Big ( 1- \Big (\frac{f_{u,\le j}- \I_{uv}^{\le j}}{n-2}\Big )^2
\Big )^{r-1} \Big ( 1- \Big (\frac{f_{u,\le j-1}-\I_{uv}^{\le j-1}}{n-2}\Big )^2 \Big )^{d-r} \\
&&\mbox{} \times\Big |\frac{1}{(n-2)^2}\sum_{w}\I_{uw}^j \I_{vw}^k\sum_{v' \not = u,w} \I_{v'u}^C\I_{v'w}^{\le j}\\
&&\mbox{} - \frac{1}{(n-2)^4}\sum_{w \not = u,v}\I_{uw}^j \sum_{w' \not = u,v} \I_{vw'}^k
\sum_{v' \not = u,w} \I_{v'u}^C \sum_{v'' \not = u,w}\I_{v''w}^{\le j}\Big |,
\end{eqnarray*}
and similarly that
\begin{eqnarray*}
\lefteqn{\Big |\frac{1}{(n-2)^2}\sum_{w \not = u,v}\I_{uw}^j \I_{vw}^k\sum_{v' \not = u,w}
\I_{v'u}^C\I_{v'w}^{\le j}} \\
&&\mbox{} - \frac{1}{(n-2)^4}\sum_{w \not = u,v}\I_{uw}^j \sum_{w' \not = u,v} \I_{vw'}^k
\sum_{v' \not = u,w} \I_{v'u}^C \sum_{v'' \not = u,w}\I_{v''w}^{\le j} \Big |\\
&\le& (C+1)\phi^1 +
\Big |\frac{1}{(n-2)^3}\sum_{w \not = u,v}\I_{uw}^j \I_{vw}^k\sum_{v' \not = u,w} \I_{v'u}^C
\sum_{v'' \not = u,w}\I_{v''w}^{\le j} \\
&&\mbox{} - \frac{1}{(n-2)^4}\sum_{w \not = u,v}\I_{uw}^j \sum_{w' \not = u,v} \I_{vw'}^k
\sum_{v' \not = u,w} \I_{v'u}^C \sum_{v'' \not = u,w}\I_{v''w}^{\le j} \Big |\\
&\le& (C+1)\Big (\phi^1 + \phi^2 + \frac{2}{n-2} + \frac{f_{u,C}- \I_{uv}^C}{n-2}
\frac{f_{u,\le j}-\I_{uv}^{\le j}}{n-2} \phi^1 \Big ) \\
&\le& (C+1) \Big (2 \phi^1 + \phi^2 + \frac{2}{n-2} \Big ).
\end{eqnarray*}
It follows that
\begin{eqnarray*}
\Big |\frac{1}{n-2} \sum_{w}P_{u,w,j}^+  \I_{vw}^k - \widehat{P}^+_{u,v,j,k} \Big |
& \le & d(d-1) (C+1)^2\Big (\phi^1 + \phi^2 + \frac{2}{n-2} \Big ) \\
&&\mbox{} + d (C+1) \Big (2 \phi^1 + \phi^2 + \frac{2}{n-2} \Big )\\
& \le &  d^2 (C+1)^2\Big (2\phi^1 + \phi^2 + \frac{2}{n-2} \Big ).
\end{eqnarray*}
If $d=1$, we may replace the above bound by $(C+1)(2\phi^1 + \phi^2 + \frac{2}{n-2})$.

Similarly, but slightly more easily,
\begin{eqnarray*}
\Big |\frac{1}{(n-2)^2} \sum_{w}P_{u,w,j}^+ \sum_{w' \not = u,v} \I_{vw'}^k - \widehat{P}^+_{u,v,j,k} \Big | \le d^2 (C+1)^2\Big (2\phi^1 + \phi^2 + \frac{2}{n-2} \Big ).
\end{eqnarray*}
Hence
\begin{eqnarray*}
\Big |\frac{1}{n-2} \sum_{w} P^+_{u,w,j} \Big ( \I_{vw}^k - \Big (\frac{1}{n-2}\sum_{w'} \I_{vw'}^k \Big )\Big )
\le 2d^2 (C+1)^2\Big (2\phi^1 + \phi^2 + \frac{2}{n-2} \Big ).
\end{eqnarray*}


Similar calculations for $P^-_{u,w,j}$, $Q^+_{u,w,j}$, $Q^-_{u,w,j}$ show that, for each $u,v,j,k$,
\begin{eqnarray*}
|a_{u,v,j,k}| \le 4\lambda d^2 (C+1)^3\Big (2\phi^1 + \phi^2 + \frac{2}{n-2} \Big ).
\end{eqnarray*}
If $d=1$, we have the improved bound
$|a_{u,v,j,k}| \le 4 \lambda(C+1) ( 2\phi^1 + \phi^2 + \frac{2}{n-2})$.

Hence, on $\widetilde{A}_t$, we have, for all $u,v,k,j$,
\begin{eqnarray*}
\lefteqn{\Big(\lambda {n \choose 2} + \big\lfloor 6 \lambda {n \choose 2} \big\rfloor \Big)
\E \big[ \big|\Delta \phi^1_{u,v,j,k}(\widehat{X}_t)\big| \mid \widehat{\cF}_{t-1}\big]} \\
&\le&(4C + 4 + 4 \lambda) \phi (\widehat{X}_{t-1}) + \frac{14d\lambda}{n-2}
+ 16\lambda d^2 (C+1)^3\Big (3\phi(\widehat{X}_{t-1}) + \frac{2}{n-2} \Big ) \\
&\le& 52(\lambda +1) d^2 (C+1)^3  \phi(\widehat{X}_{t-1}) + \frac{46\lambda d^2 (C+1)^3}{n-2}.
\end{eqnarray*}
For $d=1$, we have the improved bound
\begin{eqnarray*}
\lefteqn{\Big(\lambda {n \choose 2} + \big\lfloor 6 \lambda {n \choose 2} \big\rfloor \Big)
\E \big[ \big|\Delta \phi^1_{u,v,j,k}(\widehat{X}_t)\big| \mid \widehat{\cF}_{t-1}\big]} \\
&\le& 52 (\lambda+1) (C+1) \phi(\widehat{X}_{t-1}) + \frac{46\lambda (C+1)}{n-2}.
\end{eqnarray*}



Next, we consider $\Delta \phi^2_{u,v,j} = \frac{1}{n-2} \Delta (f_{u,j} - f_{v,j})$,
for $u,v \in V_n$ and $0 < j < C$.  We have
\begin{eqnarray*}
\lefteqn{\Big (\lambda {n \choose 2} + \big\lfloor 6 \lambda {n \choose 2} \big\rfloor \Big )
\E[\big|\Delta \big(f_{u,j}(\widehat{X}_t) - f_{v,j}(\widehat{X}_t)\big)\big| \mid \widehat{\cF}_{t-1}]}\\
&\le& \Big\{ \lambda |f_{u,j-1}-f_{v,j-1}| + (\lambda + j) |f_{u,j}-f_{v,j}| +
(j+1)|f_{u,j+1}-f_{v,j+1}|\\
&&\mbox{} + \lambda |g_{u,j-1} - g_{v,j-1}| + \lambda|g_{u,j} - g_{v,j}|\Big\} (\widehat{X}_{t-1}).
\end{eqnarray*}

We now refer to the functions $\widehat{P}^+_{u,j}$, $\widehat{P}^-_{u,j}$, $\widehat{Q}^+_{u,j}$ and
$\widehat{Q}^-_{u,j}$ defined in the proof of Lemma~\ref{lem.gen-exp}, as well as their sum
$\widehat{g}_{u,j}$.
Using (\ref{gghat}), we have that
\begin{eqnarray*}
\lefteqn{\Big (\lambda {n \choose 2} + \big\lfloor 6 \lambda {n \choose 2} \big\rfloor \Big )
\E[\big|\Delta \big(f_{u,j}(\widehat{X}_t) - f_{v,j}(\widehat{X}_t)\big)\big| \mid \widehat{\cF}_{t-1}]}\\
&\le& \Big\{ (2\lambda + 2C +1)(n-2)\phi^2 + 24d^2(C+1)^3(n-2) \phi \\
&&\mbox{} + \lambda |\widehat{g}_{u,j-1} - \widehat{g}_{v,j-1}| +
\lambda|\widehat{g}_{u,j} - \widehat{g}_{v,j}|\Big\} (\widehat{X}_{t-1}).
\end{eqnarray*}
If $d=1$, the term $24d^2(C+1)^3(n-2) \phi$ becomes $16(C+1)(n-2)\phi$, by
(\ref{gghat1}).

An easy calculation shows that, for $n \ge 4$, and each $u$, $v$ and $j$,
\begin{eqnarray*}
\lefteqn{|\widehat{P}^+_{u,j} - \widehat{P}^+_{v,j}|} \\
&=& \Big | \frac{1}{(n-2)^{2d}} \sum_{r=1}^d\sum_{u'}\I_{u'u}^C
\sum_{w_r,w'_r}\I_{uw_r}^j\I_{u'w'_r}^{\le j} \prod_{s=1}^{r-1}
\sum_{w_s,w'_s} (1 - \I_{u'w_s}^{\le j} \I_{uw'_s}^{\le j} )\\
&&\mbox{} \times \prod_{s=r+1}^{d} \sum_{w_s,w'_s} (1 - \I_{u'w_s}^{\le j-1} \I_{uw'_s}^{\le j-1} )
 - \frac{1}{(n-2)^{2d}}\sum_{r=1}^d\sum_{u'}\I_{u'v}^C
\sum_{w_r,w'_r}\I_{vw_r}^j\I_{u'w'_r}^{\le j} \\
&&\mbox{} \times \prod_{s=1}^{r-1} \sum_{w_s,w'_s} (1 - \I_{u'w_s}^{\le j} \I_{vw'_s}^{\le j} )
\prod_{s=r+1}^{d} \sum_{w_s,w'_s} (1 - \I_{u'w_s}^{\le j-1} \I_{vw'_s}^{\le j-1} )\Big | \\
&\le& d^2 C \Big ((n-2) \phi^2 + 2 \Big )
+ \Big | \frac{1}{(n-2)^{2d}}\sum_{r=1}^d\sum_{u'}\I_{u'u}^C
\sum_{w_r,w'_r}\I_{vw_r}^j\I_{u'w'_r}^{\le j} \\
&&\mbox{} \times \prod_{s=1}^{r-1} \sum_{w_s,w'_s} (1 - \I_{u'w_s}^{\le j} \I_{vw'_s}^{\le j} )\prod_{s=r+1}^{d}
\sum_{w_s,w'_s} (1 - \I_{u'w_s}^{\le j-1} \I_{vw'_s}^{\le j-1} ) \\
&&\mbox{} - \frac{1}{(n-2)^{2d}}\sum_{r=1}^d\sum_{u'}\I_{u'v}^C
\sum_{w_r,w'_r}\I_{vw_r}^j\I_{u'w'_r}^{\le j} \\
&&\mbox{} \times \prod_{s=1}^{r-1}
\sum_{w_s,w'_s} (1 - \I_{u'w_s}^{\le j} \I_{vw'_s}^{\le j} )\prod_{s=r+1}^{d}
\sum_{w_s,w'_s} (1 - \I_{u'w_s}^{\le j-1} \I_{vw'_s}^{\le j-1} )\Big | \\
&\le&  d^2 (C+1) \Big ((n-2) \phi^2 + 2 \Big ) +
d(2d+1) (C+1) \Big ((n-2) \phi^2 + 2 \Big )\\
&\le& 4d^2 (C+1) \Big ((n-2)\phi^2 + 2 \Big ).
\end{eqnarray*}

Similarly, for each $u$, $v$ and $j$,
\begin{eqnarray*}
|\widehat{P}^-_{u,j} - \widehat{P}^-_{v,j}| \le
4d^2 (C+1)^2 \Big ((n-2)\phi^2 + 2 \Big );
\end{eqnarray*}
\begin{eqnarray*}
|\widehat{Q}^+_{u,j} - \widehat{Q}^+_{v,j}| \le
4d^2 (C+1) \Big ((n-2)\phi^2 + 2 \Big );
\end{eqnarray*}
\begin{eqnarray*}
|\widehat{Q}^-_{u,j} - \widehat{Q}^-_{v,j}| \le
4d^2 (C+1)^2 \Big ((n-2)\phi^2 + 2 \Big ).
\end{eqnarray*}
It follows that for $n \ge 4$, and each $u$, $v$ and $j$, on $\widetilde{A}_t$,
\begin{eqnarray*}
\lefteqn{\Big (\lambda {n \choose 2} + \big\lfloor 6 \lambda {n \choose 2} \big\rfloor \Big )
\E [\big|\Delta \phi^2_{u,v,j} (\widehat{X}_t)\big| \mid \widehat{\mathcal F}_{t-1}]}\\
&\le& \big(2 \lambda + 2C +1 + 24d^2 (C+1)^3
+ 32\lambda d^2 (C+1)^2 \big) \phi (\widehat{X}_{t-1}) \\
&&\mbox{} + \frac{64}{n-2} \lambda d^2 (C+1)^2 \\
&\le& 35 (\lambda +1) d^2 (C+1)^3 \phi (\widehat{X}_{t-1}) + \frac{64 \lambda d^2 (C+1)^2}{n-2}.
\end{eqnarray*}
In the case $d=1$, we may obtain
\begin{eqnarray*}
\lefteqn{\Big (\lambda {n \choose 2} + \big\lfloor 6 \lambda {n \choose 2} \big\rfloor \Big )
\E [\big|\Delta \phi^2_{u,v,j} (\widehat{X}_t)\big| \mid \widehat{\mathcal F}_{t-1}]}\\
&\le& 35 (\lambda +1) (C+1) \phi (\widehat{X}_{t-1}) + \frac{64\lambda (C+1)}{n-2}.
\end{eqnarray*}


Finally, we consider the expectation of the absolute value of $\Delta \phi^3_{u,v} (\widehat{X}_t)$,
conditional on $\widehat{\cF}_{t-1}$.
We have
\begin{eqnarray*}
\lefteqn{\Big(\lambda {n \choose 2} + \big\lfloor 6 \lambda {n \choose 2} \big\rfloor \Big)
\E \Big[ \Delta \phi^3_{u,v} (\widehat{X}_t) \mid \widehat{\cF}_{t-1} \Big]} \\
& = &\Big( \frac{\lambda {n \choose 2} + \big\lfloor 6 \lambda {n \choose 2} \big\rfloor}{n-2} \Big)
\E \Big[ \Delta \Big( \sum_{w\not = u,v} \widehat{X}_t (\{u,v\},w)\Big) \mid \widehat{\cF}_{t-1} \Big] \\
&=& - \frac{1}{n-2} \sum_{w\not = u,v} \widehat{X}_{t-1} (\{u,v\},w) +
\frac{\lambda \I_{uv}^C}{n-2} \Big( 1 - \frac{1}{(n-2)^d} \sum_{{\bf w}} \prod_{s=1}^d
\big( 1 - \I_{uv,w_s}^{\le C-1} \big) \Big),
\end{eqnarray*}
on event $\widetilde{A}_t$.
So we have
\begin{eqnarray*}
\lefteqn{\Big(\lambda {n \choose 2} + \big\lfloor 6 \lambda {n \choose 2} \big\rfloor \Big)
\E \Big[ \Big| \Delta \phi^3_{u,v}(\widehat{X}_t) \Big| \mid \widehat{\cF}_{t-1} \Big]} \\
&\le & \frac{1}{n-2} \left( \sum_{w\not = u,v} \widehat{X}_{t-1} (\{u,v\},w) + \lambda \I_{uv}^C \right)
\le  \phi^3 (\widehat{X}_{t-1}) + \frac{\lambda}{n-2}.
\end{eqnarray*}

For $\rho$ any of the functions under consideration, we now have
\begin{eqnarray*}
\lefteqn{\Big (\lambda {n \choose 2} + \big\lfloor 6 \lambda {n \choose 2} \big\rfloor \Big )
\E [\big|\Delta \rho (\widehat{X}_t)\big| \mid \widehat{\mathcal F}_{t-1}]}\\
&\le& 52 (\lambda +1) d^2 (C+1)^3 \phi (\widehat{X}_{t-1}) + \frac{64\lambda d^2(C+1)^3}{n-2}.
\end{eqnarray*}
For $n \ge n_0$, $\lambda {n \choose 2} + \big\lfloor 6 \lambda {n \choose 2} \big\rfloor \ge 6\lambda \binom{n}{2} \ge 2n^2$,
and the lemma follows.
\end{proofof}

\end{document}